\documentclass{amsart}
\usepackage{amsxtra, microtype}
\usepackage{stmaryrd }
\usepackage{mathrsfs}
\usepackage[margin=1.45in]{geometry}
\usepackage{amssymb,amsmath,amsthm,stmaryrd,latexsym,wasysym}
\usepackage[all]{xy}
\usepackage{hyperref}
\usepackage{tikz}
\usetikzlibrary{decorations.pathmorphing}
\usetikzlibrary{arrows}
\usepackage{cancel}
\usepackage{appendix}
\theoremstyle{plain}
\usepackage{xcolor}
\usepackage{xspace}

\newtheorem{theorem}{Theorem}[section]
\newtheorem{lemma}[theorem]{Lemma}
\newtheorem{proposition}[theorem]{Proposition}
\newtheorem{corollary}[theorem]{Corollary}
\newtheorem{definition}[theorem]{Definition} \theoremstyle{definition}
\newtheorem{example}[theorem]{Example}
\newtheorem{remark}[theorem]{Remark}

\newcommand{\duer}{^*_}

\newcommand{\E}{{\mathbb{E}}}

\newcommand{\R}{\mathbb{R}} 
 
\newcommand{\inv}{^{-1}}
\newcommand{\N}{\mathbb{N}} 

\newcommand{\mx}{\mathfrak{X}} 
\newcommand{\dr}{\mathbf{d}}
\newcommand{\ldr}[1]{{{\pounds}}_{#1}}
\newcommand{\ip}[1]{{\mathbf{i}}_{#1}}
\newcommand{\an}[1]{\arrowvert_{#1}} 
\newcommand{\lb}{\llbracket} 
\newcommand{\rb}{\rrbracket}

\newcommand{\Beta}{\boldsymbol{\beta}}

\DeclareMathOperator{\dom}{Dom}

\DeclareMathOperator{\pr}{pr}
\DeclareMathOperator{\Hom}{Hom}

\DeclareMathOperator{\Id}{Id}
\DeclareMathOperator{\Jac}{Jac}

\newcommand{\nsp}[2]  {%
\langle\mspace{-6.8mu}%
\langle\mspace{-6.8mu}%
\langle\mspace{-6.8mu}%
\langle\mspace{-6.8mu}%
\langle\mspace{-6.8mu}%
\langle\mspace{-6.8mu}%
\langle{#1,\,}{#2}%
\rangle%
\mspace{-6.8mu}\rangle%
\mspace{-6.8mu}\rangle%
\mspace{-6.8mu}\rangle%
\mspace{-6.8mu}\rangle%
\mspace{-6.8mu}\rangle%
\mspace{-6.8mu}\rangle}

\allowdisplaybreaks

\begin{document}
\title[Lie 2-algebroids and matched pairs of $2$-representations]{Lie 2-algebroids and matched pairs of $2$-representations -- a geometric approach}


\author{M. Jotz Lean} \address{Mathematisches Institut, Georg-August Universit\"at G\"ottingen}  \email{madeleine.jotz-lean@mathematik.uni-goettingen.de}
\subjclass[2010]{Primary: 53B05, 
  Secondary:
  53D17. 
}

\begin{abstract}
  Li-Bland's correspondence between linear Courant algebroids and Lie
  $2$-algebroids is explained and shown to be an equivalence of
  categories.  Decomposed VB-Courant algebroids are shown to be
  equivalent to split Lie 2-algebroids in the same manner as
  decomposed VB-algebroids are equivalent to 2-term representations up
  to homotopy (Gracia-Saz and Mehta). Several classes of examples are
  discussed, yielding new examples of split Lie 2-algebroids.

  We prove that the bicrossproduct of a matched pair of
  $2$-representations is a split Lie $2$-algebroid and we explain this
  result geometrically, as a consequence of the equivalence of
  VB-Courant algebroids and Lie $2$-algebroids. This explains in
  particular how the two notions of ``double'' of a matched pair of
  representations are geometrically related. In the same manner, we
  explain the geometric link between the two notions of double of a
  Lie bialgebroid.
\end{abstract}
\maketitle

\tableofcontents

\section{Introduction}
A matched pair of Lie algebroids is a pair of Lie algebroids $A$ and
$B$ over a smooth manifold $M$, together with a representation of $A$
on $B$ and a representation of $B$ on $A$, satisfying some
compatibility conditions, which can be
interpreted in two manners: first the direct sum $A\oplus B$ carries a
Lie algebroid structure over $M$, such that $A$ and $B$ are Lie
subalgebroids and such that the representations give ``mixed''
brackets
\[ [(a,0),(0,b)]=(-\nabla_ba,\nabla_ab)
\]
for all $a\in\Gamma(A)$ and $b\in\Gamma(B)$. The direct sum $A\oplus
B$ with this Lie algebroid structure is called here the
\emph{bicrossproduct of the matched pair}. Note that conversely, any
Lie algebroid with two transverse and complementary subalgebroids
defines a matched pair of Lie algebroids \cite{Mokri97}.

Alternatively, the fiber product $A\times_M B$, which has a double
vector bundle structure with sides $A$ and $B$ and with trivial core,
is as follows a double Lie algebroid: for $a\in\Gamma(A)$, we write
$a^l\colon B\to A\times_MB$, $b_m\mapsto(a(m),b_m)$, for the linear
section of $A\times_M B\to B$, and similarly, a section
$b\in\Gamma(B)$ defines a linear section $b^l\in\Gamma_A(A\times_MB)$.
The Lie algebroid structure on $A\times_MB\to B$ is defined by
\[ \left[a_1^l,a_2^l\right]=[a_1,a_2]^l
\quad \text{ and } \quad \rho(a^l)=\widehat{\nabla_a}\in\mx^l(B)
\]
for $a, a_1, a_2\in\Gamma(A)$, where we denote by
$\widehat{D}\in\mx(B)$ the linear vector field defined by a derivation
$D$ on $B$. The Lie algebroid structure on $A\times_M B\to A$ is
defined accordingly by the Lie bracket on sections of $B$ and the
$B$-connection on $A$. The double Lie algebroid $A\times_M B$ is then
called the \emph{double of the matched pair}. Note that conversely,
any double Lie algebroid with trivial core is the fiber product of two
vector bundles and defines a matched pair of Lie algebroids
\cite{Mackenzie11}.

These two constructions encoding the compatibility conditions for a
matched pair of representations seem at first sight only related by
the fact that they both encode matched pairs. A similar phenomenon can
be observed with the notion of Lie bialgebroid: A Lie bialgebroid is a
pair of Lie algebroids $A, A^*\to M$ in duality, satisfying some compatibility conditions,
which can be described in two manners. First, the direct sum $A\oplus
A^*\to M$ inherits a Courant algebroid structure with the two Lie
algebroids $A$ and $A^*$ as transverse Dirac
structures, and mixed brackets given by
\[ \lb (a,0), (0,\alpha)\rb=(-\ip{\alpha}\dr_{A^*}a, \ldr{a}\alpha)
\]
for all $a\in\Gamma(A)$ and $\alpha\in\Gamma(A^*)$. Alternatively, the
cotangent bundle $T^*A$, a double vector bundle with sides $A$ and
$A^*$ and core $T^*M$, which is isomorphic as a double vector bundle
to $T^*A^*$ carries two linear Lie algebroid structures. The first, on
$T^*A\to A$ is the cotangent Lie algebroid induced by the
linear Poisson structure defined on $A$ by the Lie algebroid structure
on $A^*$. The second, on $T^*A\simeq T^*A^*\to A^*$ is defined in the
same manner by the Lie algebroid structure on $A$. The compatibility
conditions for $A$ and $A^*$ to build a Lie bialgebroid are equivalent
to the double Lie algebroid condition for $(T^*A, A, A^*, M)$ 
\cite{Mackenzie11,GrJoMaMe17}. Again, the cotangent double of the Lie
algebroid and the bicrossproduct Courant algebroid seem only related
by the fact that they are two elegant ways of encoding the Lie
bialgebroid conditions. 

One feature of this paper is the explanation of the deeper, more
intrinsic relation between the bicrossproduct of a matched pair of Lie
algebroids and its double on the one hand, and between the
bicrossproduct of a Lie bialgebroid and its cotangent double on the
other hand.  In both cases, the bicrossproduct can be understood as a
purely algebraic construction, which is \emph{geometrised} by the
corresponding double Lie algebroid.  More generally, we explain how
the matched pair of two $2$-term representations up to homotopy
\cite{GrJoMaMe17} defines a bicrossproduct split Lie $2$-algebroid,
and we relate the latter to the decomposed double Lie algebroid found
in \cite{GrJoMaMe17} to be equivalent to the matched pair of $2$-representations.

These three classes of examples of bicrossproduct constructions vs
double Lie algebroid constructions are in fact three special cases of
the equivalence between the category of VB-Courant algebroid, and the
category of Lie $2$-algebroids.

Let us be a little more precise. Supermanifolds were introduced in the
1970's by physicists, as a formalism to describe supersymmetric field
theories, and have been extensively studied since then (see
e.g.~\cite{Sardanashvily09,Varadarajan04} and references therein).  A
supermanifold is a smooth manifold the algebra of functions of which
is enriched by anti-commuting coordinates. Supermanifolds with an
additional $\mathbb{Z}$-grading have been used since the late 1990's
among others
in relation with Poisson geometry and Lie and Courant algebroids
\cite{Severa05,Roytenberg02,Voronov02}.

An equivalence between Courant algebroids and $\N$-manifolds of degree
$2$ endowed with a symplectic structure and a compatible homological
vector field \cite{Roytenberg02} is at the heart of the current
interest in $\N$-graded manifolds in Poisson geometry, as this algebraic description of
Courant algebroids leads to possible paths to their integration
\cite{Severa05,LiSe12,MeTa11}. In \cite{Jotz17a} we showed how the
category of $\N$-manifolds of degree $2$ is equivalent to a category
of double vector bundles endowed with a linear involution. 
The latter involutive double vector bundles are dual to double vector
bundles endowed with a linear metric.
In this paper
we extend this correspondence to an equivalence between the category
of $\N$-manifolds of degree $2$ endowed with a homological vector
field and a category of VB-Courant algebroids, i.e.~metric double
vector bundles endowed with a linear Courant algebroid structure (see
also \cite{Li-Bland12}).

\subsubsection*{Original motivation}
Let us explain in more detail our methodology and our original
motivation.  A VB-Lie algebroid is a double vector bundle $(D;A,B;M)$
with one side $D\to B$ endowed with a Lie algebroid bracket and an
anchor that are \emph{linear} over a Lie algebroid structure on $A\to M$.  Gracia-Saz and Mehta prove in
\cite{GrMe10a} that linear decompositions of VB-algebroids are
equivalent to super-representations, or in other words, to
$2$-representations.

The definition of a VB-Courant algebroid is very similar to the one of a
VB-algebroid.  The Courant bracket, the anchor and the non-degenerate
pairing all have to be linear. In \cite{Jotz13a} we prove that the
standard Courant algebroid over a vector bundle can be decomposed into
a connection, a Dorfman connection, a curvature term and a vector
bundle map, in a manner that resembles very much the one in
\cite{GrMe10a}. In other words, as linear splittings of the tangent
space $TE$ of a vector bundle $E$ are equivalent to linear connections
on the vector bundle, linear splittings of the Pontryagin bundle
$TE\oplus T^*E$ over $E$ are equivalent to a certain class of Dorfman
connections \cite{Jotz13a}.  Further, as the Lie algebroid structure
on $TE\to E$ can be described in a splitting in terms of the
corresponding connection, the Courant algebroid structure on
$TE\oplus T^*E\to E$ is completely encoded in a splitting by the
corresponding Dorfman connection \cite{Jotz13a}.

Our original goal in this project was to show that this is in fact a
very special case of a general result on linear splittings of
VB-Courant algebroids, in the spirit of Gracia-Saz and Mehta's work
\cite{GrMe10a}. Along the way, we proved the equivalence of
$[2]$-manifolds with metric double vector bundles \cite{Jotz17a}.
This paper builds up on those results to explain in a very precise
manner the `bijection' found by Li-Bland \cite{Li-Bland12} between Lie
$2$-algebroids and VB-Courant algebroids; it is now more correctly
formulated as an equivalence of categories.

Note that while the methods used in \cite{Jotz17a,Li-Bland12} do not
use splittings of the $[2]$-manifolds and metric double vector
bundles, it appears more natural to us to work here with split
objects. First, the equivalence of the underlying $[2]$-manifolds with
metric double vector bundles was already established and it is now
much more convenient to work in splittings vs Lagrangian double
vector bundle charts -- the 'coordinate free' definition of the
homological vector field that corresponds to a linear Courant
algebroid structure is easily done (see Section
\ref{dorfman_eq_split}), but we did not find a good coordinate free
definition of it.  Second, working with splittings is necessary in
order to exhibit the similarity with Gracia-Saz and Mehta's techniques
in \cite{GrMe10a}; which is one of our main goals. Finally, as
explained below, the
construction of the bicrossproduct of a matched pair of
$2$-representations is an algebraic description of the construction of
a \emph{decomposed} VB-Courant algebroid from a \emph{decomposed}
double Lie algebroid, just as $2$-representations are equivalent to
\emph{decomposed} VB-Lie algebroids.

\subsubsection*{Application}
The equivalence of matched pairs of $2$-representations with a certain
class of split Lie $2$-algebroids appears as a natural class of
examples of our equivalence of VB-Courant algebroids with Lie
$2$-algebroids. A double vector bundle $(D;A,B;M)$ with core $C$ and
two linear Lie algebroid structures on $D\to A$ and $D\to B$ is a
double Lie algebroid if and only if the pair of duals
$(D\duer A; D\duer B)$ is a VB-Lie bialgebroid over
$C^*$. Equivalently, $D\duer A\oplus_{C^*} D\duer B$ is a VB-Courant
algebroid over $C^*$, with side $A\oplus B$ and core $B^*\oplus A^*$,
and with two transverse Dirac structures $D\duer A$ and $D\duer B$.  A
decomposition of $D$ defines on the one hand a matched pair of
$2$-representations \cite{GrJoMaMe17}, and on the other hand a
decomposition of $D\duer A\oplus_{C^*} D\duer B$, hence a split Lie
$2$-algebroid. Once this geometric correspondence has been found, it
is straightforward to construct algebraically the split Lie
$2$-algebroid from the matched pair, and vice-versa.

\subsection*{Outline, main results and applications}
This paper is organised as follows.

\paragraph{In \textbf{Section \ref{preliminaries}}} we describe the
main result in \cite{Jotz17a} -- the equivalence of $[2]$-manifolds
with metric double vector bundles -- and we recall the background on
VB-algebroids and representations up to homotopy that will be
necessary for our main application on the bicrossproduct of a matched
pair of $2$-representations.

\paragraph{In \textbf{Section \ref{sec:split_lie_2}}}
we start by recalling necessary background on Courant algebroids,
Dirac structures and Dorfman connections.  Then we formulate in our
own manner Sheng and Zhu's definition of split Lie $2$-algebroids
\cite{ShZh17}. We write in coordinates the homological
vector field corresponding to a split Lie $2$-algebroid, showing where
the components of the split Lie $2$-algebroid appear. In Section
\ref{examples_split_lie_2}, we give several classes of examples of
split Lie $2$-algebroids, introducing in particular the standard split
Lie $2$-algebroids defined by a vector bundle.  Finally we describe morphisms of split Lie $2$-algebroids.

\paragraph{In \textbf{Section \ref{sec:VB_cour}}}
we give the definition of VB-Courant algebroids \cite{Li-Bland12} and
we relate split Lie $2$-algebroids with Lagrangian splittings of
VB-Courant algebroids, in the spirit of Gracia-Saz and Mehta's
description of split VB-algebroids via $2$-term representations up to
homotopy \cite{GrMe10a}. Then we describe the VB-Courant algebroids
corresponding to the examples of split Lie $2$-algebroids found in the
preceding section, and we prove that the equivalence of categories
established in \cite{Jotz17a} induces an equivalence of the category
of VB-Courant algebroids with the category of Lie $2$-algebroids.

\paragraph{In \textbf{Section \ref{double}}} we construct the
bicrossproduct of a matched pair of $2$-representations and prove that
it is a split Lie $2$-algebroid. We then explain geometrically this
result by studying VB-bialgebroids and double Lie algebroids.

\paragraph{\textbf{Appendices}}
We give in Section
\ref{appendix_proof_of_main} the proof of our main theorem, describing
decomposed VB-Courant algebroids via split Lie $2$-algebroids. In
Section \ref{appendix_dual} we quickly recall how double vector
bundles and their splittings are dualised.

\subsection*{Acknowledgement}
The author warmly thanks Chenchang Zhu for giving her a necessary
insight at the origin of her interest in Lie $2$-algebroids, and Alan
Weinstein, David Li-Bland, Rajan Mehta, Dmitry Roytenberg and Arkady
Vaintrob for interesting conversations or comments. Thanks go also to
Yunhe Sheng for his help on a technical detail 
and in particular to Rohan Jotz Lean for many useful comments.  This
work was partially supported by a \emph{Fellowship for prospective
  researchers (PBELP2\_137534) of the Swiss NSF} for a postdoctoral
stay at UC Berkeley.

\subsection*{Prerequisites, notation and conventions}
We write $p_M\colon TM\to M$, $q_E\colon E\to M$ for vector bundle
maps. For a vector bundle $Q\to M$ we often identify without further
mentioning the vector bundle $(Q^*)^*$ with $Q$ via the canonical
isomorphism. We write $\langle\cdot\,,\cdot\rangle$ for the canonical
pairing of a vector bundle with its dual; i.e.~$\langle
a_m,\alpha_m\rangle=\alpha_m(a_m)$ for $a_m\in A$ and $\alpha_m\in
A^*$. We use several different pairings; in general, which pairing is
used is clear from its arguments.  Given a section $\varepsilon$ of
$E^*$, we always write $\ell_\varepsilon\colon E\to \R$ for the linear
function associated to it, i.e.~the function defined by $e_m\mapsto
\langle \varepsilon(m), e_m\rangle$ for all $e_m\in E$.

Let $M$ be a smooth manifold. We denote by $\mx(M)$ and $\Omega^1(M)$
the sheaves of local smooth sections of the tangent and the cotangent
bundle, respectively. For an arbitrary vector bundle $E\to M$, the
sheaf of local sections of $E$ will be written $\Gamma(E)$.  Let
$f\colon M\to N$ be a smooth map between two smooth manifolds $M$ and
$N$.  Then two vector fields $X\in\mx(M)$ and $Y\in\mx(N)$ are said to
be \textbf{$f$-related} if $Tf\circ X=Y\circ f$ on $\dom(X)\cap
f\inv(\dom(Y))$.  We write then $X\sim_f Y$. In the same manner, if
$\phi\colon A\to B$ is a vector bundle morphism over $\phi_0\colon
M\to N$, then a section $a\in\Gamma_M(A)$ is $\phi$-related to
$b\in\Gamma_N(B)$ if $\phi(a(m))=b(\phi_0(m))$ for all $m\in M$. We
write then $a\sim_\phi b$. The dual of the morphism $\phi$ is in general not a morphism
of vector bundles, but a relation $R_{\phi^*}\subseteq A^*\times
B^*$ defined as follows:
\[
R_{\phi^*}=\{(\phi_m^*\beta_{\phi_0(m)},\beta_{\phi_0(m)})\mid
m\in M, \beta_{\phi_0(m)}\in B^*_{\phi_0(m)}\},
\]
where $\phi_m\colon A_m\to B_{\phi_0(m)}$ is the morphism of
vector spaces.
 
\medskip
We will say $2$-representations for $2$-term representations up to
homotopy.  We write ``$[n]$-manifold'' for ``$\N$-manifolds of degree
$n$''. 
Let $E_1$ and $E_{2}$ be smooth vector bundles of finite
ranks $r_1,r_2$ over $M$. 
The $[2]$-manifold
$E_{1}[-1]\oplus E_{2}[-2]$ has local basis sections of
${E_{i}}^*$ as local generators of degree $i$, for $i=1,2$, and
so dimension $(p;r_1,r_2)$. 
A $[2]$-manifold $\mathcal M=E_{1}[-1]\oplus E_{2}[-2]$ defined in
this manner by a graded vector bundle is called a \textbf{split
  $[2]$-manifold}. In other words, we have $C^\infty(\mathcal
M)^0=C^\infty(M)$, $C^\infty(\mathcal M)^1=\Gamma(E_{1}^*)$ and
$C^\infty(\mathcal M)^2=\Gamma(E_{2}^*\oplus \wedge^2E_{1}^*)$. Let
$\mathcal N:=F_1 [-1]\oplus F_{2}[-2]$  be a second $[2]$-manifold over a base $N$. A
morphism $\mu\colon F_{1}[-1]\oplus F_{2}[-2]\to E_{1}[-1]\oplus
E_{2}[-2]$ of split $[2]$-manifolds over the bases $N$ and $M$,
respectively, consists of a smooth map $\mu_0\colon N\to M$, three
vector bundle morphisms $\mu_1\colon F_{1}\to E_{1}$, $\mu_2\colon
F_{2}\to E_{2}$ and $\mu_{12}\colon \wedge^2F_{1}\to E_{2}$
over $\mu_0$. The morphism $\mu^\star\colon C^\infty(\mathcal M)\to
C^\infty(\mathcal N)$ sends a degree $1$ function
$\xi\in\Gamma(E_{1}^*)$ to
${\mu_1}^\star\xi\in\Gamma(F_{1}^*)$
and a degree $2$-function $\xi\in\Gamma(E_{2}^*)$ to
${\mu_2}^\star\xi+\mu_{12}^\star\xi\in\Gamma({F_{2}}^*\oplus
\wedge^2{F_{1}}^*)$.

\section{Preliminaries}\label{preliminaries}
In this section we recall the necessary background on VB-algebroids,
double Lie algebroids and VB-bialgebroids, as well as the
correspondence found in \cite{Jotz17a} between double vector bundles
endowed with a linear metric, and $\N$-manifolds of degree $2$.  The
definition of a double vector bundle can be found in Appendix
\ref{appendix_dual}, together with a summary of their properties and
an overview of our notation conventions.

\subsection{VB-algebroids, double Lie algebroids, VB-bialgebroids}
\label{subsect:VBa}
Let $(D; A, B; M)$ be a double vector bundle
with core $C$.
Then $(D \to B; A \to M)$ is a \textbf{VB-algebroid}
(\cite{Mackenzie98x}; see also \cite{GrMe10a}) if $D \to B$ has a Lie
algebroid structure the anchor of which is a bundle morphism
$\Theta_B\colon D \to TB$ over $\rho_A\colon A \to TM$ and such that
the Lie bracket is linear:
\begin{equation*} [\Gamma^\ell_B(D), \Gamma^\ell_B(D)] \subset
  \Gamma^\ell_B(D), \qquad [\Gamma^\ell_B(D), \Gamma^c_B(D)] \subset
  \Gamma^c_B(D), \qquad [\Gamma^c_B(D), \Gamma^c_B(D)]= 0.
\end{equation*}
The vector bundle $A\to M$ is then also a Lie algebroid, with anchor
$\rho_A$ and bracket defined as follows: if $\xi_1,
\xi_2\in\Gamma^\ell_B(D)$ are linear over $a_1,a_2\in\Gamma(A)$, then
the bracket $[\xi_1,\xi_2]$ is linear over $[a_1,a_2]$.  We also say
that the Lie algebroid structure on $D\to B$ is linear over the Lie
algebroid $A\to M$. Note that since the anchor $\Theta_B$ is linear,
it sends a core section $c^\dagger$, $c\in\Gamma(C)$ to a vertical
vector field on $B$.  This defines the \textbf{core-anchor}
$\partial_B\colon C\to B$; for $c\in\Gamma(C)$ we have
$\Theta_B(c^\dagger)=(\partial_Bc)^\uparrow$ (see \cite{Mackenzie92} and the following example).

Before we discuss the linear splittings of VB-algebroids, let us
discuss the following fundamental class of examples.
\begin{example}[The tangent double of a vector bundle]\label{tangent_double}
  Let $q_E\colon E\to M$ be a vector bundle.  Then the tangent bundle
  $TE$ has two vector bundle structures; one as the tangent bundle of
  the manifold $E$, and the second as a vector bundle over $TM$. The
  structure maps of $TE\to TM$ are the derivatives of the structure
  maps of $E\to M$. The space $TE$ is a double vector bundle with core
  bundle $E \to M$. The map $\bar{}\,\colon E\to p_E^{-1}(0^E)\cap
  (Tq_E)^{-1}(0^{TM})$ sends $e_m\in E_m$ to $\bar
  e_m=\left.\frac{d}{dt}\right\an{t=0}te_m\in T_{0^E_m}E$.
\begin{equation*}
\begin{xy}
\xymatrix{
TE \ar[d]_{Tq_E}\ar[r]^{p_E}& E\ar[d]^{q_E}\\
 TM\ar[r]_{p_M}& M}
\end{xy}
\end{equation*}
The core vector field corresponding to $e \in \Gamma(E)$ is the
vertical lift $e^{\uparrow}\colon E \to TE$, i.e.~the vector field
with flow $\phi\colon E\times \R\to E$, $\phi_t(e'_m)=e'_m+te(m)$. An
element of $\Gamma^\ell_E(TE)=\mx^\ell(E)$ is called a \textbf{linear
  vector field}.  A linear vector field $\xi\in\mx^l(E)$ covering
$X\in\mx(M)$ is equivalent to a derivation $D_\xi\colon \Gamma(E) \to
\Gamma(E)$ over $X\in \mx(M)$ (see
e.g.~\cite{Mackenzie05}).
The precise
correspondence is given by
\begin{equation}\label{ableitungen}
\xi(\ell_{\varepsilon}) 
= \ell_{D_\xi^*(\varepsilon)} \,\,\,\, \text{ and }  \,\,\, \xi(q_E^*f)= q_E^*(X(f))
\end{equation}
for all $\varepsilon\in\Gamma(E^*)$ and $f\in C^\infty(M)$.  We
write $\widehat D$ for the linear vector field in $\mx^l(E)$
corresponding in this manner to a derivation $D$ of $\Gamma(E)$. A
linear splitting $\Sigma$ for $(TE; TM, E; M)$ is equivalent to a
linear connection $\nabla\colon \mx(M)\times\Gamma(E)\to
\Gamma(E)$ defined by $\sigma_{TM}(X)=\widehat{\nabla_X}$ for all
$X\in\mx(M)$. 
It is easy
to see using \eqref{ableitungen} that
$\sigma:=\sigma_{TM}$ satisfies
\begin{equation}\label{Lie_bracket_VF}
  \left[\sigma(X), \sigma(Y)\right]=\sigma[X,Y]-\widetilde{R_\nabla(X,Y)},\quad
  \left[\sigma(X), e^\uparrow\right]=(\nabla_Xe)^\uparrow,\quad
  \left[e^\uparrow,e'^\uparrow\right]=0,
\end{equation}
for all $X,Y\in\mx(M)$ and $e,e'\in\Gamma(E)$.  That is, the Lie
bracket of vector fields on $M$ and the connection encode completely
the Lie bracket of vector fields on $E$.

\medskip

Now let us have a quick look at the other structure on the double
vector bundle $TE$. The lift
$\sigma_{E}^\nabla\colon\Gamma(E)\to\Gamma_{TM}^\ell(TE)$ is given by
\begin{equation*}
  \sigma_{E}^\nabla(e)(v_m) = T_me(v_m) +_{TM} (T_m0^E(v_m) -_E \overline{\nabla_{v_m} e}), \,\, v_m \in TM, \, e \in \Gamma(E).
\end{equation*}
Further, for $e\in\Gamma(E)$, the core section $e^\dagger\in\Gamma_{TM}(TE)$
is given by
$e^\dagger(v_m)=T_m0^E(v_m)+_E\left.\frac{d}{dt}\right\an{t=0}te(m)$.
\end{example}

Let $A\to M$ be a Lie algebroid and consider an $A$-connection
$\nabla$ on a vector bundle $E\to M$.  Then the space
$\Omega^\bullet(A,E)$ of $E$-valued forms has an induced
operator $\dr_\nabla$ given by:
\begin{equation*}
\begin{split}
  \dr_\nabla\omega(a_1,\ldots,a_{k+1})
=&\sum_{i<j}(-1)^{i+j}\omega([a_i,a_j],a_1,\ldots,\hat a_i,\ldots,\hat a_j,\ldots, a_{k+1})\\
  &\qquad +\sum_i(-1)^{i+1}\nabla_{a_i}(\omega(a_1,\ldots,\hat
  a_i,\ldots,a_{k+1}))
\end{split}
\end{equation*}
for all $\omega\in\Omega^k(A,E)$ and $a_1,\ldots,a_{k+1}\in\Gamma(A)$.
The connection is flat if and only if $\dr_\nabla^2=0$. Consider a
vector bundle $\mathcal E= E_0\oplus E_1$ and a Lie algebroid $A$ over
$M$.
A \emph{2-term representation up to homotopy of $A$ on $\mathcal E$}
\cite{ArCr12} (or a \emph{superrepresentation} \cite{GrMe10a}) 
is 
\begin{enumerate}
\item [(1)] a vector bundle map $\partial\colon E_0\to E_1$,
\item [(2)] two $A$-connections, $\nabla^0$ and $\nabla^1$ on $E_0$
  and $E_1$, respectively, such that $\partial \circ \nabla^0 =
  \nabla^1 \circ \partial$, \item [(3)] an element $R \in \Omega^2(A,
  \Hom(E_1, E_0))$ such that $R_{\nabla^0} = R\circ \partial$,
  $R_{\nabla^1}=\partial \circ R$ and $\dr_{\nabla^{\Hom}}R=0$,
  where $\nabla^{\Hom}$ is the connection induced on $\Hom(E_1,E_0)$
  by $\nabla^0$ and $\nabla^1$.
\end{enumerate}
For brevity we will call such a 2-term representation up to homotopy a
\textbf{2-re\-pre\-sen\-ta\-tion}.
\medskip

Let $(D\to B, A\to M)$ be a VB-Lie algebroid and choose a linear
splitting $\Sigma\colon A\times_MB\to D$. Since the anchor of a linear
section is a linear vector field, for each $a\in \Gamma(A)$ the vector field
$\Theta_B(\sigma_A(a))$ defines a derivation of $\Gamma(B)$ with
symbol $\rho(a)$. This defines a linear
connection $\nabla^{AB}\colon \Gamma(A)\times\Gamma(B)\to\Gamma(B)$ by
$\Theta_B(\sigma_A(a))=\widehat{\nabla_a^{AB}}$ for all
$a\in\Gamma(A)$.  Since the bracket of a linear section with a core
section is again a core section, we find a linear connection
$\nabla^{AC}\colon\Gamma(A)\times\Gamma(C)\to\Gamma(C)$ such that
$[\sigma_A(a),c^\dagger]=(\nabla_a^{AC}c)^\dagger$ for all
$c\in\Gamma(C)$ and $a\in\Gamma(A)$.  The difference
$\sigma_A[a_1,a_2]-[\sigma_A(a_1), \sigma_A(a_2)]$ is a core-linear
section for all $a_1,a_2\in\Gamma(A)$.  This defines a form
$R\in\Omega^2(A,\operatorname{Hom}(B,C))$ such that $[\sigma_A(a_1),
\sigma_A(a_2)]=\sigma_A[a_1,a_2]-\widetilde{R(a_1,a_2)}$, for all
$a_1,a_2\in\Gamma(A)$. This is done in \cite{GrMe10a}, which proves the following theorem.
\begin{theorem}\label{rajan}
  Let $(D \to B; A \to M)$ be a VB-algebroid and choose a linear
  splitting $\Sigma\colon A\times_MB\to D$.  The triple
  $(\nabla^{AB},\nabla^{AC},R)$ defined as above is a
  $2$-representation of $A$ on the complex $\partial_B\colon C\to B$.

  Conversely, let $(D;A,B;M)$ be a double vector bundle with core $C$
  such that $A$ has a Lie algebroid structure, and choose a linear
  splitting $\Sigma\colon A\times_MB\to D$. If
  $(\nabla^{AB},\nabla^{AC},R)$ is a 2-representation of $A$ on a
  morphism $\partial_B\colon C\to B$, then the three equations above
  and the core-anchor $\partial_B$ define a VB-algebroid structure on
  $(D\to B; A\to M)$.

\end{theorem}

\begin{example}\label{double_ruth}
  Choose a linear connection
  $\nabla\colon\mx(M)\times\Gamma(E)\to\Gamma(E)$ and consider the
  corresponding linear splitting $\Sigma^\nabla$ of $TE$ as in Example
  \ref{tangent_double}.  The description of the Lie bracket of vector
  fields in \eqref{Lie_bracket_VF} shows that the 2-representation
  induced by $\Sigma^\nabla$ is the 2-representation of $TM$ on
  $\Id_E\colon E\to E$ given by $(\nabla,\nabla,R_\nabla)$.
\end{example}

\subsubsection{Double Lie algebroids and matched pairs of $2$-representations}
\label{matched_pair_2_rep_sec}
If $D$ is a VB-algebroid with Lie algebroid structures on $D\to B$ and
$A\to M$, then the dual vector bundle $D\duer B\to B$ (see Appendix
\ref{appendix_dual}) has a \emph{Lie-Poisson structure} (a linear
Poisson structure), and the structure on $D\duer B$ is also
Lie-Poisson with respect to $D\duer B\to C^*$
\cite[3.4]{Mackenzie11}. Dualising this bundle gives a Lie algebroid
structure on $(D\duer B)\duer{C^*}\to C^*$. This equips the double
vector bundle $((D\duer B)\duer{C^*}; C^*,A;M)$ with a VB-algebroid
structure. Using the isomorphism defined by $-\nsp{\cdot}{\cdot}$ (see
Appendix \ref{appendix_dual}), the double vector bundle $(D\duer A\to
C^*;A\to M)$ is also a VB-algebroid. In the same manner, if $(D\to A,
B\to M)$ is a VB-algebroid then we use $\nsp{\cdot}{\cdot}$ to get a
VB-algebroid structure on $(D\duer B\to C^*;B\to M)$.

Let $\Sigma\colon A\times_MB\to D$ be a linear splitting of $D$ and
denote by $(\nabla^B,\nabla^C,R_A)$ the 2-representation of the Lie
algebroid $A$ on $\partial_B\colon C\to B$. The linear splitting
$\Sigma$ induces a linear splitting $\Sigma^\star\colon A\times_M
C^*\to D\duer A$ of $D\duer A$ (see Appendix \ref{appendix_dual}).
The 2-representation of $A$ that is associated to this splitting is
then $({\nabla^C}^*,{\nabla^B}^*,-R_A^*) $ on the complex
$\partial_B^*\colon B^*\to C^*$. This is proved in
the appendix  of \cite{DrJoOr15}. 

\medskip
A \textbf{double Lie algebroid} \cite{Mackenzie11} is a double vector
bundle $(D;A,B;M)$ with core $C$, and with Lie algebroid
structures on each of $A\to M$, $B\to M$, $D\to A$ and $D\to B$ such
that each pair of parallel Lie algebroids gives $D$ the structure of a
VB-algebroid, and such that
the pair $(D\duer A, D\duer B)$ with the induced Lie algebroid
structures on base $C^*$ and the pairing $\nsp{\cdot}{\cdot}$, is a
Lie bialgebroid.

Consider a double vector bundle $(D;A,B;M)$ with core $C$ and a VB-Lie
algebroid structure on each of its sides.  After the choice of a
splitting $\Sigma\colon A\times_M B\to D$, the Lie algebroid
structures on the two sides of $D$ are described as above by two
$2$-representations.  We prove in \cite{GrJoMaMe17} that $(D\duer A,
D\duer B)$ is a Lie bialgebroid over $C^*$ if and only if, for any
splitting of $D$, the two induced 2-representations form a matched
pair as in the following definition.

\begin{definition}\cite{GrJoMaMe17}\label{matched_pair_2_rep}
  Let $(A\to M, \rho_A, [\cdot\,,\cdot])$ and $(B\to M, \rho_B,
  [\cdot\,,\cdot])$ be two Lie algebroids and assume that $A$ acts on
  $\partial_B\colon C\to B$ up to homotopy via
  $(\nabla^{B},\nabla^{C}, R_{A})$ and $B$ acts on
  $\partial_A\colon C\to A$ up to homotopy via
  $(\nabla^{A},\nabla^{C}, R_{B})$\footnote{For the sake of
    simplicity, we write in this definition $\nabla$ for all the four
    connections. It will always be clear from the indexes which
    connection is meant. We write $\nabla^A$ for the $A$-connection
    induced by $\nabla^{AB}$ and $\nabla^{AC}$ on $\wedge^2 B^*\otimes
    C$ and $\nabla^B$ for the $B$-connection induced on $\wedge^2
    A^*\otimes C$.  }.  Then we say that the two representations up to
  homotopy form a matched pair if
\begin{enumerate}
\item[(M1)]
  $\nabla_{\partial_Ac_1}c_2-\nabla_{\partial_Bc_2}c_1=-(\nabla_{\partial_Ac_2}c_1-\nabla_{\partial_Bc_1}c_2)$,
\item[(M2)] $[a,\partial_Ac]=\partial_A(\nabla_ac)-\nabla_{\partial_Bc}a$,
\item[(M3)] $[b,\partial_Bc]=\partial_B(\nabla_bc)-\nabla_{\partial_Ac}b$,
\item[(M4)]
$\nabla_b\nabla_ac-\nabla_a\nabla_bc-\nabla_{\nabla_ba}c+\nabla_{\nabla_ab}c=
R_{B}(b,\partial_Bc)a-R_{A}(a,\partial_Ac)b$,
\item[(M5)]
  $\partial_A(R_{A}(a_1,a_2)b)=-\nabla_b[a_1,a_2]+[\nabla_ba_1,a_2]+[a_1,\nabla_ba_2]+\nabla_{\nabla_{a_2}b}a_1-\nabla_{\nabla_{a_1}b}a_2$,
\item[(M6)]
  $\partial_B(R_{B}(b_1,b_2)a)=-\nabla_a[b_1,b_2]+[\nabla_ab_1,b_2]+[b_1,\nabla_ab_2]+\nabla_{\nabla_{b_2}a}b_1-\nabla_{\nabla_{b_1}a}b_2$,
\end{enumerate}
for all $a,a_1,a_2\in\Gamma(A)$, $b,b_1,b_2\in\Gamma(B)$ and
$c,c_1,c_2\in\Gamma(C)$, and
\begin{enumerate}\setcounter{enumi}{6}
\item[(M7)] $\dr_{\nabla^A}R_{B}=\dr_{\nabla^B}R_{A}\in \Omega^2(A,
  \wedge^2B^*\otimes C)=\Omega^2(B,\wedge^2 A^*\otimes C)$, where
  $R_{B}$ is seen as an element of $\Omega^1(A, \wedge^2B^*\otimes
  C)$ and $R_{A}$ as an element of $\Omega^1(B, \wedge^2A^*\otimes
  C)$.
\end{enumerate}
\end{definition}

\subsection{The equivalence of $[2]$-manifolds with metric double vector bundles}\label{recall_n}
We quickly recall in this section the main result in
\cite{Jotz17a}. We refer the reader to
\cite{BoPo13, Jotz17a} for a quick review of split $\N$-manifolds, and for our
notation convention.  \medskip

A \textbf{metric double vector bundle} is a double vector bundle
$(\mathbb E, Q; B, M)$ with core $Q^*$, equipped with a \textbf{linear symmetric
  non-degenerate pairing $\mathbb E\times_B\mathbb E\to \R$},
i.e.~such that 
\begin{enumerate}
\item $\langle \tau_1^\dagger, \tau_2^\dagger\rangle=0$ for
  $\tau_1,\tau_2\in\Gamma(Q^*)$,
\item $\langle \chi, \tau^\dagger\rangle=q_B^*\langle q,\tau\rangle$
  for $\chi\in\Gamma_B^l(\mathbb E)$ linear over $q\in\Gamma(Q)$ and
  $\tau\in\Gamma(Q^*)$ and
\item $\langle\chi_1,\chi_2\rangle$ is a linear function on $B$ for
  $\chi_1,\chi_2\in \Gamma_B^l(\mathbb E)$.
\end{enumerate}
Note that the \textbf{opposite} $(\overline{\mathbb E};Q;B,M)$ of
a metric double vector bundle $(\mathbb E;B;Q,M)$ is the
metric double vector bundle with
$\langle\cdot\,,\cdot\rangle_{\overline{\mathbb
    E}}=-\langle\cdot\,,\cdot\rangle_{\mathbb E}$.

A linear
splitting $\Sigma\colon Q\times_MB\to \mathbb E$ is said to be
\textbf{Lagrangian} if its image is maximal isotropic in $\mathbb E\to
B$.  The corresponding horizontal lifts
$\sigma_Q\colon\Gamma(Q)\to\Gamma^l_B(\mathbb E)$ and
$\sigma_B\colon\Gamma(B)\to\Gamma^l_Q(\mathbb E)$
  are then also said to be
  \textbf{Lagrangian}. By definition, a horizontal lift $\sigma_Q\colon
\Gamma(Q)\to \Gamma^l_B(\mathbb E)$ is Lagrangian if and only if
$\langle \sigma_Q(q_1), \sigma_Q(q_2)\rangle=0$ for all $q_1,q_2\in
\Gamma(Q)$.
Showing the existence of a Lagrangian splitting of $\mathbb E$ is
relatively easy \cite{Jotz17a}.  Further, if $\Sigma^1$ and
$\Sigma^2\colon Q\times_M B\to \mathbb E$ are Lagrangian, then the
change of splitting $\phi_{12}\in\Gamma(Q^*\otimes Q^*\otimes B^*)$
defined by $\Sigma^2(q,b)=\Sigma^1(q,b)+\widetilde{\phi(q,b)}$ for all
$(q,b)\in Q\times_M B$, is a section of
$Q^*\wedge Q^* \otimes B^*$.

\begin{example}\label{metric_connections}
  Let $E\to M$ be a vector bundle endowed
  with a symmetric non-degenerate pairing
  $\langle\cdot\,,\cdot\rangle\colon E\times_M E\to \R$ (a
  \emph{metric vector bundle}).  Then
  $E\simeq E^*$ and the tangent double is a metric double vector
  bundle $(TE,E;TM,M)$ with pairing $TE\times_{TM}TE\to \R$ the
  tangent of the pairing $E\times_M E\to \R$. In particular, we have $\langle Te_1, Te_2\rangle_{TE}=\ell_{\dr\langle e_1,e_2\rangle}$, 
$\langle Te_1, e_2^\dagger\rangle_{TE}=p_M^* \langle e_1,e_2\rangle$ and $\langle e_1^\dagger, e_2^\dagger\rangle_{TE}=0$ for
$e_1,e_2\in\Gamma(E)$.

Recall that linear splittings of $TE$ are
equivalent to linear connections \linebreak $\nabla\colon
\mx(M)\times\Gamma(E)\to\Gamma(E)$.
The Lagrangian splittings of $TE$ are exactly the linear splittings
that correspond to \textbf{metric} connections, i.e.~linear
connections $\nabla\colon \mx(M)\times\Gamma(E)\to\Gamma(E)$ that
preserve the metric:
$\langle\nabla_\cdot e_1, e_2\rangle+\langle e_1,\nabla_\cdot
e_2\rangle=\dr\langle e_1, e_2\rangle$ for $e_1,e_2\in\Gamma(E)$.
\end{example}

Let $(\mathbb E, B; Q, M)$ be a metric double vector bundle. Define $\mathcal C(\mathbb
E)\subseteq \Gamma_Q^l(\mathbb E)$ as the $C^\infty(M)$-submodule of 
linear sections with isotropic image in $\mathbb E$.
After the choice of a Lagrangian splitting $\Sigma\colon Q\times_MB\to
\mathbb E$, $\mathcal C(\mathbb E)$ can be written
$\mathcal C(\mathbb
E):=\sigma_B(\Gamma(B))+
\{\tilde\omega\mid \omega\in\Gamma(Q^*\wedge Q^*)\}$.
This shows that $\mathcal
C(\mathbb E)$ together with $\Gamma^c_Q(\mathbb E)\simeq\Gamma(Q^*)$
span $\mathbb E$ as a vector bundle over $Q$.
\medskip

An \textbf{involutive double vector bundle} is a double vector bundle 
$(D,Q,Q,M)$ with core $B^*$ equipped with a morphism $\mathcal I\colon
D\to D$ of double vector bundles
satisfying $\mathcal I^2=\Id_D$ and $\pi_1\circ\mathcal I=\pi_2$,
$\pi_2\circ\mathcal I=\pi_1$, where $\pi_1,\pi_2\colon D\to Q$ are the
two side projections,
and with core morphism $-\Id_{B^*}\colon B^*\to B^*$. A \textbf{morphism} 
$\Omega\colon D_1\to D_2$
of \textbf{involutive double vector bundles} is a morphism 
of double vector bundles such that
$\Omega\circ\mathcal I_1=\mathcal I_2\circ\Omega$.
\cite[Proposition 3.15]{Jotz17a} proves a duality of involutive double vector bundles
with metric double vector bundles: the dual $(D\duer{\pi_1}; Q,B;M)$
with core $Q^*$ carries an induced linear metric. Conversely, given a
metric double vector bundle $(\mathbb E; Q,B;M)$ with core $Q^*$, the
dual
$(\mathbb E\duer{Q}; Q,Q;M)$ with core $B^*$ carries an induced
involution as above. 
We define morphisms of metric double vector bundles as the dual morphisms to morphisms of involutive double vector bundles.
A \textbf{morphism} 
$\Omega\colon \mathbb F\to\mathbb E$
of \textbf{metric double vector bundles}  is  hence a  relation
$\Omega\subseteq \overline{\mathbb F}\times\mathbb E$ that is the dual
of a morphism of involutive double vector bundles
$\omega\colon \mathbb F\duer P\to \mathbb E\duer Q$. 
\begin{equation*}
  \begin{xy}
    \xymatrix{\mathbb F\duer P\ar[rrr]^{\omega}\ar[rd]\ar[dd]& && \mathbb E\duer Q\ar[rd]\ar[dd]&\\
& P\ar[dd]\ar[rrr]^{\omega_P}&&&Q\ar[dd]\\
      P^{**}\ar[rd]\ar[rrr]&&&Q^{**}\ar[rd]&\\
&N\ar[rrr]^{\omega_0}&&&M}
\end{xy}
\end{equation*}
Note that he dual of $\Omega$ is compatible with the involutions if
and only if $\Omega$ is an isotropic subspace of
$\overline{\mathbb F}\times\mathbb E$.  
Equivalently \cite{Jotz17a}, one can define a morphism
$\Omega\colon \mathbb F\to\mathbb E$ of metric double vector bundles
as a pair of maps
$\omega^\star \colon \mathcal C(\mathbb E)\to \mathcal C(\mathbb F)$,
$\omega_P^\star \colon \Gamma(Q^*)\to \Gamma(P^*)$ together with a
smooth map $\omega_0 \colon N\to M$ such that
\begin{enumerate}
\item $\omega^\star\left(\widetilde{\tau_1\wedge\tau_2}\right)=\widetilde{\omega_P^\star\tau_1\wedge\omega_P^\star\tau_2}$,
\item
  $\omega^\star(q_Q^*f\cdot\chi)=q_P^*(\omega_0^*f)\cdot\omega^\star(\chi)$
  and
\item $\omega_P^\star(f\cdot \tau)=\omega_0^*f\cdot \omega_P^\star\tau$
\end{enumerate}
for all $\tau,\tau_1,\tau_2\in\Gamma(Q^*)$, $f\in C^\infty(M)$ and
$\chi \in \mathcal C(\mathbb E)$. We write $\operatorname{MDVB}$
for the obtained category of metric double vector bundles.
The following theorem is proved in \cite{Jotz17a} and
independently in \cite{delCarpio-Marek15}. 
\begin{theorem}[\cite{Jotz17a}]\label{main_crucial}
  There is a (covariant) equivalence between the category of
  $[2]$-manifolds and the category of involutive double vector
  bundles.
\end{theorem}
Combining the obtained equivalence with the (contravariant)
dualisation equivalence of $\operatorname{IDVB}$ with
$\operatorname{MDVB}$ yields a (contravariant) equivalence between the category of
metric double vector bundles with the morphisms defined above and the
category of $[2]$-manifolds.  This equivalence establishes in
particular an equivalence between split $[2]$-manifold
$\mathcal M=Q[-1]\oplus B^*[-2]$ and the decomposed metric double
vector bundle $(Q\times_M B\times_MQ^*,B,Q,M)$ with the obvious linear
metric over $B$.

We quickly describe the functors between the two categories.  To
construct the geometrisation functor
$\mathcal G\colon [2]\rm{-Man}\to \operatorname{MDVB}$, take a
$[2]$-manifold and considers its local trivialisations.  Changes of
local trivialisation define a set of cocycle conditions, that
correspond exactly to cocycle conditions for a double vector bundle
atlas. The local trivialisations can hence be collated to a double
vector bundle, which naturally inherits a linear pairing. 
See \cite{Jotz17a} for more details, and remark that this
construction is as simple as the construction of a vector
bundle from a locally free and finitely generated sheaf of
$C^\infty(M)$-modules.

Conversely, the algebraisation functor $\mathcal A$ sends a metric
double vector bundle $\mathbb E$ to the $[2]$-manifold defined as
follows: the functions of degree $1$ are the sections of
$\Gamma_Q^c(\mathbb E)\simeq\Gamma(Q^*)$, and the functions of degree
$2$ are the elements of $\mathcal C(\mathbb E)$. The multiplication of
two core sections $\tau_1,\tau_2\in\Gamma(Q^*)$ is the core-linear
section $\widetilde{\tau_1\wedge\tau_2}\in\mathcal C(\mathbb E)$.

Note that while that equivalence can be
seen as the special case of trivial homological vector field vs trivial
bracket and anchor of Li-Bland's bijection of Lie $2$-algebroids with
VB-Courant algebroids \cite{Li-Bland12}, this corollary is not given
there and only a very careful study of Li-Bland's proof, which would
amount to the work done in \cite{Jotz17a} would yield it.

\section{Split Lie 2-algebroids}\label{sec:split_lie_2}
In this section we recall the notions of Courant algebroids, Dirac
structures, dull algebroids, Dorfman connections and (split) Lie
$2$-algebroids.
\subsection{Courant algebroids}\label{background_courant_notions}
We introduce in this section a generalisation of the notion of Courant
algebroid, namely the one of \emph{degenerate Courant algebroid with
  pairing in a vector bundle}.  Later we will see that the fat bundle
associated to a VB-Courant algebroid carries a natural Courant
algebroid structure with pairing in the dual of the base.

An anchored vector bundle is a vector bundle $Q\to M$ endowed with a
vector bundle morphism $\rho_Q\colon Q\to TM$ over the identity.
Consider an anchored vector bundle $(\mathsf E\to M, \rho)$ and a
vector bundle $V$ over the same base $M$ together with a morphism
$\tilde\rho\colon\mathsf E\to \operatorname{Der}(V)$, such that the
symbol of $\tilde\rho(e)$ is $\rho(e)\in \mx(M)$ for all $e\in
\Gamma(\mathsf E)$.  Assume that $\mathsf E$ is paired with itself via
a nondegenerate pairing $\langle\cdot\,,\cdot\rangle\colon \mathsf
E\times_M \mathsf E\to V$ with values in $V$. Define $\mathcal D\colon
\Gamma(V)\to \Gamma(\mathsf E)$ by $\langle \mathcal Dv,
e\rangle=\tilde\rho(e)(v)$ for all $v\in\Gamma(V)$.  Then $\mathsf
E\to M$ is a \textbf{Courant algebroid with pairing in $V$} if
$\mathsf E$ is in addition equipped with an $\R$-bilinear bracket
$\lb\cdot\,,\cdot\rb$ on the smooth sections $\Gamma(\mathsf E)$ such
that the following conditions are satisfied:
\begin{enumerate}
\item[(CA1)] $\lb e_1, \lb e_2, e_3\rb\rb= \lb \lb e_1, e_2\rb, e_3\rb+ \lb
  e_2, \lb e_1, e_3\rb\rb$,
\item[(CA2)] $\tilde\rho(e_1 )\langle e_2, e_3\rangle= \langle\lb e_1,
  e_2\rb, e_3\rangle + \langle e_2, \lb e_1 , e_3\rb\rangle$,
\item[(CA3)] $\lb e_1, e_2\rb+\lb e_2, e_1\rb =\mathcal D\langle e_1 ,
  e_2\rangle$,
\item[(CA4)]     $\tilde\rho\lb e_1, e_2\rb = [\tilde\rho(e_1), \tilde\rho(e_2)]$
\end{enumerate}
for all $e_1, e_2, e_3\in\Gamma(\mathsf E)$ and $f\in C^\infty(M)$.
Equation (CA2) implies $\lb e_1, f e_2\rb= f \lb e_1 , e_2\rb+
(\rho(e_1 )f )e_2$ for all $e_1, e_2\in\Gamma(\mathsf E)$ and $f\in
C^\infty(M)$.  If $V=\mathbb R\times M\to M$ is in addition the
trivial bundle, then $\mathcal D = 
\rho^*\circ\dr \colon C^\infty(M)\to\Gamma(\mathsf E)$, where $\mathsf
E$ is identified with $\mathsf E^*$ via the pairing.  The quadruple
$(\mathsf E\to M, \rho, \langle\cdot\,,\cdot\rangle,
\lb\cdot\,,\cdot\rb)$ is then a \textbf{Courant algebroid}
\cite{LiWeXu97,Roytenberg99} and (CA4) follows then from (CA1), (CA2)
and (CA3) (see \cite{Uchino02} and also \cite{Jotz13a} for a quicker
proof).

Note that Courant algebroids with a pairing in a vector bundle $E$ were defined in
\cite{ChLiSh10} and called \emph{$E$-Courant algebroids}. 
It is easy
to check that Li-Bland's \emph{$AV$-Courant algebroids}
\cite{Li-Bland11} yield a special class of degenerate Courant
algebroids with pairing in $V$. The examples of Courant algebroids
with pairing in a vector bundle that we will get in Theorem \ref{fat} are
\emph{not} $AV$-Courant algebroids, so the two notions are
distinct.

\medskip

In our study of VB-Courant algebroids, we will need the following two lemmas. 
\begin{lemma}[\cite{Roytenberg02}]\label{roytenberg_useful}
  Let $(\mathsf E\to M, \rho, \langle\cdot\,,\cdot\rangle,
  \lb\cdot\,,\cdot\rb)$ be a Courant algebroid.  For all $\theta\in
  \Omega^1(M)$ and $e\in\Gamma(\mathsf E)$, we have:
  \[ \lb e,
  \rho^*\theta\rb=\rho^*(\ldr{\rho(e)}\theta),
  \qquad \lb\rho^*\theta,
  e\rb=-\rho^*(\ip{\rho(e)}\dr\theta)
\]
and so
$\rho(\rho^*\theta)=0$, which implies $\rho\circ \mathcal D=0$.
\end{lemma}

\begin{lemma}[\cite{Li-Bland12}]\label{useful_lemma}
  Let $\mathsf {E}\to M$ be a vector bundle, $\rho\colon \mathsf{E}\to
  TM$ be a bundle map, $\langle\cdot,\cdot\rangle$ a nondegenerate
  pairing on $\mathsf{E}$, and let $\mathcal
  S\subseteq\Gamma(\mathsf{E})$ be a subspace of sections which
  generates $\Gamma(\mathsf{E})$ as a $C^\infty(M)$-module. Suppose
  that $\lb\cdot\,,\cdot\rb:\mathcal S\times\mathcal S\to \mathcal S$
  is a bracket satisfying
\begin{enumerate}
\item $\lb s_1,\lb s_2,s_3\rb\rb=\lb\lb s_1,s_2\rb,s_3\rb
+\lb s_2,\lb s_1,s_3\rb\rb$, 
\item $\rho(s_1)\langle s_2,s_3\rangle=\langle
  \lb s_1,s_2\rb, s_3\rangle+\langle s_2, \lb
  s_1,s_3\rb\rangle$,
\item $\lb s_1,s_2\rb+\lb s_2,s_1\rb=\rho^*\dr
  \langle s_1,s_2\rangle$,
\item $\rho\lb s_1,s_2\rb=[\rho(s_1),\rho(s_2)]$,
\end{enumerate}
for any $s_i\in \mathcal S$, and that $\rho\circ\rho^*=0$. Then
there is a unique extension of $\lb\cdot\,,\cdot\rb$ to a 
bracket on all of $\Gamma(\mathsf{E})$ such that $(\mathsf E, \rho,
\langle\cdot\,,\cdot\rangle, \lb\cdot\,,\cdot\rb)$ is a Courant algebroid.
\end{lemma}

\medskip

A \textbf{Dirac structure with support} \cite{AlXu01} in a Courant
algebroid $\mathsf E\to M$ is a subbundle $D\to S$ over a sub-manifold
$S$ of $M$, such that $D(s)$ is maximal isotropic in $\mathsf E(s)$
for all $s\in S$ and
\[ e_1\an{S}\in\Gamma_S(D), e_2\an{S}\in\Gamma_S(D) \quad \Rightarrow\quad \lb e_1, e_2\rb\an{S}\in\Gamma_S(D)
\]
for all $e_1,e_2\in\Gamma(\mathsf E)$.
We leave to the reader the proof of the following lemma.
\begin{lemma}\label{useful_for_dirac_w_support}
  Let $\mathsf E\to M$ be a Courant algebroid and $D\to S$ a
  subbundle; with $S$ a sub-manifold of $M$. Assume that $D\to S$ is
  spanned by the restrictions to $S$ of a family $\mathcal S\subseteq
  \Gamma(\mathsf E)$ of sections of $\mathsf E$. Then $D$ is a Dirac
  structure with support $S$ if and only if
\begin{enumerate}
\item $\rho_{\mathsf E}(e)(s)\in T_sS$ for all $e\in\mathcal S$ and $s\in S$,
\item $D_s$ is Lagrangian in $\mathbb E_s$ for all $s\in S$ and
\item $\lb e_1,e_2\rb\an{S}\in\Gamma_S(D)$ for all $e_1,e_2\in\mathcal
  S$.
\end{enumerate}
\end{lemma}
\medskip

Next we recall the notion of Dorfman connection \cite{Jotz13a}.
  Let $(Q\to M,\rho_Q)$ be an anchored vector bundle.  
A
  \textbf{Dorfman ($Q$-)connection on $Q^*$} is an $\R$-linear map
 $ \Delta\colon \Gamma(Q)\to \operatorname{Der}(Q^*)$
such that 
\begin{enumerate}
\item $\Delta_q$ is a derivation over $\rho_Q(q)\in\mx(M)$,
\item $\Delta_{f q}\tau=f\Delta_q\tau+\langle q, \tau\rangle \cdot \rho_Q^*\dr f$ and
\item $\Delta_{q}\rho_Q^*\dr f=\rho_Q^*\dr(\rho_Q(q)f)$
\end{enumerate}
for all $f\in C^\infty(M)$, $q,q'\in\Gamma(Q)$, $\tau\in\Gamma(Q^*)$.

The map $\lb\cdot\,,\cdot\rb_\Delta=\Delta^*\colon
\Gamma(Q)\times\Gamma(Q)\to \Gamma(Q)$ that is dual to $\Delta$ in the
sense of dual derivations, i.e. $\langle \Delta^*_{q_1}q_2,
\tau\rangle=\rho_Q(q_1)\langle q_2,\tau\rangle -\langle q_2,
\Delta_{q_1}\tau\rangle$ for all $q_1,q_2\in \Gamma(Q)$ and
$\tau\in\Gamma(Q^*)$, is then a \emph{dull bracket} on $\Gamma(Q)$ in
the following sense.
A \textbf{dull algebroid} is an anchored vector bundle $(Q\to M,
  \rho_Q)$ with a bracket $\lb\cdot\,,\cdot\rb $ on $\Gamma(Q)$ such that
\begin{equation}\label{anchor_preserves_bracket}
  \rho_Q\lb q_1, q_2\rb=[\rho_Q(q_1),\rho_Q(q_2)]
\end{equation}
and  (the Leibniz identity)
\begin{equation*} \lb f_1 q_1, f_2
  q_2\rb=f_1f_2\lb q_1,
  q_2\rb+f_1\rho_Q(q_1)(f_2)q_2-f_2\rho_Q(q_2)(f_1)q_1
\end{equation*}
for all $f_1,f_2\in C^\infty(M)$, $q_1, q_2\in\Gamma(Q)$.
In other words, a dull algebroid is a \textbf{Lie algebroid} if its
bracket is in addition skew-symmetric and satisfies the Jacobi
identity. Note that a dull bracket can easily be skew-symmetrised. 

The \emph{curvature} of a Dorfman connection $\Delta\colon
\Gamma(Q)\times\Gamma(Q^*)\to\Gamma(Q^*)$ is the map
\[R_\Delta\colon \Gamma(Q)\times\Gamma(Q)\to\Gamma(Q^*\otimes Q^*),\]
defined on $q,q'\in\Gamma(Q)$ by
$R_\Delta(q,q'):=\Delta_q\Delta_{q'}-\Delta_{q'}\Delta_q-\Delta_{\lb q,q'\rb_\Delta}$. The curvature satisfies
\begin{equation}\label{curv_dual_Jac}
\langle \tau, \Jac_{\lb\cdot\,,\cdot\rb_\Delta}(q_1,q_2,q_3)\rangle =\langle R_\Delta(q_1,q_2)\tau, q_3\rangle
\end{equation}
for $q_1,q_2,q_3\in\Gamma(Q)$ and $b\in\Gamma(B)$, where 
\[ \Jac_{\lb\cdot\,,\cdot\rb_\Delta}(q_1,q_2,q_3)=\lb \lb
q_1,q_2\rb_\Delta,q_3\rb_\Delta+\lb q_2,\lb q_1,q_3\rb_\Delta-\lb
q_1,\lb q_2,q_3\rb_\Delta\rb_\Delta
\]
is the Jacobiator of $\lb\cdot\,,\cdot\rb_\Delta$ in Leibniz
form. Hence, the Dorfman connection is flat if and only if the
corresponding dull bracket satisfies the Jacobi identity in Leibniz
form. 

\subsection{Split Lie 2-algebroids}\label{dorfman_eq_split}
A \textbf{homological} vector field $\chi$ on an $[n]$-manifold
$\mathcal M$ is a derivation of degree $1$ of $C^\infty(\mathcal M)$
such that $\mathcal Q^2=\frac{1}{2}[\mathcal Q,\mathcal Q]$ vanishes.
A homological vector field on a $[1]$-manifold $\mathcal M=E[-1]$ is
the de Rham differential $\dr_E$ associated to a Lie algebroid
structure on $E$ \cite{Vaintrob97}. A \textbf{Lie n-algebroid} is an
$[n]$-manifold endowed with a homological vector field (an
\emph{$\N\mathcal Q$-manifold} of degree $n$).

A \textbf{split Lie n-algebroid} is a split $[n]$-manifold endowed
with a homological vector field.  Split Lie n-algebroids were studied
by Sheng and Zhu \cite{ShZh17} and described as vector bundles endowed
with a bracket that satisfies the Jacobi identity up to some
correction terms, see also \cite{BoPo13}. Our definition of a split
Lie $2$-algebroid turns out to be a Lie algebroid version of Baez and
Crans' definition of a Lie $2$-algebra \cite{BaCr04}.
\begin{definition}\label{split_Lie2}
  A split Lie 2-algebroid $B^*\to Q$ is the pair of an
  anchored vector bundle \footnote{The names that we choose for the
    vector bundles will become natural in a moment.}  $(Q\to M,
  \rho_Q)$ and a vector bundle $B\to M$, together with a vector bundle
  map $l\colon B^*\to Q$, a skew-symmetric dull bracket\footnote{To
    get the definition in \cite{ShZh17}, set
    $l_1:=-l$, $l_3:=\omega$ and consider the skew symmetric bracket
    $l_2\colon \Gamma(Q\oplus B^*)\times\Gamma(Q\oplus B^*)\to
    \Gamma(Q\oplus B^*)$, $l_2((q_1,\beta_1),(q_2,\beta_2))=(\lb q_1,
    q_2\rb, \nabla_{q_1}^*\beta_2-\nabla_{q_2}^*\beta_1)$ for
    $q_1,q_2\in\Gamma(Q)$ and $\beta_1,\beta_2\in\Gamma(B^*)$. Note
    that this bracket satisfies a Leibniz identity with anchor
    $\rho_Q\circ \pr_Q\colon Q\oplus B^*\to TM$ and that the
    Jacobiator of this bracket is given by $\Jac_{l_2}((q_1,\beta_1),
    (q_2,\beta_2), (q_3,\beta_3)) =(-l(\omega(q_1,q_2,q_3)),
    \omega(q_1,q_2,l(\beta_3))+\rm{c.p.}$ } $\lb\cdot\,,\cdot\rb\colon
  \Gamma(Q)\times\Gamma(Q)\to\Gamma(Q)$, a linear connection
  $\nabla\colon\Gamma(Q)\times\Gamma(B)\to\Gamma(B)$ and a vector
  valued $3$-form $\omega\in \Omega^3(Q,B^*)$, such that
\begin{enumerate}
\item[(i)] $\nabla_{l(\beta_1)}^*\beta_2+\nabla_{l(\beta_2)}^*\beta_1=0$ for
  all $\beta_1,\beta_2\in\Gamma(B^*)$,
\item[(ii)] $\lb q, l(\beta)\rb=l(\nabla_q^*\beta)$ for
  $q\in\Gamma(Q)$ and $\beta\in\Gamma(B^*)$,
\item [(iii)]$\Jac_{\lb \cdot,\cdot\rb}=l\circ \omega\in \Omega^3(Q,Q)$,
\item[(iv)] $R_{\nabla}(q_1,q_2)b=l^*\langle \ip{q_2}\ip{q_1}\omega,b\rangle$
  for $q_1,q_2\in\Gamma(Q)$ and $b\in\Gamma(B)$, and 
\item[(v)] $\dr_{\nabla^*}\omega=0$.
\end{enumerate}
\end{definition}
From (iii) follows the identity $\rho_Q\circ l=0$.
In the following, we will also work with
$\partial_B:=l^*\colon Q^*\to B$, with the Dorfman connection
$\Delta\colon\Gamma(Q)\times\Gamma(Q^*)\to\Gamma(Q^*)$ that is dual to
$\lb\cdot\,,\cdot\rb$, and with
$R_\omega\in\Omega^2(Q,\operatorname{Hom}(B,Q^*))$ defined by
$R_\omega(q_1,q_2)b=\langle\ip{q_2}\ip{q_1}\omega,b\rangle$. Then (ii)
is equivalent to $\partial_B\circ \Delta_q=\nabla_q\circ \partial_B$
and (iii) is $R_\omega(q_1,q_2)\circ \partial_B=R_\Delta(q_1,q_2)$ for
$q, q_1, q_2\in\Gamma(Q)$.

\subsection{Split Lie-$2$-algebroids as split $[2]$Q-manifolds}\label{explicit_Q}
Before we go on with the study of examples, we briefly describe how to
construct from the objects in Definitions \ref{split_Lie2} the
corresponding homological vector fields on split $[2]$-manifolds.
Note that local descriptions of homological vector fields are also
given in \cite{ShZh17} and \cite{BoPo13}.

Consider a split $[2]$-manifold $\mathcal M=Q[-1]\oplus B^*[-2]$.
Assume that $Q$ is endowed with an anchor $\rho_Q$ and a
skew-symmetric dull bracket $\lb\cdot\,,\cdot\rb$, that it acts on $B$
via a linear connection $\nabla\colon
\Gamma(Q)\times\Gamma(B)\to\Gamma(B)$, that $\omega$ is an element of
$\Omega^3(Q,B^*)$ and that $\partial_B\colon Q^*\to B$ is a vector
bundle morphism.  Define a vector field $\mathcal Q$ of degree $1$ on
$\mathcal M$ by the following formulas:
\[\mathcal Q(f)=\rho_Q^*\dr f\in \Gamma(Q^*)
\]
for $f\in C^\infty(M)$, 
\[\mathcal Q(\tau)=\dr_Q\tau+\partial_B\tau \in \Omega^2(Q)\oplus\Gamma(B)
\]
for $\tau\in\Gamma(Q^*)$ 
and 
\[\mathcal Q(b)=\dr_\nabla b-\langle\omega,b\rangle\in \Omega^1(Q,B)\oplus\Omega^3(Q).
\]
for $b\in\Gamma(B)$.  Conversely, a relatively easy degree count and
study of the graded Leibniz identity for an arbitrary vector field of
degree $1$ on $\mathcal M=Q[-1]\oplus B^*[-2]$ shows that it must be
given as above, defining therefore an anchor $\rho_Q$, and the
structure objects $\lb\cdot\,,\cdot\rb$, $\nabla$, $\omega$ and
$\partial_B$.

We show that $\mathcal Q^2=0$ if and only if $(\partial_B^*\colon
B^*\to Q,\lb\cdot\,,\cdot\rb,\nabla,\omega)$ is a split Lie
$2$-algebroid anchored by $\rho_Q$.  For
$f\in C^\infty(M)$ we have \[ \mathcal Q^2(f)=\dr_Q(\rho_Q^*\dr
f)+\partial_B(\rho_Q^*\dr f)\in \Omega^2(Q)\oplus \Gamma(B).
\] Hence $\mathcal Q^2(f)=0$ for all $f\in C^\infty(M)$ if and only if
$\partial_B\circ\rho_Q^*=0$
and $\rho_Q\lb q_1,q_2\rb_\Delta=[\rho_Q(q_1), \rho_Q(q_2)]$ for all
$q_1,q_2\in\Gamma(Q)$. 
Now we assume that these two conditions are satisfied.
For $\tau\in\Gamma(Q^*)$ we have 
\[\mathcal
Q^2(\tau)=(\dr_Q^2\tau-\langle\omega,\partial_B\tau\rangle)+(\partial_B\dr_Q\tau+\dr_\nabla(\partial_B\tau))\in\Omega^3(Q)\oplus\Omega^1(Q,
B),
\]
where $\partial_B\colon \Omega^k(Q)\to\Omega^{k-1}(Q,B)$ is
the vector bundle morphism defined by
\[\partial_B(\tau_1\wedge\ldots\wedge\tau_k)=\sum_{i=1}^k(-1)^{i+1}\tau_1
\wedge\ldots \wedge\hat i\wedge \ldots\tau_k\wedge\partial_B\tau_i\]
for all $\tau_1,\tau_2\in\Gamma(Q^*)$. We find
$\dr_Q^2\tau(q_1,q_2,q_3)=\langle\operatorname{Jac}_{\lb\cdot\,,\cdot\rb}(q_1,q_2,q_3),
\tau\rangle$ and
$(\partial_B\dr_Q\tau)(q,\beta)=-\langle \partial_B\Delta_q\tau,\beta\rangle$,
and so 
$\mathcal Q^2(\tau)=0$ for all $\tau\in\Gamma(Q^*)$ if and only if 
$\operatorname{Jac}_{\lb\cdot\,,\cdot\rb}(q_1,q_2,q_3)=\partial_B^*\omega(q_1,q_2,q_3)$ for
all $q_1,q_2,q_3\in\Gamma(Q)$ and 
$\partial_B\Delta_q\tau=\nabla_q(\partial_B\tau)$ for all
$q\in\Gamma(Q)$ and $\tau\in\Gamma(Q^*)$.

Finally, we find for $b\in\Gamma(B)$:
\[\mathcal Q^2(b)=\mathcal Q(\dr_\nabla b)-\dr_Q\langle\omega,
b\rangle -\partial_B\langle\omega,
b\rangle.
\]
The term $\partial_B\langle\omega,
b\rangle$ is an element of $\Omega^2(Q,B)$ and the term $\dr_Q\langle\omega,
b\rangle $ is an element of $\Omega^4(Q)$.
A computation yields that the $\Omega^4(Q)$-term of $\mathcal Q(\dr_\nabla 
b)$ is
$-\langle\omega,\dr_\nabla b\rangle$, which is defined by
\[\langle\omega,\dr_\nabla
b\rangle(q_1,q_2,q_3,q_4)=\sum_{\sigma\in
  Z_4}(-1)^\sigma\langle\omega(q_{\sigma(1)},q_{\sigma(2)},q_{\sigma(3)}),\nabla_{q_{\sigma(4)}}b\rangle,
\]
where $Z_4$ is the group of cyclic permutations of $\{1,2,3,4\}$.  The
$\Omega^2(Q, B)$-term is $R_\nabla(\cdot,\cdot)b$ and the
$\Gamma(S^2B)$-term is $\nabla_{\partial_B^*}b$ defined by
$(\nabla_{\partial_B^*}b)(\beta_1,\beta_2)
=\langle\nabla_{\partial_B^*\beta_1}b,\beta_2\rangle+\langle\nabla_{\partial_B^*\beta_2}b,\beta_1\rangle$
for all $\beta_1,\beta_2\in\Gamma(B^*)$. Hence $\mathcal Q^2(b)=0$ if
and only if $\dr_Q\langle\omega, b\rangle+\langle\omega,\dr_\nabla
b\rangle=0$, which is equivalent to $\dr_{\nabla^*}\omega=0$;
$\nabla_{\partial_B^*}b=0$, which is equivalent to
$\nabla_{\partial_B^*\beta_1}^*\beta_2+\nabla_{\partial_B^*\beta_2}^*\beta_1=0$
for all $\beta_1,\beta_2\in\Gamma(B^*)$; and
$R_\nabla(\cdot,\cdot)b=\partial_B\langle\omega, b\rangle$, which is
equivalent to
$R_{\nabla^*}(q_1,q_2)\beta=\omega(q_1,q_2,\partial_B^*\beta)$ for all
$q_1,q_2\in\Gamma(Q)$ and $\beta\in\Gamma(B^*)$.

\subsection{Examples of split Lie 2-algebroids}\label{examples_split_lie_2}
We describe here four classes of examples of split Lie
2-algebroids. Later we will discuss their geometric meanings. We do
not verify in detail the axioms of
split Lie 2-algebroids. The computations in order to do this for
Examples \ref{standard} and \ref{adjoint} are long, but
straightforward. Note that, alternatively, the next section will
provide a geometric proof of the fact that the following objects are
split Lie $2$-algebroids, since we
will find them to be equivalent to special classes of VB-Courant
algebroids. Note finally that a fifth important class of examples in
discussed in Section \ref{double}.

\subsubsection{Lie algebroid representations} 
Let $(Q\to M,\rho,[\cdot\,,\cdot])$ be a Lie algebroid and
$\nabla\colon\Gamma(Q)\times\Gamma(B)\to\Gamma(B)$ a representation of
$Q$ on a vector bundle $B$.  Then $(0\colon B^*\to Q, [\cdot\,,\cdot],
\nabla, 0)$ is a split Lie 2-algebroid.  It is a semi-direct extension
of the Lie algebroid $Q$ (and a special case of the bicrossproduct Lie
2-algebroids defined in \S\ref{bicross}): the corresponding bracket $l_2$ is given by
$l_2(q_1+\beta_1,q_2+\beta_2)=[q_1,q_2]+(\nabla_{q_1}^*\beta_2-\nabla_{q_2}^*\beta_1)$
for $q_1,q_2\in\Gamma(Q)$ and $\beta_1,\beta_2\in\Gamma(B^*)$. Hence
$(Q\oplus B^*\to M, \rho=\rho_Q\circ\pr_Q, l_2)$ is simply a Lie
algebroid.

\subsubsection{Standard split Lie 2-algebroids}\label{standard}
Let $E\to M$ be a vector bundle, set $\partial_E=\pr_E\colon E\oplus
T^*M\to E$, consider a skew-symmetric dull bracket
$\lb\cdot\,,\cdot\rb$ on $\Gamma(TM\oplus E^*)$, with $TM\oplus E^*$
anchored by $\pr_{TM}$, and let $\Delta\colon \Gamma(TM\oplus
E^*)\times\Gamma(E\oplus T^*M)\to\Gamma(E\oplus T^*M)$ be the dual
Dorfman connection. This defines as follows a split Lie 2-algebroid
structure on the vector bundles $(TM\oplus E^*,\pr_{TM})$ and $E^*$.

Let $\nabla\colon\Gamma(TM\oplus E^*)\times\Gamma(E)\to\Gamma(E)$ be
the ordinary linear connection\footnote{To see that
  $\nabla=\pr_E\circ\Delta\circ\iota_E$ is an ordinary connection,
  recall that since $TM\oplus E^*$ is anchored by $\pr_{TM}$, the map
  $\dr_{E\oplus T^*M}=\pr_{TM}^*\dr\colon C^\infty(M)\to
  \Gamma(E\oplus T^*M)$ sends $f\to (0,\dr f)$.}  defined by
$\nabla=\pr_E\circ\Delta\circ\iota_E$.  The vector bundle map $l=\pr_E^*\colon
E^*\to TM\oplus E^*$ is just the canonical inclusion.
Define $\omega$ by
$\omega(v_1,v_2,v_3)=\Jac_{\lb\cdot\,,\cdot\rb}(v_1,v_2,v_3)$. 
Note that since $TM\oplus E^*$ is anchored by
$\pr_{TM}$, the tangent part of the dull bracket must just be the Lie
bracket of vector fields. The Jacobiator $\Jac_{\lb\cdot\,,\cdot\rb}$
can hence be seen as an element of $\Omega^3(TM\oplus E^*,
E^*)$.

A straightforward verification of the axioms shows that $l$,
$\lb\cdot\,,\cdot\rb$, $\nabla^*$, $\omega$ define a split Lie
2-algebroid. For reasons that will become clearer in
\S\ref{standard_VB_Courant_ex}, we call \emph{standard} this type of
split Lie 2-algebroid.

\subsubsection{Adjoint split Lie 2-algebroids}\label{adjoint}
The \emph{adjoint} split Lie 2-algebroids can be described as follows.  
Let $\mathsf E\to M$ be a Courant algebroid with anchor $\rho_{\mathsf
  E}$ and bracket $\lb\cdot\,,\cdot\rb$ and choose a metric linear
connection $\nabla\colon\mx(M)\times\Gamma(\mathsf E)\to\Gamma(\mathsf
E)$, i.e.~a linear connection that preserves the pairing. Set
$\partial_{TM}=\rho_{\mathsf E}\colon \mathsf E\to TM$ and identify
$\mathsf E$ with its dual via the pairing.  The map $\Delta\colon
\Gamma(\mathsf E)\times\Gamma(\mathsf E)\to\Gamma(\mathsf E)$,
\[\Delta_ee'=\lb e,e'\rb+\nabla_{\rho(e')}e\]
is a Dorfman connection, which we call the \emph{basic Dorfman connection associated to
$\nabla$}. The dual skew-symmetric (!) dull bracket is given by $
\lb e,e'\rb_\Delta=\lb e,e'\rb-\rho^*\langle\nabla_\cdot
e,e'\rangle $ for all $e,e'\in\Gamma(\mathsf E)$.  The map
$\nabla^{\rm bas}\colon\Gamma(\mathsf E)\times\mx(M)\to\mx(M)$,
$\nabla^{\rm bas}_eX=[\rho(e),X]+\rho(\nabla_{X}e)$
is a linear connection, the \emph{basic connection associated to $\nabla$}.

We now define the \emph{basic curvature} $R_\Delta^{\rm
  bas}\in\Omega^2(\mathsf E,\Hom(TM,\mathsf E))$ by\footnote{
  We have then $R_\Delta^{\rm bas}(e_1,e_2)X=
  -\nabla_X\lb e_1,e_2\rb_\Delta+\lb \nabla_Xe_1,e_2\rb_\Delta+\lb
  e_1,\nabla_Xe_2\rb_\Delta +\nabla_{\nabla^{\rm
      bas}_{e_2}X}e_1-\nabla_{\nabla^{\rm
      bas}_{e_1}X}e_2-\beta\inv\rho^*\langle
  R_\nabla(X,\cdot)e_1,e_2\rangle$. Using
  $-R_\nabla^*=R_{\nabla^*}=R_\nabla$ (where we identify $\mathsf E$
  with its dual using $\langle\cdot\,,\cdot\rangle$), the identity
  $R_\Delta^{\rm bas}(e_1,e_2)=-R_\Delta^{\rm bas}(e_2,e_1)$ is then
  immediate.}
\begin{align}
  R_\Delta^{\rm bas}(e_1,e_2)X=& -\nabla_X\lb e_1,e_2\rb+\lb
  \nabla_Xe_1,e_2\rb+\lb e_1,\nabla_Xe_2\rb \label{def_R_nabla_bas}\\
&+\nabla_{\nabla^{\rm
      bas}_{e_2}X}e_1-\nabla_{\nabla^{\rm
      bas}_{e_1}X}e_2-\beta\inv\langle \nabla_{\nabla^{\rm bas}_\cdot
    X}e_1,e_2\rangle\nonumber
\end{align}
for all $e_1,e_2\in\Gamma(\mathsf E)$ and $X\in\mx(M)$. 
Note the similarity of these constructions with the one of the adjoint
representation up to homotopy (see \cite{GrMe10a}). The meaning of
this similarity will become clear in \S\ref{tangent_Courant}.
The map $l$ is $ \rho_{\mathsf E}^*\colon T^*M\to\mathsf E$ and 
the form $\omega\in\Omega^3(\mathsf
E,T^*M)$ is given by $\omega(e_1,e_2,e_3)=\langle R_\Delta^{\rm
  bas}(e_1,e_2), e_3\rangle$. Note that it corresponds to the tensor $\Psi$
defined in \cite[Definition 4.1.2]{Li-Bland12} (the right-hand side of
\eqref{def_R_nabla_bas}).  The adjoint split
Lie 2-algebroids are exactly the \emph{split symplectic Lie
  2-algebroids}, and correspond hence to splittings of the tangent
doubles of Courant algebroids \cite{Jotz17c}.

\subsubsection{Split Lie 2-algebroid defined by a 2-representation}\label{semi-direct}
Let $(\partial_B\colon C\to B,\nabla,\nabla,R)$ be a representation up
to homotopy of a Lie algebroid $A$ on $B\oplus C$.  We anchor $A\oplus
C^*$ by $\rho_A\circ\pr_A$ and define $\Delta\colon\Gamma(A\oplus
C^*)\times\Gamma(C\oplus A^*)\to\Gamma(C\oplus A^*)$ by
\begin{equation*}
\Delta_{(a,\gamma)}(c,\alpha)=(\nabla_ac,\ldr{a}\alpha+\langle\nabla^*_\cdot\gamma,c\rangle),
\end{equation*}
$\nabla\colon\Gamma(A\oplus C^*)\times\Gamma(B)\to\Gamma(B)$ by $\nabla_{(a,\gamma)}b=\nabla_ab$.
The vector bundle map $l$ is here
$l=\iota_{C^*}\circ\partial_B^*$, where $\iota_{C^*}\colon C^*\to
A\oplus C^*$ is the canonical inclusion, and the dull bracket that is
dual to $\Delta$ is
given
by \[\lb (a_1,\gamma_1),(a_2,\gamma_2)\rb=([a_1,a_2],\nabla_{a_1}^*\gamma_2-\nabla_{a_2}^*\gamma_1)\]
for $a_1,a_2\in\Gamma(A)$, $\gamma_1,\gamma_2\in\Gamma(C^*)$. The tensor $\omega$ is given
by \[\omega((a_1,\gamma_1), (a_2,\gamma_2), (a_3,\gamma_3))=\langle
R(a_1,a_2), \gamma_3\rangle+\text{c.p.}
\]

Note that if we work with the dual $A$-representation up to
homotopy $(\partial_B^*\colon B^*\to C^*, \nabla^*,
\nabla^*, -R^*)$, then we get the Lie 2-algebroid defined in
\cite[Proposition 3.5]{ShZh17} 
as the semi-direct product of a 2-
representation and a Lie algebroid.  This is then also a special case
of the bicrossproduct of a matched pair of 2-representations (see \S
\ref{bicross}).  Later we will explain why the choice that we make
here is more natural.

\subsection{Morphisms of (split) Lie 2-algebroids}\label{mor_lie_2_alg}\label{morphism_Lie_2}
In this section we quickly discuss morphisms of split Lie
$2$-algebroids, see also \cite{BoPo13}.

 A morphism $\mu\colon (\mathcal M_1,\mathcal Q_1)\to (\mathcal
 M_2,\mathcal Q_2)$ of Lie 2-algebroids is a morphism $\mu\colon
  \mathcal M_1\to \mathcal M_2$ of the underlying $[2]$-manifolds,
  such that
\begin{equation}\label{morphism_of_Q}
\mu^\star\circ\mathcal Q_2=\mathcal Q_1\circ\mu^\star\colon
C^\infty(\mathcal M_2)\to C^\infty(\mathcal M_1).
\end{equation}

Assume that the two $[2]$-manifolds $\mathcal M_1$ and $\mathcal M_2$
are split $[2]$-manifold $\mathcal M_1=Q_1[-1]\oplus B_1^*[-2]$ and
$\mathcal M_2=Q_2[-1]\oplus B_2^*[-2]$.  Then the homological vector
fields $\mathcal Q_1$ and $\mathcal Q_2$ are defined as in
\S\ref{explicit_Q} with two split Lie $2$-algebroids; $(\rho_{1}\colon
Q_1\to TM_1,\partial_{1}\colon Q_1^*\to B_1, \lb\cdot\,,\cdot\rb_1,
\nabla^1, \omega_1)$ and $(\rho_{2}\colon Q_2\to
TM_2, \partial_{2}\colon Q_2^*\to B_2, \lb\cdot\,,\cdot\rb_2,
\nabla^2, \omega_2)$. Further, the morphism $\mu^\star\colon
C^\infty(\mathcal M_2)\to C^\infty(\mathcal M_1)$ over $\mu_0^*\colon
C^\infty(M_2)\to C^\infty(M_1)$ decomposes as $\mu_Q\colon Q_1\to
Q_2$, $\mu_B\colon B_1^*\to B_2^*$ and $\mu_{12}\colon \wedge^2 Q_1\to
B_2^*$, all morphisms over $\mu_0\colon M_1\to M_2$.  We study
\eqref{morphism_of_Q} in these decompositions.
\begin{enumerate}
\item The condition $\mu^\star(\mathcal Q_2(f))=\mathcal
  Q_1(\mu^\star(f))$ for all $f\in C^\infty(M_2)$ is 
  $\mu_Q^\star(\rho_2^*\dr f)=\rho_1^*\dr(\mu_0^*f)$ for all $f\in C^\infty(M_2)$, which is
  equivalent to
\[T_m\mu_0(\rho_1(q_m))=\rho_2(\mu_Q(q_m))
\]
for all $q_m\in Q_1$. In other words
$\mu_Q\colon Q_1\to Q_2$ over $\mu_0\colon M_1\to M_2$ is compatible 
with the anchors $\rho_1\colon Q_1\to TM_1$ and $\rho_2\colon Q_2\to TM_2$.
\item The condition $\mu^\star(\mathcal Q_2(\tau))=\mathcal
  Q_1(\mu^\star(\tau))$ for all $\tau\in \Gamma(Q_2^*)$ reads
\[\mu^\star(\dr_2\tau+\partial_2\tau)=\partial_1(\mu_Q^\star\tau)+\dr_1(\mu_Q^\star\tau)\]
for all $\tau\in \Gamma(Q_2^*)$. The left-hand side is 
\[\underset{\in
  \Omega^2(Q_1)}{\underbrace{\mu_Q^\star(\dr_2\tau)+\mu_{12}^\star(\partial_2\tau)}}+\underset{\in
  \Gamma(B_1)}{\underbrace{\mu_B^\star(\partial_2\tau)}}\]
and the right-hand side is 
\[\partial_1(\mu_Q^\star\tau)+\dr_1(\mu_Q^\star\tau)\in
\Gamma(B_1)\oplus\Omega^2(Q_1).
\]
Hence, $\mu^\star\circ\mathcal Q_2=\mathcal Q_1\circ\mu^\star$ on
degree $1$ functions if and only if
$\mu_Q\circ\partial_1^*=\partial_2^*\circ\mu_B$ and
$\mu_Q^\star(\dr_2\tau)+\mu_{12}^\star(\partial_2\tau)=\dr_1(\mu_Q^\star\tau)$
for all $\tau\in\Gamma(Q_2^*)$.  
\item Finally we find that $\mu^\star(\mathcal Q_2(b))=\mathcal
  Q_1(\mu^\star(b))$ for all $b\in\Gamma(B_2)$ if and only if
\[\mu^\star(\dr_{\nabla^2}b)=\dr_{\nabla^1}(\mu_B^\star(b))+\partial_1(\mu_{12}^\star(b))
\,\in \Omega^1(Q_1,B_1)
\]
for all $b\in\Gamma(B_2)$ and
\[\mu_Q^\star\omega_2=\mu_B\circ \omega_1-\dr_{\mu_0^\star\nabla^2}\mu_{12}\, \in
\Omega^3(Q_1,\mu_0^*B_2^*).
\]

\end{enumerate}

In the equalities above we have used the following constructions.  The
form $\mu^\star(\dr_{\nabla^2}b)\in\Omega^1(Q_1,B_1)$ is defined by
$(\mu^\star(\dr_{\nabla^2}b))(q_m)={\mu_{B}}_m^*(\nabla^2_{\mu_Q(q_m)}b)\in
B_1(m)$ for all $q_m\in Q_1$.  Recall that $\mu_{12}$ can be seen as
an element of $\Omega^2(Q_1,\mu_0^*B_2^*)$.  The tensors
$\mu_Q^\star\omega_{2}\in\Omega^2(Q_1,\mu_0^*B_2^*)$ and $\mu_B\circ
\omega_{1}\in\Omega^2(Q_1,\mu_0^*B_2^*)$ can be defined as follows:
\[(\mu_Q^\star\omega_{2})(q_1(m),q_2(m),q_3(m))=\omega_{2}(\mu_Q(q_1(m)),\mu_Q(q_2(m)),\mu_Q(q_3(m)))
\]
in $B_2^*(\mu_0(m))$, and 
\[(\mu_B\circ \omega_{1})(q_1(m),q_2(m),q_3(m))=\mu_B(\omega_{1})(q_1(m),q_2(m),q_3(m)))\in B_2^*(\mu_0(m))
\]
for all $q_1,q_2,q_3\in\Gamma(Q_1)$. The linear connection
$\mu_Q^\star\nabla^2\colon\Gamma(Q_1)\times\Gamma(\mu_0^*B_2^*)\to\Gamma(\mu_0^*B_2^*)$
is defined by
\[
(\mu_Q^\star\nabla^2)_q(\mu_0^!\beta)(m)={\nabla^2}^*_{\mu_Q(q(m))}\beta\in
B_2^*(\mu_0(m))
\]
for all $q\in\Gamma(Q_1)$ and $\beta\in\Gamma(B_2^*)$.

  We call a triple $(\mu_Q,\mu_B,\mu_{12})$ over $\mu_0$ as above a
  \textbf{morphism of split Lie $2$-algebroids}.
In particular, if $\mathcal M_1=\mathcal M_2$, $\mu_0=\Id_M\colon M\to
M$, $\mu_Q=\Id_Q\colon Q\to Q$ and $\mu_B=\Id_{B^*}\colon B^*\to B^*$,
then $\mu_{12}\in\Omega^2(Q, B^*)$ is just a change
of splitting. The five conditions above simplify to
\begin{enumerate}
\item The dull brackets are related by $\lb q, q'\rb_{2}=\lb q,
  q'\rb_1 +\partial_B^*\mu_{12}(q,q')$.
\item The connections are related by
  $\nabla^2_qb=\nabla^1_qb-\partial_B\langle\mu_{12}(q,\cdot),b\rangle$.
\item The curvature terms are related by
  $\omega_{1}-\omega_{2}=\dr_{1,\nabla^2}\mu_{12}$.
\end{enumerate}
The operator $\dr_{1,\nabla^2}\colon\Omega^\bullet(Q,
B^*)\to\Omega^{\bullet+1}(Q, B^*)$ is defined by the dull bracket
$\lb\cdot\,,\cdot\rb_1$ and the connection ${\nabla^2}^*$.

\section{VB-Courant algebroids and Lie 2-algebroids}\label{sec:VB_cour}
In this section we describe and prove in detail the equivalence
between VB-Courant algebroids and Lie 2-algebroids. In short, a
homological vector field on a $[2]$-manifold defines an anchor and a
Courant bracket on the corresponding metric double vector bundle. This
Courant bracket and this anchor are automatically compatible with the
metric and define so a linear Courant algebroid structure on the
double vector bundle.  Note that a correspondence of Lie 2-algebroids
and VB-Courant algebroids has already been discussed by Li-Bland
\cite{Li-Bland12}.  Our goal is to make this result constructive by
deducing it from the results in \cite{Jotz17a} and presenting it as
the counterpart of the main result in \cite{GrMe10a}, and to
illustrate it with several (partly new) examples.

\subsection{Definition and observations}
We will work with the following definition of a VB-Courant
algebroid, which is due to Li-Bland \cite{Li-Bland12}.
\begin{definition}
A VB-Courant algebroid is a metric double vector bundle 
\[\begin{xy}
\xymatrix{\mathbb{E}\ar[r]^{\pi_B}\ar[d]_{\pi_Q}&B\ar[d]^{q_B}\\
Q\ar[r]_{q_Q}&M
}
\end{xy}
\]
with core $Q^*$ such that $\mathbb E\to B$ is a Courant algebroid and
the following conditions are satisfied.
\begin{enumerate}
\item The anchor map $\Theta\colon \mathbb E\to TB$ is linear. That is, 
\begin{equation}
\begin{xy}
  \xymatrix{\mathbb{E}\ar[rr]^{\pi_B}\ar[dd]_{\pi_Q}&&B\ar[dd]^{q_B}&
    & &TB\ar[rr]^{p_B}\ar[dd]_{Tq_B}&&B\ar[dd]^{q_B}\\
    &C\ar[dr]&   &\ar[r]^\Theta& &&B\ar[dr]&\\
    Q\ar[rr]_{q_Q}&&M&&&TM\ar[rr]_{p_M}&&M }
\end{xy}
\end{equation}
is a morphism of double vector bundles.
\item The Courant bracket is linear. That is 
\begin{align*}
  \left\lb\Gamma^l_B(\mathbb{E}),\Gamma^l_B(\mathbb{E})\right\rb&\subseteq\Gamma^l_B(\mathbb{E}),\quad 
  \left\lb\Gamma^l_B(\mathbb{E}),\Gamma^c_B(\mathbb{E})\right\rb\subseteq\Gamma^c_B(\mathbb{E}),\quad\left\lb\Gamma^c_B(\mathbb{E}),\Gamma^c_B(\mathbb{E})\right\rb=0.
\end{align*}
\end{enumerate}
\end{definition}

We make the following observations. Let $\rho_Q\colon Q\to TM$ be the
side map of the anchor, i.e.~if $\pi_Q(\chi)=q$ for $\chi\in \mathbb
E$, then $Tq_B(\Theta(\chi))=\rho_Q(q)$. In other words, if $\chi
\in\Gamma^l_B(\mathbb E) $ is linear over $q\in\Gamma(Q)$ then
$\Theta(\chi)$ is linear over $\rho_Q(q).$ Let $\partial_B\colon Q^*\to
B$ be the core map defined as follows by the anchor $\Theta$:
\begin{equation}\label{anchor_on_pullbacks}
\Theta(\sigma^\dagger)=(\partial_B\sigma)^\uparrow
\end{equation}
for all $\sigma\in\Gamma(Q^*)$. In the following, we call $\rho_Q$ the
\emph{side-anchor} and $\partial_B$ the \emph{core-anchor}.  The
operator $\mathcal D=\Theta^*\dr\colon C^\infty(B)\to \Gamma_B(\mathbb
E)$ satisfies $\mathcal D(q_B^*f)=(\rho_Q^*\dr f)^\dagger$ for all
$f\in C^\infty(M)$ and Lemma \ref{roytenberg_useful} yields
immediately
\begin{equation}\label{rho_delta}
\partial_B\circ \rho_Q^*=0,\quad  \text{ which is equivalent to } \quad \rho_Q\circ \partial_B^*=0.
\end{equation}
Recall that if $\chi \in\Gamma^l_B(\mathbb E) $ is linear over
$q\in\Gamma(Q)$, then $\langle \chi, \tau^\dagger\rangle=q_B^*\langle
q, \tau\rangle$ for all $\tau\in\Gamma(Q^*)$.

\subsection{The fat Courant algebroid}
Here we denote by $\widehat{\mathbb E}\to M$ the fat bundle, that is
the vector bundle whose sheaf of sections is the sheaf of
$C^\infty(M)$-modules $\Gamma^l_B(\mathbb E)$, the linear sections of
$\mathbb E$ over $B$.  Gracia-Saz and Mehta show in \cite{GrMe10a}
that if $\mathbb E$ is endowed with a linear Lie algebroid structure
over $B$, then $\widehat{\mathbb E}\to M$ inherits a Lie algebroid
structure, which is called the ``fat Lie algebroid''. For
completeness, we describe here quickly the counterpart of this in the
case of a linear Courant algebroid structure on $\mathbb E\to B$.

Note that the restriction of the pairing on $\mathbb E$ to linear
sections of $\mathbb E$ defines a nondegenerate pairing on
$\widehat{\mathbb E}$ with values in $B^*$.  Since the Courant bracket
of linear sections is again linear, we get the following theorem.
\begin{theorem}\label{fat} Let $(\mathbb E, B,Q,M)$ be a VB-Courant algebroid.
  Then $\widehat{\mathbb E}$ is a Courant algebroid with
  pairing in $B^*$.
\end{theorem}
Note that in \cite{JoKi16} we explain how the Courant algebroid with
pairing in $E^*$ that is obtained from the VB-Courant algebroid
$TE\oplus T^*E$, for a vector bundle $E$, is equivalent to the
omni-Lie algebroids described in \cite{ChLi10,ChLiSh11}.

We will come back in Corollary \ref{fat_CA} to the structure found in
Theorem \ref{fat}.  Recall that for
$\phi\in\Gamma(\operatorname{Hom}(B,Q^*))$, the core-linear section
$\widetilde{\phi}$ of $\mathbb E\to B$ is defined by
$\widetilde{\phi}(b_m)=0_{b_m}+_B\overline{\phi(b_m)}$.  Note that
$\widehat{\mathbb E}$ is also naturally paired with $Q^*$: $\langle
\chi(m), \sigma(m)\rangle=\langle \pi_Q(\chi(m)), \sigma(m)\rangle$
for all $\chi\in\Gamma^l_B(\mathbb E)=\Gamma(\widehat{\mathbb E})$ and
$\sigma\in\Gamma(Q^*)$.  This pairing is degenerate since it restricts
to $0$ on $\operatorname{Hom}(B,Q^*)\times_MQ^*$.  The following
proposition can easily be proved.
\begin{proposition}
\begin{enumerate}
\item The map 
$\Delta\colon \Gamma(\widehat{\mathbb E})\times \Gamma(Q^*)\to
\Gamma(Q^*)$ defined by $\left(\Delta_\chi \tau\right)^\dagger=\lb \chi,
\tau^\dagger\rb$
is a flat Dorfman connection, where $\widehat{\mathbb E}$ is endowed
with the anchor $\rho_Q\circ \pi_Q$ and paired with $Q^*$ as above.
\item The map
$\nabla\colon \Gamma(\widehat{\mathbb E})\times \Gamma(B)\to
\Gamma(B)$ 
defined by 
$\Theta(\chi)=\widehat{\nabla_\chi}\in\mx^l(B)$ is a flat connection.
\end{enumerate}
$\Delta$ and $\nabla$ satisfy
\begin{equation*}
\partial_B\circ \Delta=\nabla\circ \partial_B \quad\text{ and }\quad \left\lb \chi, \widetilde{\phi}\right\rb_{\widehat{\mathbb E}}=\widetilde{\Delta_\chi\circ \phi-\phi\circ
  \nabla_\chi}
\end{equation*}
for $\chi\in\Gamma(\widehat{\mathbb E})$ and
$\phi\in\Gamma(\operatorname{Hom}(B,Q^*))$.
\end{proposition}

\begin{proof} The first part of the proposition is easy to prove.
 Choose $\chi\in\Gamma^l_B(\mathbb E)$ and $\tau\in\Gamma(Q^*)$. Then 
\begin{equation*}
  (\partial_B\circ\Delta_\chi\tau)^\uparrow=
\Theta(\Delta_\chi\tau^\dagger)=\Theta\left(\left\lb \chi,\tau^\dagger\right\rb\right)
  =\left[\Theta(\chi), (\partial_B\tau)^\uparrow\right]=(\nabla_\chi(\partial_B\tau))^\uparrow.
\end{equation*}
 The second equation is easy to check by writing $\widetilde{\phi}=\sum_{i=1}^n\ell_{\beta_i}\cdot\tau_i^\dagger$ 
with $\beta_i\in\Gamma(B^*)$ and $\tau_i\in\Gamma(Q^*)$.
\end{proof}

\begin{lemma}\label{formulas}
  For $\phi,\psi\in\Gamma(\operatorname{Hom}(B,Q^*))$ and
  $\tau\in\Gamma(Q^*)$, we have
\begin{enumerate}
\item $\left\lb \tau^\dagger, \widetilde{\phi}\right\rb
  =(\phi(\partial_B \tau))^\dagger= -\left\lb \widetilde{\phi},
    \tau^\dagger\right\rb$ and
\item $\left\lb \widetilde{\phi},
    \widetilde{\psi}\right\rb=\widetilde{\psi\circ\partial_B\circ\phi-\phi\circ\partial_B\circ\psi}
  $.
\end{enumerate}
\end{lemma}

\begin{remark}
  Note that the second equality is the induced Lie algebra bundle
  structure induced on $\operatorname{Hom}(B,Q^*)$ by $\partial_B$.
\end{remark}

\begin{proof}[Proof of Lemma \ref{formulas}] We write $\phi=\sum_{i=1}^n\beta_i\otimes \tau_i$ and
  $\psi=\sum_{j=1}^n\beta'_j\otimes\tau_j$ with
  $\beta_1,\ldots,\beta_n,\beta'_1,\ldots,\beta'_n\in\Gamma(B^*)$ and
  $\tau_1,\ldots,\tau_n\in\Gamma(Q^*)$.  Hence, we have
  $\widetilde{\phi}=\sum_{i=1}^n\ell_{\beta_i}\tau_i^\dagger$ and
  $\widetilde{\psi}=\sum_{j=1}^n\ell_{\beta'_j}\tau_j^\dagger$.  First we
  compute \begin{equation*}
\begin{split}
\left\lb \tau^\dagger,
    \sum_{i=1}^n\ell_{\beta_i}\tau_i^\dagger\right\rb=\sum_{i=1}^n(\partial_B\tau)^\uparrow(\ell_{\beta_i})\tau_i^\dagger
=\sum_{i=1}^nq_B^*\langle \partial_B\tau, \beta_i\rangle
  \tau_i^\dagger=\left(\sum_{i=1}^n \langle \partial_B\tau,
    \beta_i\rangle \tau_i\right)^\dagger
\end{split}
\end{equation*}
and we get (1).  Since $\langle
\tau^\dagger,\widetilde{\phi}\rangle=0$, the second equality follows.
Then we have 
\begin{align*}
  \left\lb \sum_{i=1}^n\ell_{\beta_i}\tau_i^\dagger,
    \sum_{j=1}^n\ell_{\beta'_j}\tau_j^\dagger\right\rb
  &=\sum_{i=1}^n\sum_{j=1}^n\ell_{\beta_i}(\partial_B\tau_i)^\uparrow(\ell_{\beta'_j})\tau_j^\dagger-\ell_{\beta'_j}(\partial_B\tau_j)^\uparrow(\ell_{\beta_i})\tau_i^\dagger\\
  &=\left(\sum_{i=1}^n\sum_{j=1}^n\langle \partial_B\tau_i,
    \beta'_j\rangle\cdot\beta_i\cdot\tau_j-\langle\partial_B\tau_j,\beta_i\rangle\cdot\beta'_j\cdot\tau_i
  \right)^\dagger,
\end{align*}
which leads to (2).
\end{proof}

\subsection{Lagrangian decompositions of VB-Courant algebroids}
In this section, we study in detail the structure of VB-Courant
algebroids, using Lagrangian decompositions of the underlying metric double
vector bundle.  Our goal is the following theorem. Note the similarity
of this result with Theorem \ref{rajan} \cite{GrMe10a} in the VB-algebroid case.
\begin{theorem}\label{main}
  Let $(\mathbb E; Q,B;M)$ be a VB-Courant algebroid and choose a
  Lagrangian splitting $\Sigma\colon Q\times_MB\to \mathbb E$. Then
  there is a split Lie $2$-algebroid structure
  $(\rho_Q,l=\partial_B^*,\lb\cdot\,,\cdot\rb,\nabla, \omega)$ on
  $Q\oplus B^*$ such that
\begin{equation}\label{VB_def1}
\begin{split}
  \Theta(\sigma_Q(q))&=\widehat{\nabla_q}\in\mx(B), \quad  \lb
  \sigma_Q(q), \tau^\dagger\rb=(\Delta_q\tau)^\dagger \text{ and }\\
  \lb \sigma_Q(q_1), \sigma_Q(q_2)\rb &=\sigma_Q\lb q_1,
    q_2\rb-\widetilde{R_\omega(q_1,q_2)},\\
\end{split}
\end{equation}
for all $q,q_1,q_2\in\Gamma(Q)$ and $\tau\in\Gamma(Q^*)$, where
$\Delta\colon\Gamma(Q)\times\Gamma(Q^*)\to\Gamma(Q^*)$ is
the Dorfman connection that is dual to the dull bracket.

Conversely, a Lagrangian splitting $\Sigma\colon Q\times B^*\to
\mathbb E$ of the metric double vector bundle $\mathbb E$ together
with a split Lie $2$-algebroid on $Q\oplus B^*$ define by
\eqref{VB_def1} a linear Courant algebroid structure on $\mathbb E$.
\end{theorem}

First we will construct the objects $\lb\cdot\,,\cdot\rb_\Delta,
\Delta, \nabla, R$ as in the theorem, and then we will prove in the appendix that they
satisfy the axioms of a split Lie 2-algebroid.

\subsubsection{Construction of the split Lie 2-algebroid}\label{construction_of_objects}
First recall that, by definition, the Courant bracket of two linear
sections of $\mathbb E\to B$ is again linear. Hence, we can denote by
$\lb q_1, q_2\rb$ the section of $Q$ such that
\begin{equation}
  \pi_Q\circ \lb \sigma_Q(q_1), \sigma_Q(q_2)\rb =\lb q_1, q_2\rb\circ
  q_B.
\end{equation}

Since for each $q\in\Gamma(Q)$, the anchor $\Theta(\sigma_Q(q))$ is a linear vector field on
  $B$ over $\rho_Q(q) \in\mx(M)$, there
  exists a derivation $\nabla_q\colon\Gamma(B)\to\Gamma(B)$ over
  $\rho_Q(q) $ such that $\Theta(\sigma_Q(q))=\widehat{\nabla_q}\in\mx^l(B)$.  This defines a linear $Q$-connection
  $\nabla\colon\Gamma(Q)\times\Gamma(B)\to\Gamma(B)$. 
For $q\in\Gamma(Q)$ and $\tau\in\Gamma(Q^*)$, the bracket
  $\left\lb \sigma_Q(q), \tau^\dagger\right\rb$ is a core section. It
  is easy to check that the map
  $\Delta\colon\Gamma(Q)\times\Gamma(Q^*)\to\Gamma(Q^*)$ defined by
  \[ \left\lb \sigma_Q(q),
    \tau^\dagger\right\rb=(\Delta_q\tau)^\dagger
\]
is a Dorfman connection.\footnote{ Note that Condition $(C3)$ then implies that 
$\lb \tau^\dagger, \sigma_Q(q)\rb=\left(-\Delta_q\tau+\rho_Q^*\dr\langle \tau,
  q\rangle\right)^\dagger$.}

The difference of the two linear sections $\lb \sigma_Q(q_1),
\sigma_Q(q_2)\rb-\sigma_Q(\lb q_1, q_2\rb_\sigma)$ is again a linear
section, which projects to $0$ under $\pi_Q$.  Hence, there exists a
vector bundle morphism $R(q_1,q_2)\colon B\to Q^*$ such that
$\sigma_Q(\lb q_1, q_2\rb_\sigma)-\lb \sigma_Q(q_1),
\sigma_Q(q_2)\rb=\widetilde{R(q_1,q_2)}$.  This defines a map
$R\colon\Gamma(Q)\times\Gamma(Q)\to\Gamma(\operatorname{Hom}(B,Q^*))$.
We show in the appendix that $R$ defines a $3$-form
$\omega\in\Omega^3(Q,B^*)$ by $R=R_\omega$, that $(l=\partial_B^*,
\lb\cdot\,,\cdot\rb,\nabla, \omega)$ is a split Lie $2$-algebroid, and
that $\lb\cdot\,,\cdot\rb$ is dual to $\Delta$.

Conversely, choose a Lagrangian splitting $\Sigma\colon Q\times_MB$ of
a metric double vector bundle $(\mathbb E, Q; B, M)$ with core $Q^*$
and let $\mathcal S\subseteq\Gamma_B(\mathbb E)$ be the subset
$\{\tau^\dagger\mid \tau\in\Gamma(Q^*)\}\cup \{\sigma_Q(q)\mid
q\in\Gamma(Q)\}\subseteq \Gamma(\mathbb E)$. Choose a split Lie
$2$-algebroid $(l,\lb\cdot\,,\cdot\rb, \nabla, \omega)$ on $Q\oplus
B^*$ with an anchor $\rho_Q$ on $Q$. Consider the Dorfman connection
$\Delta$ that is dual to the dull bracket. Define then a vector bundle
map $\Theta\colon\mathbb E\to TB$ over the identity on $B$ by
$\Theta(\sigma_Q(q))=\widehat{\nabla_q}$ and
$\Theta(\tau^\dagger)=(l^*\tau)^\dagger$ and a bracket
$\lb\cdot\,,\cdot\rb$ on $\mathcal S$ by
\begin{equation*}
\begin{split}
  \lb \sigma_Q(q_1), \sigma_Q(q_2)\rb=\sigma_Q\lb q_1,
    q_2\rb-\widetilde{R_\omega(q_1,q_2)}, \quad \lb \sigma_Q( q),
  \tau^\dagger\rb=(\Delta_q\tau)^\dagger, \\\quad \lb
  \tau^\dagger, \sigma_Q(q)\rb=\left(-\Delta_q\tau+\rho_Q^*\dr\langle
    \tau, q\rangle\right)^\dagger, \quad \lb \tau_1^\dagger,
  \tau_2^\dagger\rb=0.
\end{split}
\end{equation*}
We show in the appendix that this bracket, the pairing and the anchor satisfy the
conditions of Lemma \ref{useful_lemma}, and so that $(\mathbb E, B; Q, M)$ with this structure 
is a VB-Courant algebroid.

\subsubsection{Change of Lagrangian decomposition}
Next we study how the split Lie 2-algebroid $(\partial_B^*\colon B^*
\to Q, \nabla, \lb\cdot\,,\cdot\rb, \omega)$ associated to a Lagrangian decomposition
of a VB-Courant algebroid changes when the Lagrangian decomposition
changes.

The proof of the following proposition is straightforward and left to
the reader. Compare this result with the equations at the end of \S\ref{morphism_Lie_2}, that describe a change of splittings
of a Lie $2$-algebroid.
 
\begin{proposition}\label{change_of_lift}
  Let $\Sigma^1,\Sigma^2\colon B\times_MQ\to\mathbb E$ be two
  Lagrangian splittings and let $\phi\in\Gamma(Q^*\otimes Q^*\otimes
  B^*)$ be the change of lift.
\begin{enumerate}
\item The Dorfman connections are related by
$\Delta_q^2\tau=\Delta_q^1\tau-\phi(q)(\partial_B\tau)$
\item and the dull brackets consequently by $\lb q, q'\rb_{2}=\lb q,
  q'\rb_1 +\partial_B^*\phi(q)^*(q')$.  
\item The connections are related by
  $\nabla^2_q=\nabla^1_q-\partial_B\circ\phi(q)$.
\item The curvature terms are related by
  $\omega_{1}-\omega_{2}=\dr_{{\nabla^2}^*}\phi$, where the
  operator $\dr_{{\nabla^2}^*}$ is defined with the dull bracket
  $\lb\cdot\,,\cdot\rb_1$ on $\Gamma(Q)$. 
\end{enumerate}
\end{proposition}

As an application, we get the following corollary of Theorem
\ref{main} and Theorem \ref{fat}.  Given $\Delta\colon
\Gamma(Q)\times\Gamma(Q^*)\to\Gamma(Q^*)$ and
$\nabla\colon\Gamma(Q)\times\Gamma(B)\to\Gamma(B)$, we define the
derivations $\lozenge\colon
\Gamma(Q)\times\Gamma(\operatorname{Hom}(B,Q^*))\to\Gamma(\operatorname{Hom}(B,Q^*))$
by $(\lozenge_q\phi)(b)=\Delta_q(\phi(b))-\phi(\nabla_qb)$.
\begin{corollary}\label{fat_CA}
  Let $(Q\oplus B^*\to M, \rho_Q,
  \partial_B^*, \lb\cdot\,,\cdot\rb, \nabla, \omega)$ be a split Lie 2-algebroid. Then the vector bundle $\mathsf
  E:=Q\oplus\operatorname{Hom}(B,Q^*)$ is a
  Courant algebroid with pairing in $B^*$ given by $\langle
  (q_1,\phi_1), (q_2,\phi_2)\rangle=\phi_1^*(q_2)+\phi_2^*(q_1)$, with
  the anchor $\tilde\rho\colon \mathsf E\to
  \widehat{\operatorname{Der}(B)}$,
  $\tilde\rho(q,\phi)^*=\nabla_q^*+\phi^*\circ \partial_B^*$ over $\rho(q)$ and
  the bracket given by
\begin{equation*}
\begin{split}
  \lb (q_1,\phi_1), (q_2, \phi_2)\rb=\Bigl(\lb q_1,
  q_2\rb_\Delta+\partial_B(\phi_1^*(q_2)),
  &\lozenge_{q_1}\phi_2-\lozenge_{q_2}\phi_1+\nabla_\cdot^*(\phi_1^*(q_2))
  \\
  +&\phi_2\circ \partial_B\circ\phi_1-\phi_1\circ
  \partial_B\circ\phi_2+R_\omega(q_1,q_2)\Bigr).
\end{split}
\end{equation*}
The map $\mathcal D\colon \Gamma(B^*)\to \Gamma(\mathsf E)$ sends $q $
to $(\partial_B^*q,\nabla^*_\cdot q)$. The bracket does not depend on the
choice of splitting.
\end{corollary}

\subsection{Examples of VB-Courant algebroids and the corresponding split Lie 2-algebroids}
We give here some examples of VB-Courant algebroids, and we compute
the corresponding classes of split Lie $2$-algebroids. We find the
split Lie $2$-algebroids described in Section
\ref{examples_split_lie_2}. In each of the examples below, it is easy
to check that the Courant algebroid structure is linear. Hence, it is
easy to check geometrically that the objects described in
\ref{examples_split_lie_2} are indeed split Lie $2$-algebroids.

\subsubsection{The standard Courant algebroid over a vector bundle}\label{standard_VB_Courant_ex}
We have discussed this example in great detail in \cite{Jotz13a}, but
not in the language of split Lie $2$-algebroids.  Note further that,
in \cite{Jotz13a}, we worked with general, not necessarily Lagrangian,
linear splittings.

Let $q_E\colon E\to M$ be a vector bundle and consider the VB-Courant
algebroid
\begin{equation*}
\begin{xy}
  \xymatrix{
    TE\oplus T^*E\ar[rr]^{\Phi_E:=({q_E}_*, r_E)}\ar[d]_{\pi_E}&& TM\oplus E^*\ar[d]\\
    E\ar[rr]_{q_E}&&M }
\end{xy}
\end{equation*} with base $E$
and side $TM\oplus E^*\to M$, and 
with core $E\oplus T^*M\to M$, or in other words the standard (VB-)Courant
algebroid over a vector bundle $q_E\colon E\to M$. Recall that $TE\oplus T^*E$ has a natural linear metric
(see \cite{Jotz17a}). Linear splittings of
$TE\oplus T^*E$ are in bijection with dull brackets on sections of
$TM\oplus E^*$ \cite{Jotz13a}, and so also with Dorfman connections $\Delta\colon
\Gamma(TM\oplus E^*)\times\Gamma(E\oplus T^*M)\to\Gamma(E\oplus
T^*M)$, and Lagrangian splittings of $TE\oplus T^*E$ are in bijection with
skew-symmetric  dull brackets on sections of
$TM\oplus E^*$ \cite{Jotz17a}.

The anchor $\Theta=\pr_{TE}\colon TE\oplus T^*E\to TE$ restricts to
the map $\partial_{E}=\pr_{E}\colon E\oplus T^*M\to E$ on the cores,
and defines an anchor $\rho_{TM\oplus E^*}=\pr_{TM}\colon TM\oplus
E^*\to TM$ on the side.  In other words, the anchor of
$(e,\theta)^\dagger$ is $e^\uparrow\in \mx^c(E)$ and if
$\widetilde{(X,\epsilon)}$ is a linear section of $TE\oplus T^*E\to E$
over $(X,\epsilon)\in\Gamma(TM\oplus E^*)$, the anchor
$\Theta(\widetilde{(X,\epsilon)})\in \mx^l(E)$ is linear over
$X$. 
Let $\iota_E\colon E\to E\oplus T^*M$ be the canonical inclusion. In
\cite{Jotz13a} we proved that for 
 $q, q_1,q_2\in\Gamma(TM\oplus E^*)$ and $\tau,\tau_1,
  \tau_2\in\Gamma(E\oplus T^*M)$,   
the Courant-Dorfman bracket on sections of $TE\oplus T^*E\to E$ is
given by
\begin{enumerate}
\item $\left\lb \sigma(q),
    \tau^\dagger\right\rb=\left(\Delta_{q}\tau\right)^\dagger$,
 \item $ \left\lb \sigma(q_1), \sigma(q_2)\right\rb =\sigma(\lb q_1,
   q_2\rb_\Delta)-\widetilde{R_\Delta(q_1, q_2)\circ \iota_E}$,
\end{enumerate}
and that the anchor $\rho$ is described by
$\Theta(\sigma(q))=\widehat{\nabla_q^*}\in\mx(E)$,
where $\nabla\colon \Gamma(TM\oplus E^*)\times\Gamma(E)\to \Gamma(E)$
is defined by $\nabla_q=\pr_E\circ\Delta_q\circ\iota_E$ for all $q\in\Gamma(TM\oplus E^*)$.

Hence, if we choose a Lagrangian splitting of $TE\oplus T^*E$, we
find the split Lie 2-algebroid of Example \ref{standard}. 

\subsubsection{VB-Courant algebroid defined by a VB-Lie algebroid}\label{VB_Courant_alg}
More generally, let
\begin{equation*}
\begin{xy}
  \xymatrix{
    D\ar[r]^{\pi_B}\ar[d]_{\pi_A}& B\ar[d]^{q_B}\\
    A\ar[r]_{q_A}&M }
\end{xy}
\end{equation*}
(with core $C$) be endowed with a VB-Lie algebroid structure $(D\to B, A\to
M)$. Then the pair $(D,D\duer B)$ of vector bundles over $B$ is a Lie
bialgebroid, with $D\duer B$ endowed with the trivial Lie algebroid
structure.  We get a linear Courant algebroid $D\oplus_B (D\duer B)$
over $B$ with side $A\oplus C^*$
\begin{equation*}
\begin{xy}
  \xymatrix{
   D\oplus_B (D\duer B)\ar[r]\ar[d]&B\ar[d]\\
   A\oplus C^* \ar[r]&M }
\end{xy}
\end{equation*}
and core $C\oplus A^*$.  We check that the Courant algebroid structure
is linear. Let $\Sigma\colon A\times_M B\to D$ be a linear splitting
of $D$. Recall from Appendix \ref{appendix_dual} that we can define a
linear splitting of $D\duer B$ by $\Sigma^\star\colon B\times_M C^*\to
D\duer B$, $\langle\Sigma^\star(b_m,\gamma_m),
\Sigma(a_m,b_m)\rangle=0$ and $\langle\Sigma^\star(b_m,\gamma_m),
c^\dagger(b_m)\rangle=\langle \gamma_m,c(m)\rangle$ for all $b_m\in
B$, $a_m\in A$, $\gamma_m\in C^*$ and $c\in\Gamma(C)$. The linear
splitting $\tilde\Sigma\colon B\times_M(A\oplus C^*)\to D\oplus_B
(D\duer B)$,
$\tilde\Sigma(b_m,(a_m,\gamma_m))=(\Sigma(a_m,b_m),\Sigma^\star(b_m,\gamma_m))$
is then a Lagrangian splitting.  A computation shows that the Courant
bracket on $\Gamma_B(D\oplus_B(D\duer B))$ is given by
\begin{equation*}
\begin{split}
  &\left\lb \tilde\sigma_{A\oplus C^*}(a_1,\gamma_1),
    \tilde\sigma_{A\oplus
      C^*}(a_2,\gamma_2)\right\rb\\
&=([\sigma_A(a_1),\sigma_A(a_2)],
  \ldr{\sigma_A(a_1)}\sigma^\star_{C^*}(\gamma_2)-\ip{\sigma_A(a_2)}\dr\sigma^\star_{C^*}(\gamma_1))\\
  &=\left(\sigma_A[a_1,a_2]-\widetilde{R(a_1,a_2)},
    \sigma^\star_{C^*}(\nabla_{a_1}^*\gamma_2-\nabla_{a_2}^*\gamma_1)-\widetilde{\langle\gamma_2,
      R(a_1, \cdot)\rangle}
    +\widetilde{\langle\gamma_1, R(a_2, \cdot)\rangle}\right)\\
  &\left\lb \tilde\sigma_{A\oplus C^*}(a,\gamma), (c,\alpha)^\dagger\right\rb
=\left(\nabla_ac^\dagger,(\ldr{a}\alpha +\langle\nabla_\cdot^*\gamma,c\rangle)^\dagger\right)\\
  &\left\lb (c_1,\alpha_1)^\dagger, (c_2,\alpha_2)^\dagger\right\rb=0,
\end{split}
\end{equation*}
and the anchor of $D\oplus_B(D\duer B)$ is defined by
\begin{equation*}
\begin{split}
\Theta(\tilde\sigma_{A\oplus C^*}(a,\gamma))=\Theta(\sigma_{A}(a))=\widehat{\nabla_a}\in \mx^l(B), \quad 
\Theta((c,\alpha)^\dagger)=(\partial_Bc)^\uparrow\in \mx^c(B),
\end{split}
\end{equation*}
where $(\partial_B\colon C\to B,
\nabla\colon\Gamma(A)\times\Gamma(B)\to\Gamma(B),\nabla\colon\Gamma(A)\times\Gamma(C)\to\Gamma(C),R)$
is the 2-representation of $A$ associated to the splitting
$\Sigma\colon A\times_MB\to D$ of the VB-algebroid $(D\to B, A\to M)$.
Hence, we have found the split Lie 2-algebroid described in Example \ref{semi-direct}.

\subsubsection{The tangent Courant algebroid}\label{tangent_Courant}
We consider here a Courant algebroid\linebreak $(\mathsf
E,\rho_{\mathsf E},\lb\cdot\,,\cdot\rb,\langle \cdot\,,\cdot\rangle)$.
In this example, $\mathsf E$ will always be anchored by the Courant
algebroid anchor map $\rho_{\mathsf E}$ and paired with itself by
$\langle \cdot\,,\cdot\rangle$ and $\mathcal
D=\Beta\inv\circ\rho_{\mathsf E}^*\circ\dr\colon C^\infty(M)\to
\Gamma(\mathsf E)$.  Note that $\lb\cdot\,,\cdot\rb$ is not a dull bracket.

We show that, after the choice of a metric connection on $\mathsf E$
and so of a Lagrangian splitting $\Sigma^\nabla\colon
TM\times_M\mathsf E\to T\mathsf E$ (see Examples
\ref{metric_connections} and \ref{tangent_double}), the VB-Courant
algebroid structure on $(T\mathsf E\to TM, \mathsf E\to M)$ is
equivalent to the split Lie 2-algebroid defined by $\nabla$ as in
Example \ref{adjoint}.

\begin{theorem}\label{tangent_courant_double}
  Choose a linear connection $\nabla\colon\mx(M)\times\Gamma(\mathsf
  E)\to\Gamma(\mathsf E)$ that preserves the pairing on $\mathsf E$.
  The Courant algebroid structure on $T\mathsf E\to TM$ can be
  described as follows:
\begin{enumerate}
\item The pairing is given by 
\[\langle e_1^\dagger, e_2^\dagger\rangle=0, 
\quad \langle \sigma_{\mathsf E}^\nabla(e_1), e_2^\dagger\rangle=p_M^*\langle e_1,
e_2\rangle, \quad \text{ and } \langle \sigma_{\mathsf E}^\nabla (e_1),
  \sigma_{\mathsf E}^\nabla (e_2)\rangle=0,\]
\item the anchor is given by $\Theta(\sigma_{\mathsf E}^\nabla
  (e))=\widehat{\nabla^{\rm bas}_e}$ and
  $\Theta(e^\dagger)=(\rho_{\mathsf E}(e))^\uparrow$,
\item and the bracket is given by
\[\left\lb e_1^\dagger, e_2^\dagger\right\rb=0, \qquad 
\left\lb \sigma_{\mathsf E}^\nabla (e_1), e_2^\dagger\right\rb=(\Delta_{e_1}
e_2)^\dagger\] and    \[\left\lb \sigma_{\mathsf E}^\nabla (e_1), \sigma_{\mathsf E}^\nabla (e_2)\right\rb=\sigma_{\mathsf E}^\nabla (\lb e_1,
  e_2\rb_\Delta)-\widetilde{R_\Delta^{\rm bas}(e_1, e_2)}\]
\end{enumerate}
for all $e, e_1,e_2\in\Gamma(\mathsf E)$.
\end{theorem}

\begin{proof} We use the characterisation of the tangent Courant
  algebroid in \cite{BoZa09} (see also \cite{Li-Bland12}): the pairing
has already been discussed in Example \ref{metric_connections}.
It is given by $\langle Te_1, Te_2\rangle=\ell_{\dr\langle e_1,
    e_2\rangle}$ and $\langle Te_1, e_2^\dagger\rangle=p_M^*\langle
  e_1, e_2\rangle$.  The anchor is given by
  $\Theta(Te)=\widehat{\ldr{\rho_{\mathsf E}(e)}}\in\mx(TM)$ and
  $\Theta(e^\dagger)=(\rho_{\mathsf E}(e))^\uparrow\in \mx(TM)$. The
  bracket is given by $\lb Te_1, Te_2\rb =T\lb e_1, e_2\rb$ and $\lb
  Te_1, e_2^\dagger\rb=\lb e_1, e_2\rb^\dagger$ for all $e, e_1,
e_2\in \Gamma(\mathsf E)$.

(1) is easy to check (see also Example \ref{metric_connections} and
\cite{Jotz17a}).  We check (2), i.e.~that the anchor satisfies
$\Theta(\sigma^\nabla_{\mathsf E}(e))=\widehat{\nabla_e^{\rm bas}}$.
We have for $\theta\in\Omega^1(M)$ and $v_m\in TM$: $
\Theta(\sigma^\nabla_{\mathsf E} (e)(v_m))(\ell_\theta)=
\ell_{\ldr{\rho_{\mathsf E}(e)}\theta}(v_m)-\langle \theta_m, \rho_{\mathsf
  E}(\nabla_{v_m}e)\rangle=\ell_{ {\nabla_e^{\rm bas}}^*\theta}(v_m)$
and for $f\in C^\infty(M)$: $\Theta(\sigma^\nabla_{\mathsf E}
(e))(p_M^*f)=p_M^*(\rho_{\mathsf E}(e)f)$.  This proves the equality.

Then we compute the brackets of our linear and core sections.  Choose
sections $\phi,\phi'$ of $\operatorname{Hom}(TM,\mathsf E)$. Then
$\left\lb T e, \widetilde{\phi}\right\rb=\widetilde{\ldr{e}\phi}$,
with $\ldr{e}\phi\in\Gamma(\Hom(TM,\mathsf E))$ defined by
$(\ldr{e}\phi)(X)=\lb e, \phi(X)\rb-\phi([\rho_{\mathsf E}(e),X])$ for
all $X\in\mx(M)$. The equality $\left\lb\widetilde{\phi}, T
  e\right\rb=-\widetilde{\ldr{e}\phi}+\mathcal
D\ell_{\langle\phi(\cdot),e\rangle}$ follows. For
$\theta\in\Omega^1(M)$, we compute $\left\langle \mathcal
  D\ell_\theta,
  e^\dagger\right\rangle=\Theta(e^\dagger)(\ell_\theta)=p_M^*\langle
\rho_{\mathsf E}(e), \theta\rangle$.  Thus, $\mathcal
D\ell_\theta=T(\Beta\inv\rho_{\mathsf E}^*\theta)+\widetilde{\psi}$
for a section $\psi\in\Gamma(\operatorname{Hom}(TM,\mathsf E))$ to be
determined.  Since $\left\langle
  \mathcal D\ell_\theta,
  Te\right\rangle=\Theta(Te)(\ell_\theta)=\ell_{\ldr{\rho_{\mathsf E}(e)}\theta}$, 
the bracket $\langle T(\Beta\inv\rho_{\mathsf E}^*\theta)+\widetilde{\psi},
Te\rangle=\ell_{\dr\langle\theta,\rho_{\mathsf E}(e)\rangle+\langle\psi(\cdot),
  e\rangle}$ must equal $\ell_{\ldr{\rho_{\mathsf E}(e)}\theta}$, and we find
$\langle\psi(\cdot), e\rangle=\ip{\rho_{\mathsf E}(e)}\dr\theta$. Because
$e\in\Gamma(\mathsf E)$ was arbitrary we find 
$\psi(X)=-\Beta\inv\rho_{\mathsf E}^*\ip{X}\dr\theta$ for
$X\in\mx(M)$.  We get in particular
\[\left\lb\widetilde{\phi}, T
  e\right\rb=-\widetilde{\ldr{e}\phi}+T(\Beta\inv\rho_{\mathsf
  E}^*\langle\phi(\cdot),e\rangle)-\widetilde{\Beta\inv\rho_{\mathsf
    E}^*\ip{X}\dr\langle\phi(\cdot),e\rangle}.
\]
The formula $\left\lb \widetilde{\phi},
  \widetilde{\phi'}\right\rb=\widetilde{\phi'\circ\rho_{\mathsf
    E}\circ\phi-\phi\circ\rho_{\mathsf E}\circ\phi'}$
can easily be checked, as well as $\left\lb \widetilde{\phi},
  e^\dagger\right\rb=-\left\lb e^\dagger,
  \widetilde{\phi}\right\rb=-(\phi(\rho_{\mathsf E}(e)))^\dagger$.
Using this, we find now easily that
\begin{equation*}
\begin{split}&\left\lb \sigma^\nabla_{\mathsf E} (e_1), \sigma^\nabla_{\mathsf E} (e_2)\right\rb=\left\lb
    Te_1-\widetilde{\nabla_\cdot e_1}, Te_2-\widetilde{\nabla_\cdot
      e_2}\right\rb\\
  &=T\left\lb e_1, e_2\right\rb-\widetilde{\ldr{e_1}\nabla_\cdot e_2}
  +\widetilde{\ldr{e_2}\nabla_\cdot e_1}-T(\Beta\inv\rho_{\mathsf
    E}^*\langle\nabla_\cdot e_1,e_2\rangle)\\
  &\quad+\widetilde{\Beta\inv\rho_{\mathsf E}^*\dr\langle \nabla_\cdot
    e_1,e_2\rangle}+\widetilde{\nabla_{\rho_{\mathsf E}(\nabla_\cdot e_1)} e_2}-\widetilde{\nabla_{\rho_{\mathsf E}(\nabla_\cdot e_2)} e_1 }\\
  &=T\left\lb e_1,
    e_2\right\rb_\Delta-\widetilde{\ldr{e_1}\nabla_\cdot
    e_2}+\widetilde{\ldr{e_2}\nabla_\cdot
    e_1}+\widetilde{\Beta\inv\rho_{\mathsf E}^*\dr\langle \nabla_\cdot
    e_1,e_2\rangle}\\
  &\quad +\widetilde{\nabla_{\rho_{\mathsf E}(\nabla_\cdot e_1)}
    e_2}-\widetilde{\nabla_{\rho_{\mathsf E}(\nabla_\cdot e_2)} e_1 }.
\end{split}
\end{equation*}
Since for all $X\in \mx(M)$, we have 
\begin{equation*}
\begin{split}
  &-(\ldr{e_1}\nabla_\cdot e_2)(X)+(\ldr{e_2}\nabla_\cdot
  e_1)(X)+\beta\inv\rho_{\mathsf E}^*\ip{X}\dr\langle \nabla_\cdot
  e_1,e_2\rangle\\
  &=-\lb e_1, \nabla_X e_2\rb +\nabla_{[\rho_{\mathsf E}(e_1),
    X]}e_2+\lb e_2, \nabla_X e_1\rb-\nabla_{[\rho_{\mathsf E}(e_2),
    X]}e_1 +\beta\inv\rho_{\mathsf E}^*\ip{X}\dr\langle \nabla_\cdot
  e_1,e_2\rangle\\
  &=-\lb e_1, \nabla_X e_2\rb +\nabla_{[\rho_{\mathsf E}(e_1),
    X]}e_2-\lb \nabla_X e_1, e_2\rb-\nabla_{[\rho_{\mathsf E}(e_2),
    X]}e_1 +\beta\inv\rho_{\mathsf E}^*\ldr{X}\langle \nabla_\cdot
  e_1,e_2\rangle,
\end{split}
\end{equation*}
we find that $\left\lb \sigma^\nabla_{\mathsf E} (e_1),
  \sigma^\nabla_{\mathsf E} (e_2)\right\rb=T\lb e_1, e_2\rb_\Delta-\widetilde{R_\Delta^{\rm
  bas}(e_1,e_2)}$.  Finally we compute $ \left\lb \sigma^\nabla_{\mathsf E} (e_1),
  e_2^\dagger\right\rb=\left\lb Te_1-\widetilde{\nabla_\cdot e_1},
  e_2^\dagger\right\rb=\lb e_1, e_2\rb^\dagger+\nabla_{\rho_{\mathsf
    E}(e_2)}e_1^\dagger=\Delta_{e_1}e_2^\dagger$.
\end{proof}

\subsection{Categorical equivalence of Lie 2-algebroids and VB-Courant
  algebroids}
In this section we quickly describe morphisms of VB-Courant
algebroids. Then we find an equivalence between the category of
VB-Courant algebroids and the category of Lie $2$-algebroids. Note
that a bijection between VB-Courant algebroids and Lie $2$-algebroids
was already described in \cite{Li-Bland12}.

\subsubsection{Morphisms of VB-Courant algebroids}\label{mor_VB_Courant}
Recall from \S\ref{recall_n} that a morphism $\Omega\colon \mathbb
E_1\to\mathbb E_2$ of metric double vector bundles is an
isotropic relation $\Omega\subseteq \overline{\mathbb
  E_1}\times\mathbb E_2$ that is the dual of a morphism $\mathbb
(E_1)\duer{Q_1}\to \mathbb (E_2)\duer{Q_2}$. Assume that $\mathbb E_1$ and
$\mathbb E_2$ have linear Courant algebroid structures. Then $\Omega$
is a morphism of VB-Courant algebroid if it is a Dirac structure (with
support) in $\overline{\mathbb E_1}\times\mathbb E_2$.

\medskip

Choose two Lagrangian splittings $\Sigma^1\colon Q_1\times
B_1\to\mathbb E_1$ and $\Sigma^2\colon Q_2\times B_2\to\mathbb
E_2$. Then there exists four
structure maps $\omega_0\colon M_1\to M_2$, $\omega_Q\colon Q_1\to
Q_2$, $\omega_B\colon B_1^*\to B_2^*$ and
$\omega_{12}\in\Omega^2(Q_1,\omega_0^*B_2^*)$ that define completely
$\Omega$. More precisely, $\Omega$ is spanned over
$\operatorname{Graph}(\omega_Q\colon Q_1\to Q_2)$ by sections
$\tilde b\colon \operatorname{Graph}(\omega_Q)\to \Omega$,
\[\tilde b(q_m,\omega_Q(q_m))= \left(\sigma_{B_1}(\omega_B^\star
  b)(q_m)+\widetilde{\omega_{12}^\star(b)}(q_m),
  \sigma_{B_2}(b)(\omega_Q(q_m))\right)
\]
for all $b\in\Gamma_{M_2}(B_2)$,
and $\tau^\times\colon \operatorname{Graph}(\omega_Q)\to \Omega$,
\[
\tau^\times(q_m,\omega_Q(q_m))=
\left((\omega_Q^\star\tau)^\dagger(q_m),
  \tau^\dagger(\omega_Q(q_m))\right)
\]
for all $\tau\in\Gamma_{M_2}(Q_2^*)$.
Note that $\Omega$ projects under $\pi_{B_1}\times\pi_{B_2}$ to
$R_{\omega_B^*}\subseteq B_1\times B_2$.  If $q\in \Gamma(Q_1)$
then $\omega_Q^!q\in\Gamma_{M_1}(\omega_0^*Q_2)$ can be written as
$\sum_{i}f_i\omega_0^!q_i$ with $f_i\in C^\infty(M_1)$ and
$q_i\in\Gamma_{M_2}(Q_2)$. The pair $\left(\sigma_{B_1}(\omega_B^\star
  b)(q_m)+\widetilde{\omega_{12}^\star(b)}(q_m),
  \sigma_{B_2}(b)(\omega_Q(q_m))\right)$ can be written as
\[\left(\left(\sigma_{Q_1}(q)+\langle\omega_{12}(q,\cdot), b(\omega_0(m))\rangle^\dagger\right)(\omega_B^\star
b(m)), \sum_if_i(m)\sigma_{Q_2}(q_i)(b(\omega_0(m)))\right).\] Hence,
$\Omega$ is spanned by the restrictions to $R_{\omega_B^*}$ of sections
\begin{equation}\label{section_morphism_1}
\left(\sigma_{Q_1}(q)\circ\pr_1+\langle\omega_{12}(q,\cdot),\pr_2\rangle^\dagger\circ\pr_1,
  \sum_i(f_i\circ q_{B_1}\circ \pr_1)\cdot(\sigma_{Q_2}(q_i)\circ\pr_2)\right)
\end{equation}
for all $q\in \Gamma_{M_1}(Q_1)$ and
\begin{equation}\label{section_morphism_2}
\left( (\omega_Q^\star\tau)^\dagger\circ\pr_1,\tau^\dagger\circ\pr_2\right)
\end{equation}
for all $\tau\in\Gamma(Q_2^*)$. 

\medskip

Checking all the conditions in Lemma \ref{useful_for_dirac_w_support}
on the two types of sections \eqref{section_morphism_1} and
\eqref{section_morphism_2} yields that $\Omega\to R_{\omega_B^*}$ is a
Dirac structure with support if and only if
\begin{enumerate}
\item $\omega_Q\colon Q_1\to Q_2$ over $\omega_0\colon M_1\to M_2$ is compatible 
with the anchors $\rho_1\colon Q_1\to TM_1$ and $\rho_2\colon Q_2\to TM_2$:
\[T_m\omega_0(\rho_1(q_m))=\rho_2(\omega_Q(q_m))
\]
for all $q_m\in Q_1$,
\item $\partial_1\circ\omega_Q^\star=\omega_B^\star\circ\partial_2$ 
as maps from $\Gamma(Q_2^*)$ to $\Gamma(B_1)$, or equivalently
$\omega_Q\circ\partial_1^*=\partial_2^*\circ\omega_B$,
\item $\omega_Q$ preserves the dull brackets up to
  $\partial_2^*\omega_{12}$:
  i.e.~$\omega_Q^\star(\dr_2\tau)+\omega_{12}^\star(\partial_2\tau)=\dr_1(\omega_Q^\star\tau)$
  for all $\tau\in\Gamma(Q_2^*)$.
\item $\omega_B$ and $\omega_Q$ intertwines the connections 
$\nabla^1$ and $\nabla^2$ up to $\partial_1\circ\omega_{12}$:
\[\omega_B^\star((\omega_Q^\star\nabla^2)_{q}b)
=\nabla^1_{q}(\omega_B^\star(b))-\partial_1\circ\langle\omega_{12}(q,\cdot), b\rangle
\in\Gamma(B_1)
\]
for all $q_m\in Q_1$ and $b\in\Gamma(B^2)$, and 
\item
  $\omega_Q^\star\omega_{R_2}-\omega_B\circ\omega_{R_1}=-\dr_{(\omega_Q^\star\nabla^2)}\omega_{12}\in\Omega^3(Q_1,\omega_0^*B_2^*)$.
\end{enumerate}
We find so that $\Omega$ is a morphism of VB-Courant algebroid if and
only if it induces a morphism of split Lie $2$-algebroids after any
choice of Lagrangian decompositions of $\mathbb E_1$ and $\mathbb
E_2$.

\subsubsection{Equivalence of categories}
The functors \S\ref{recall_n} between the
category of metric double vector bundles and the category of
$[2]$-manifolds restrict to functors between the category of 
VB-Courant algebroids and the category of Lie $[2]$-algebroids.
\begin{theorem}\label{main2}
  The category of Lie $2$-algebroids is equivalent to the
  category of VB-Courant algebroids.
\end{theorem}

\begin{proof}
  Let $(\mathcal M, \mathcal Q)$ be a Lie $2$-algebroid and consider
  the double vector bundle $\mathbb E_{\mathcal M}$ corresponding to
  $\mathcal M$.  Choose a splitting $\mathcal M\simeq Q[-1]\oplus
  B^*[-2]$ of $\mathcal M$ and consider the corresponding Lagrangian
  splitting $\Sigma$ of $\mathbb E_{\mathcal M}$.

  By Theorem \ref{main}, the split Lie
  $2$-algebroid $(Q[-1]\oplus B^*[-2],\mathcal Q)$
   defines a VB-Courant algebroid structure on the
  decomposition of $\mathbb E_{\mathcal M}$ and so by isomorphism on
  $\mathbb E_{\mathcal M}$. Further, by Proposition
  \ref{change_of_lift}, the Courant
  algebroid structure on $\mathbb E_{\mathcal M}$ does not depend on
  the choice of splitting of $\mathcal M$, since a different choice of
  splitting will induce a change of Lagrangian splitting of $\mathbb
  E_{\mathcal M}$.  This shows that the functor $\mathcal G$ restricts
  to a functor $\mathcal G_{Q}$ from the category of Lie
  $2$-algebroids to the category of VB-Courant algebroids.

 Sections \ref{mor_lie_2_alg} and \ref{mor_VB_Courant}
show that morphisms of split Lie $2$-algebroids are 
sent by $\mathcal G$ to morphisms of decomposed VB-Courant algebroids.

The functor $\mathcal F$ restricts in a similar manner to a functor
$\mathcal F_{\rm VBC}$ from the category of VB-Courant algebroids to the
category of Lie $2$-algebroids. The natural transformations found
in the proof of Theorem \ref{main_crucial} restrict to natural
transformations $\mathcal F_{\rm VBC}\mathcal G_{Q}\simeq \Id$ and
$\mathcal G_{Q}\mathcal F_{\rm VBC}\simeq \Id$.
\end{proof}

\begin{remark}
  Note that we use splittings and decompositions in order to obtain
  this equivalence of categories, which does not involve splittings
  and decompositions.

  First, while the linear metric of the linear VB-Courant algebroid
  are at the heart of the equivalence of the underlying (metric)
  double vector bundle $(\mathbb E; B,Q;M)$ with the underlying
  $[2]$-manifold of the corresponding Lie $2$-algebroid, the linear
  Courant bracket and the linear anchor do not translate to very
  elegant structures on the linear isotropic sections of $\mathbb E\to
  Q$ and on its core sections. Only in a decomposition, the
  ingredients of the linear bracket and anchor are recognised in a
  straightforward manner as the ingredients of a split lie
  $2$-algebroid.

  Since our main goal was to show that, as decomposed VB-algebroids
  are the same as $2$-representations \cite{GrMe10a}, decomposed
  VB-Courant algebroids are the same as split Lie $2$-algebroids, it
  is natural for us to establish here our equivalence in
  decompositions and splittings. The main work for the `splitting
  free' version of the equivalence was done in \cite{Jotz17a}. Another
  approach can of course be found in \cite{Li-Bland12}, but the
  equivalence there is not really constructive, in the sense that it
  is difficult to even recognise the graded functions on the
  underlying $[2]$-manifold as sections of the metric double vector
  bundle. To our understanding, the equivalence of $[2]$-manifolds
  with metric double vector bundles is not easy to recognise in the
  proof of \cite{Li-Bland12}. 

  Further, our main application in Section \ref{double} is a statement
  about a certain class of \emph{decomposed} VB-Courant algebroids
  versus \emph{split} Lie $2$-algebroids. Similarly, in a sequel of
  this paper \cite{Jotz17c}, we work exclusively with decomposed or
  split objects to express Li-Bland's definition of an LA-Courant
  algebroid \cite{Li-Bland12} in a decomposition. This yields a new
  definition that involves the `matched pair' of a split Lie
  $2$-algebroid with a self-dual $2$-representation. This new approach
  is far more useful for concrete computations, since there is no need
  anymore to consider the tangent triple vector bundle of $\mathbb E$
  (see \cite{Li-Bland12}).
\end{remark}

\section{VB-bialgebroids and bicrossproducts of matched pairs of
  2-representations}\label{double}
In this section we show that the bicrossproduct of a matched pair of
2-representations is a split Lie 2-algebroid and we geometrically explain
this result.

\subsection{The bicrossproduct of a matched pair of $2$-representations}\label{bicross}
We construct a split Lie 2-algebroid $(A\oplus B)\oplus C$ induced by
a matched pair of 2-representations as in Definition
\ref{matched_pair_2_rep}.  The vector bundle $A\oplus B\to M$ is
anchored by $\rho_A\circ\pr_A+\rho_B\circ\pr_B$ and paired with $A^*\oplus B^*$ as
follows:
\[\langle (a,b),(\alpha,\beta)\rangle=\alpha(a)-\beta(b)
\]
for all $a\in\Gamma(A)$, $b\in\Gamma(B)$, $\alpha\in\Gamma(A^*)$ and
$\beta\in\Gamma(B^*)$. The morphism $A^*\oplus B^*\to C^*$ is
$\partial_A^*\circ\pr_{A^*}+\partial_B^*\circ\pr_{B^*}$. The $A\oplus B$-Dorfman connection
on $A^*\oplus B^*$ is defined by
\[\Delta_{(a,b)}(\alpha, \beta)=(\nabla^*_b\alpha+\ldr{a}\alpha-\langle\nabla_\cdot
b,\beta\rangle, \nabla^*_a\beta+\ldr{b}\beta-\langle \nabla_\cdot
a,\alpha\rangle).
\]
The dual dull bracket on $\Gamma(A\oplus B)$ is
\begin{equation}\label{bicross_bracket}
\lb (a,b), (a',b')\rb=([a,a']+\nabla_ba'-\nabla_{b'}a,
[b,b']+\nabla_ab'-\nabla_{a'}b).
\end{equation}
The $A\oplus B$-connection on $C^*$ is simply given by
$\nabla_{(a,b)}^*\gamma=\nabla_a^*\gamma+\nabla_b^*\gamma$ and the
dual connection is $\nabla\colon \Gamma(A\oplus
B)\times\Gamma(C)\to\Gamma(C)$, 
\begin{equation}\label{nabla_bicross_C}
\nabla_{(a,b)}c=\nabla_ac+\nabla_bc.
\end{equation}
Finally, the form $\omega=\omega\in\Omega^3(A\oplus B,C)$ is given by
\begin{equation}\label{omega_matched}
\begin{split}
  \omega_R((a_1,b_1),(a_2,b_2),(a_3,b_3))&=R(a_1,a_2)b_3+R(a_2,a_3)b_1+R(a_3,a_1)b_2\\
  &\qquad -R(b_1,b_2)a_3-R(b_2,b_3)a_1-R(b_3,b_1)a_2.
\end{split}
\end{equation}
The
vector bundle $(A\oplus B)\oplus C\to M$ with the anchor
$\rho_A\circ\pr_A+\rho_B\circ\pr_B\colon A\oplus B\to TM$,
$l=(-\partial_A;\partial_B)\colon C\to A\oplus B$, $\omega_R$
and the skew-symmetric dull bracket \eqref{bicross_bracket} define a
split Lie 2-algebroid.  Moreover, we prove the following theorem:
\begin{theorem}\label{double_2_rep}
  The bicrossproduct of a matched pair of 2-representations is a split Lie
  2-algebroid with the structure given above.
  Conversely if $(A\oplus B)\oplus C$ has a split Lie 2-algebroid
  structure such that 
\begin{enumerate}
\item $\lb (a_1,0),(a_2,0)\rb=([a_1,a_2],0)$
with a section $[a_1,a_2]\in\Gamma(A)$ for all $a_1,a_2\in\Gamma(A)$
and in the same manner $\lb (0,b_1),(0,b_2)\rb=(0, [b_1,b_2])$
with a section $[b_1,b_2]\in\Gamma(B)$ for all $b_1,b_2\in\Gamma(B)$, and
\item $\omega((a_1,0), (a_2,0), (a_3,0))=0$ and $\omega((0, b_1), (0,b_2), (0,b_3))=0$
for all $a_1,a_2,a_3$ in $\Gamma(A)$ and $b_1,b_2,b_3$ in $\Gamma(B)$,
\end{enumerate} then $A$ and $B$ are Lie subalgebroids of $(A\oplus
B)\oplus C$ and $(A\oplus B)\oplus C$ is the bicrossproduct of a
matched pair of $2$-representations of $A$ on $B\oplus C$ and of $B$
on $A\oplus C$. The 2-representation of $A$ is given by
\begin{equation}\label{2_rep_A}
\begin{split}
\partial_B(c)&=\pr_B(l(c)),\,\, 
\nabla_ab=\pr_B\lb (a,0), (0,b)\rb,\,\,
\nabla_ac=\nabla_{(a,0)}c,\\
 R_{AB}(a_1,a_2)b&=\omega(a_1,a_2,b)
\end{split}
\end{equation}
and the $B$-representation is given by 
\begin{equation}
\begin{split}\label{2_rep_B}
\partial_A(c)&=-\pr_A(l(c)),\,\,
\nabla_ba=\pr_A\lb(0,b),(a,0)\rb,\,\,\nabla_bc=\nabla_{(0,b)}c,\\
 R_{BA}(b_1,b_2)a&=-\omega(b_1,b_2,a).
\end{split}
\end{equation}
\end{theorem}

\begin{proof}
  Assume first that $(A\oplus B)\oplus C$ is a split Lie 2-algebroid
  with (1) and (2). The bracket
  $[\cdot\,,\cdot]\colon\Gamma(A)\times\Gamma(A)\to\Gamma(A)$ defined
  by $\lb (a_1,0), (a_2,0)\rb=([a_1,a_2],0)$ is obviously
  skew-symmetric and $\R$-bilinear. Define an anchor $\rho_A$ on $A$
  by $\rho_A(a)=\rho_{A\oplus B}(a,0)$. Then we get immediately
  \[([a_1,fa_2],0)=\lb (a_1,0),
  f(a_2,0)\rb=f([a_1,a_2],0)+\rho_{A\oplus B}(a_1,0)(f)(a_2,0),\]
  which shows that $[a_1,fa_2]=f[a_1,a_2]+\rho_A(a_1)(f)a_2$ for all
  $a_1,a_2\in\Gamma(A)$. Further, we find
  $\Jac_{[\cdot\,,\cdot]}(a_1,a_2,a_3)=\pr_A(\Jac_{\lb\cdot\,,\cdot\rb}((a_1,0),(a_2,0),(a_3,0)))=-(\pr_A\circ
  l\circ \omega )((a_1,0),(a_2,0),(a_3,0)))=0$ since $\omega$ vanishes on
  sections of $A$.  Hence $A$ is a wide subalgebroid of the split Lie
  $2$-algebroid. In a similar manner, we find a Lie algebroid
  structure on $B$. Next we prove that \eqref{2_rep_A} defines a
  $2$-representation of $A$.
Using (ii) in Definition \ref{split_Lie2} we find for $a\in\Gamma(A)$ and $c\in\Gamma(C)$ that 
\begin{equation*}
\begin{split}
\partial_B(\nabla_ac)&=(\pr_B\circ l)(\nabla_{(a,0)}c)\overset{\rm (ii)}{=}\pr_B\lb (a,0),l(c)\rb\\
&=\pr_B\lb (a,0),(0,\pr_B(l(c)))\rb=\nabla_a(\partial_Bc).
 \end{split}
\end{equation*}
In the third equation we have used Condition (1) and in the last
equation the definitions of $\partial_B$ and
$\nabla_a\colon\Gamma(B)\to\Gamma(B)$.
In the following, we will write for simplicity $a$ for $(a,0)\in\Gamma(A\oplus B)$, etc.
We get easily
\begin{equation*}
\begin{split}
R_{AB}(a_1,a_2)\partial_Bc=\omega(a_1,a_2,\pr_B(l(c)) )=\omega(a_1,a_2,l(c))\overset{\rm (iv)}{=}R_\nabla(a_1,a_2)c
 \end{split}
\end{equation*}
and 
\[\partial_BR_{AB}(a_1,a_2)b=(\pr_B\circ l\circ \omega)(a_1,a_2,b)\overset{\rm
  (iii)}{=}-\pr_B(\Jac_{\lb\cdot\,,\cdot\rb}(a_1,a_2,b))
\]
for all $a_1,a_2\in\Gamma(A)$, $b\in\Gamma(B)$ and $c\in\Gamma(C)$.
By Condition (1) and the definition of $\nabla_a\colon\Gamma(B)\to\Gamma(B)$,
we find \begin{equation*}
\begin{split}
R_\nabla(a_1,a_2)b&=\pr_B\lb a_1,\lb a_2,b\rb\rb-\pr_B\lb a_2\lb a_1, b\rb\rb
-\pr_B\lb\lb a_1,a_2\rb, b\rb\\
&=-\pr_B(\Jac_{\lb\cdot\,,\cdot\rb}(a_1,a_2,b)).
 \end{split}
\end{equation*}
Hence, $\partial_BR_{AB}(a_1,a_2)b=R_\nabla(a_1,a_2)b$. Finally, an easy
computation along the same lines shows that
\begin{equation}\label{dl3_van1}
\langle (\dr_{\nabla^{\operatorname{Hom}}}R_{AB})(a_1,a_2,a_3), b\rangle=(\dr_{\nabla}\omega)(a_1,a_2,a_3,b)
\end{equation}
for $a_1,a_2,a_3\in\Gamma(A)$ and $b\in\Gamma(B)$.  Since
$\dr_{\nabla}\omega=0$, we find
$\dr_{\nabla^{\operatorname{Hom}}}R_{AB}=0$. In a similar manner, we
prove that \eqref{2_rep_B} defines a $2$-representation of $B$.
Further, by construction of the $2$-representations, the split Lie
$2$-algebroid structure on $(A\oplus B)\oplus C)$ must be defined as
in \eqref{bicross_bracket}, \eqref{nabla_bicross_C} and
\eqref{omega_matched}, with the anchor
$\rho_A\circ\pr_A+\rho_B\circ\pr_B$ and
$l=(-\partial_A,\partial_B)$. Hence, to conclude the proof, it only
remains to check that the split Lie $2$-algebroid conditions for these
objects are equivalent to the seven conditions in Definition
\ref{matched_pair_2_rep} for the two $2$-representations.

First, we find immediately that (M1) is equivalent to (i).
Then we find by construction
\[[a,\partial_Ac]+\nabla_{\partial_Bc}a=-[a,\pr_A(l(c))]+\nabla_{\pr_B(l(c))}a=\pr_A\lb l(c),a \rb=-\pr_A \lb a, l(c)\rb.
\]
Hence, we find (M2) if and only if 
$\pr_A \lb a, l(c)\rb=\pr_A\circ l(\nabla_ac)$. But since 
\begin{equation*}
\begin{split}
\lb a, lc\rb&=(\pr_A \lb a, l(c)\rb, \nabla_a\pr_Bl(c))=(\pr_A \lb a, l(c)\rb, \nabla_a\partial_B(c))\\
&=(\pr_A \lb a, l(c)\rb,\partial_B\nabla_ac)=(\pr_A \lb a, l(c)\rb,\pr_B(l(\nabla_ac))),
\end{split}
\end{equation*}
we have $\pr_A \lb a, l(c)\rb=\pr_A\circ l(\nabla_ac)$ if and only
if $\lb a, lc\rb=l(\nabla_ac)$. Hence (M2) is satisfied if and
only if $\lb a, l(c)\rb=l(\nabla_ac)$ for all $a\in\Gamma(A)$ and
$c\in\Gamma(C)$.  In a similar manner, we find that (M3) is equivalent
to $\lb b, lc\rb=l(\nabla_bc)$ for all $b\in\Gamma(B)$ and
$c\in\Gamma(C)$. This shows that (M2) and (M3) together are equivalent to (ii).

Next, a simple computation shows that (M4) is equivalent to
$R_\nabla(b,a)c=\omega(b,a,l(c))$. Since
$R_\nabla(a,a')c=R_{AB}(a,a')\partial_Bc=\omega(a,a',\pr_B(l(c)))=\omega(a,a',l(c))$
and $R_\nabla(b,b')c=\omega(b,b',l(c))$, we get that (M4) is equivalent
to (iv).

Two straightwforward computations show that (M5) is equivalent to
\[\pr_A(\Jac_{\lb\cdot\,,\cdot\rb}(a_1,a_2,b))=-\pr_A(l\omega(a_1,a_2,b))\]
and that (M6) is equivalent to
\[\pr_B(\Jac_{\lb\cdot\,,\cdot\rb}(b_1,b_2,a))=-\pr_B(l\omega(b_1,b_2,a)).\]
But since $\pr_B(\Jac_{\lb\cdot\,,\cdot\rb}(a_1,a_2,b))=-R_\nabla(a_1,a_2)b$
by construction and $R_\nabla(a_1,a_2)b=\partial_BR_{AB}(a_1,a_2)b=\pr_B(l\omega(a_1,a_2,b))$,
we find \[\pr_B(\Jac_{\lb\cdot\,,\cdot\rb}(a_1,a_2,b))=-\pr_B(l\omega(a_1,a_2,b)),\] and in a similar 
manner \[\pr_A(\Jac_{\lb\cdot\,,\cdot\rb}(b_1,b_2,a))=-\pr_A(l\omega(b_1,b_2,a)).\]
Since $\Jac_{\lb\cdot\,,\cdot\rb}(a_1,a_2,a_3)=0$, $\Jac_{\lb\cdot\,,\cdot\rb}(b_1,b_2,b_3)=0$, and 
$\omega$ vanishes on sections of $A$, and respectively on sections of $B$,
we conclude that (M5) and (M6) together are equivalent to (iii).

Finally, a slightly longer, but still straightforward computation
shows that
\[(\dr_{\nabla^B}R_{AB})(b_1,b_2)(a_1,a_2)-(\dr_{\nabla^A}R_{BA})(a_1,a_2)(b_1,b_2)=(\dr_\nabla \omega)(a_1,a_2,b_1,b_2)
\]
for all $a_1,a_2\in\Gamma(A)$ and $b_1,b_2\in\Gamma(B)$. This,
\eqref{dl3_van1}, the corresponding identity for $R_{BA}$, and the
vanishing of $\omega$ on sections of $A$, and respectively on sections of
$B$, show that (M7) is equivalent to (v).
\end{proof}

If $C=0$, then $R_{AB}=0$, $R_{BA}=0$, $\partial_A=0$ and
$\partial_B=0$ and the matched pair of 2-representations is just a
matched pair of Lie algebroids.  The double is then concentrated in
degree $0$, with $\omega=0$ and $l_2$ is the bicrossproduct Lie algebroid
structure on $A\oplus B$ with anchor $\rho_A+\rho_B$
\cite{Lu97,Mokri97}. Hence, in that case the split Lie 2-algebroid is
just the bicrossproduct of a matched pair of representations and the
dual (flat) Dorfman connection is the corresponding Lie
derivative. The Lie $2$-algebroid is in that case a genuine Lie
$1$-algebroid.

In a current work in progress we define the notion of matched pair of
higher representations up to homotopy, and we show that the induced
doubles are split Lie n-algebroids.

In the case where $B$ has a trivial Lie algebroid structure and acts
trivially up to homotopy on $\partial_A=0\colon C\to A$, the double is
the semi-direct product Lie 2-algebroid found in \cite[Proposition
3.5]{ShZh17} (see \S\ref{semi-direct}).

\subsection{VB-bialgebroids and double Lie
  algebroids}\label{VB_courant_from_double}
Consider a double vector bundle $(D;A,B;M)$ with core $C$ and a VB-Lie
algebroid structure on each of its sides.  Recall from \S
\ref{matched_pair_2_rep_sec} that $(D;A,B,M)$ is a double Lie
algebroid if and only if, for any linear splitting of $D$, the two
induced $2$-representations (denoted as in
\S\ref{matched_pair_2_rep_sec}) form a matched pair \cite{GrJoMaMe17}. By definition of
a double Lie algebroid, $(D\duer A, D\duer B)$ is then a Lie
bialgebroid over $C^*$ \cite{Mackenzie11}, and so the double vector
bundle
\begin{equation*}
\begin{xy}
  \xymatrix{D\duer A\oplus D\duer B\ar[r]\ar[d]&C^*\ar[d]\\
    A\oplus B\ar[r]&M }
\end{xy}
\end{equation*}
with core $B^*\oplus A^*$ has the structure of a VB-Courant algebroid
with base $C^*$ and side $A\oplus B$.  Note that we call the pair
$(D\duer A, D\duer B)$ a \textbf{VB-bialgebroid over $C^*$}.
Conversely, a VB-Courant algebroid $(\mathbb E; Q,B;M)$ with two
transverse VB-Dirac structures $(D_1;Q_1,B;M)$ and $(D_2;Q_2,B;M)$
defines a VB-bialgebroid $(D_1,D_2)$ over $B$. It is not difficult to
see that a VB-bialgebroid\footnote{$D_A$ has necessarily core $B^*$
  and $D_B$ has core $A^*$.} $(D_A\to X, A\to M)$, $(D_B\to X, B\to
M)$ is equivalent to a double Lie algebroid structure on $((D_A)\duer A;
B, A;M)\simeq ((D_B)\duer B; B, A;M)$ with core $X^*$.

Consider again a double Lie algebroid $(D;A,B;M)$, together with a
linear splitting $\Sigma\colon A\times_MB\to
D$. Then \[\tilde\Sigma\colon (A\oplus B)\times_M C^*\to D\duer
A\oplus D\duer B\] defined by
$\tilde\Sigma((a(m),b(m)),\gamma_m)=(\sigma_A^\star(a)(\gamma_m),
\sigma_B^\star(b)(\gamma_m))$, where $\sigma_A^\star\colon
\Gamma(A)\to \Gamma_{C^*}^l(D\duer A)$ and $\sigma_B^\star\colon
\Gamma(B)\to \Gamma_{C^*}^l(D\duer B)$ are defined as in Appendix
\ref{appendix_dual}, is a linear Lagrangian splitting of $D\duer
A\oplus D\duer B$ (see \eqref{equal_fat}).  Recall from
\S\ref{matched_pair_2_rep_sec} that the splitting $\Sigma^\star\colon
A\times_{M}C^*\to D\duer A$ of the VB-algebroid $(D\duer A\to C^*,
A\to M)$ corresponds to the $2$-representation
$({\nabla^C}^*,{\nabla^B}^*,-R^*) $ of $A$ on the complex
$\partial_B^*\colon B^*\to C^*$. In the same manner, the splitting
$\Sigma^\star\colon B\times_{M}C^*\to D\duer B$ of the VB-algebroid
$(D\duer B\to C^*, B\to M)$ corresponds to the $2$-representation
$({\nabla^C}^*,{\nabla^A}^*,-R^*) $ of $B$ on the complex
$\partial_A^*\colon A^*\to C^*$.

We check that the split Lie 2-algebroid corresponding to the linear
splitting $\tilde\Sigma$ of $D\duer A\oplus D\duer B$ is the
bicrossproduct of the matched pair of 2-representations. The
equalities in \eqref{equal_fat} imply that we have to consider
$A\oplus B$ as paired with $A^*\oplus B^*$ in the non standard way:
\[\langle (a,b),(\alpha,\beta)\rangle=\alpha(a)-\beta(b)
\]
for all $a\in\Gamma(A)$, $b\in\Gamma(B)$, $\alpha\in\Gamma(A^*)$ and
$\beta\in\Gamma(B^*)$.  The anchor of
$\tilde\sigma(a,b)=(\sigma^\star(a),\sigma^\star(b))$ is
$\widehat{\nabla_a^*}+\widehat{\nabla_b^*}\in\mx^l(C^*)$, and the
anchor of
$(\alpha,\beta)^\dagger=(\beta^\dagger,\alpha^\dagger)\in\Gamma_{C^*}^c(D\duer
A\oplus D\duer B)$ is
$(\partial_B^*\beta+\partial_A^*\alpha)^\uparrow\in\mx^c(C^*)$.  The
Courant bracket $\left\lb (\sigma^\star_A(a),\sigma_B^\star(b)),
  (\beta^\dagger,\alpha^\dagger)\right\rb $ is
\[\left([\sigma_A^\star(a),\beta^\dagger]+\ldr{\sigma_B^\star(b)}\beta^\dagger-\ip{\alpha^\dagger}\dr_{D\duer B}\sigma_A^\star(a),
  [\sigma_B^\star(b),\alpha^\dagger]+\ldr{\sigma_A^\star(a)}\alpha^\dagger-\ip{\beta^\dagger}\dr_{D\duer
    A}\sigma_B^\star(b) \right),
\]
where $\dr_{D\duer A}\colon\Gamma_{C^*}(\bigwedge^\bullet D\duer B)\to
\Gamma_{C^*}(\bigwedge^{\bullet+1} D\duer B)$ is defined as usual by
the Lie algebroid $D\duer A$, and similarly for $D\duer B$ (bear in
mind that some non standard signs arise from the signs in
\eqref{equal_fat}).  The derivation $\ldr{}\colon \Gamma(D\duer
A)\times \Gamma(D\duer B)\to \Gamma(D\duer B)$ is described by
\begin{equation*}
\begin{split}
\ldr{\beta^\dagger}\alpha^\dagger&=0,\quad 
\ldr{\beta^\dagger}\sigma_B^\star(b)=-\langle b,
\nabla^*_\cdot\beta\rangle^\dagger,\quad 
\ldr{\sigma_A^\star(a)}\alpha^\dagger=\ldr{a}\alpha^\dagger,\\
&\qquad 
\ldr{\sigma_A^\star(a)}\sigma_B^\star(b)=\sigma_B^\star(\nabla_ab)+\widetilde{R(a,\cdot)b},
\end{split}
\end{equation*}
in \cite[Lemma 4.8]{GrJoMaMe17}. Similar formulae hold for 
$\ldr{}\colon \Gamma(D\duer B)\times \Gamma(D\duer A)\to \Gamma(D\duer A)$.
We get
\[\left\lb (\sigma^\star_A(a),\sigma_B^\star(b)), (\beta^\dagger,\alpha^\dagger)\right\rb
=\left((\nabla_a^*\beta+\ldr{b}\beta-\langle\nabla_\cdot
  a,\alpha\rangle)^\dagger,
  (\nabla_b^*\alpha+\ldr{a}\alpha-\langle\nabla_\cdot
  b,\beta\rangle)^\dagger \right).
\]
In the same manner, we get
\begin{equation*}
\begin{split}
  &\left\lb (\sigma^\star_A(a_1),\sigma_B^\star(b_1)),
    (\sigma^\star_A(a_2),\sigma_B^\star(b_2))\right\rb\\
&  =\left(\sigma_A^\star([a,a']+\nabla_ba'-\nabla_{b'}a),
    \sigma_B^\star([b,b']+\nabla_ab'-\nabla_{a'}b)\right)\\
  &\qquad+\Bigl(-\widetilde{R(a_1,a_2)}+\widetilde{R(b_1,\cdot)a_2}-\widetilde{R(b_2,\cdot)a_1},
  -\widetilde{R(b_1,b_2)}+\widetilde{R(a_1,\cdot)b_2}-\widetilde{R(a_2,\cdot)b_1}\Bigr).
\end{split}
\end{equation*}

Hence we have the following result. Recall that we have found above
that double Lie algebroids are equivalent to VB-Courant algebroids
with two transverse VB-Dirac structures.
\begin{theorem}\label{bij_da_matched_lie_2}
  The correspondence established in Theorem \ref{main}, between
  decomposed VB-Courant algebroids and split Lie $2$-algebroids,
  restricts to a correspondence between decomposed double Lie
  algebroids and split Lie 2-algebroids that are the bicrossproducts
  of matched pairs of 2-representations.

  In other words, decomposed VB-bialgebroids are equivalent to matched
  pairs of 2-representations.
\end{theorem}

Recall that if the vector bundle $C$ is trivial, the matched pair of
$2$-representations is just a matched pair of the Lie algebroids $A$
and $B$. The corresponding double Lie algebroid is the decomposed
double Lie algebroid $(A\times_MB,A,B,M)$ found in \cite{Mackenzie11}.
The corresponding VB-Courant algebroid is 
\begin{equation*}
\begin{xy}
  \xymatrix{A\times_MB^*\oplus A^*\times_MB\ar[r]\ar[d]&0\ar[d]\\
    A\oplus B\ar[r]&M }
\end{xy}
\end{equation*}
with core $B^*\oplus A^*$. In that case there is a natural Lagrangian
splitting and the corresponding Lie $2$-algebroid is just the
bicrossproduct Lie algebroid structure defined on $A\oplus B$ by the
matched pair, see also the end of \S\ref{double}. This shows that the
two notions of double of a matched pair of Lie algebroids; the
bicrossproduct Lie algebroid in \cite{Mokri97} and the double Lie
algebroid in \cite{Mackenzie11} are just the $\N$-geometric and the
classical descriptions of the same object, and special cases of
Theorem \ref{bij_da_matched_lie_2}.

\subsection{Example: the two `doubles' of a Lie bialgebroid}
Recall that a Lie bialgebroid $(A,A^*)$ is a pair of Lie algebroids
$(A\to M, \rho, [\cdot\,,\cdot])$ and $(A^*\to M, \rho_\star,
[\cdot\,,\cdot]_\star)$ in duality such that $A\oplus A^*\to M$ with
the anchor $\rho+\rho_\star$, the pairing
\[\langle (a_1,\alpha_1), (a_2,
\alpha_2)\rangle =\alpha_1(a_2)+\alpha_2(a_1)
\]
 and the bracket
\[\lb (a_1,\alpha_1), (a_2,
\alpha_2)\rb=\left([a_1,a_2]+\ldr{\alpha_1}a_2-\ip{\alpha_2}\dr_{A^*}a_1,
[\alpha_1,\alpha_2]_\star+\ldr{a_1}\alpha_2-\ip{a_2}\dr_{A}\alpha_1\right)
\]
is a Courant algebroid. Lie bialgebroids were originally defined in a
different manner \cite{MaXu94}, and the definition above is at the
origin of the abstract definition of Courant algebroids
\cite{LiWeXu97}. This Courant algebroid is sometimes called the
bicrossproduct of the Lie bialgebroid, or the double of the Lie
bialgebroid.

Mackenzie came up in \cite{Mackenzie11} with an alternative notion of
double of a Lie bialgebroid. Given a Lie bialgebroid as above, the double vector bundle 
\begin{equation*}
\begin{xy}
  \xymatrix{T^*A\simeq T^*A^*\ar[r]\ar[d]&A^*\ar[d]\\
    A\ar[r]&M }
\end{xy}
\end{equation*}
is a double Lie algebroid with the following structures. The Lie
algebroid structure on $A$ defines a linear Poisson structure on
$A^*$, and so a linear Lie algebroid structure on $T^*A^*\to A^*$. In
the same manner, the Lie algebroid structure on $A^*$ defines a linear
Poisson structure on $A$, and so a linear Lie algebroid structure on
$T^*A\to A$ (see \cite{GrJoMaMe17} for more details and for the
matched pairs of $2$-representations associated to a choice of linear
splitting). The VB-Courant algebroid defined by this double Lie algebroid is 
$(T^*A)\duer A\oplus (T^*A)\duer{A^*}$
which is isomorphic to 
\begin{equation*}
\begin{xy}
  \xymatrix{TA\oplus TA^*\ar[r]\ar[d]&TM\ar[d]\\
    A\oplus A^*\ar[r]&M }
\end{xy}.
\end{equation*}
Computations reveal that the Courant algebroid structure is just the
tangent of the Courant algebroid structure on $A\oplus A^*$, and so
that the two notions of double of a Lie bialgebroid can be understood
as an algebraic and a geometric interpretation of the same object.

 \appendix

\section{Proof of Theorem \ref{main}}\label{appendix_proof_of_main} 
Let $(\mathbb E;Q,B;M)$ be a VB-Courant algebroid and choose a
Lagrangian splitting $\Sigma\colon Q\times_MB$. We prove here that the
obtained split linear Courant algebroid is equivalent to a split Lie
$2$-algebroid.  Recall the construction of the objects
$\partial_B,\Delta,\nabla,\lb\cdot\,,\cdot\rb_\sigma, R$ in
\S\ref{construction_of_objects}, and recall that $\mathcal
S\subseteq\Gamma_B(\mathbb E)$ is the subset $\{\tau^\dagger\mid
\tau\in\Gamma(Q^*)\}\cup \{\sigma_Q(q)\mid q\in\Gamma(Q)\}\subseteq
\Gamma_B(\mathbb E)$.  Recall also that the tangent double $(TB\to
B;TM\to M)$ has a VB-Lie algebroid structure, which is described in
Example \ref{tangent_double}. We begin by giving two useful lemmas.
\begin{lemma}\label{formula_for_Dl}
For $\beta\in\Gamma(B^*)$, we have 
\[\mathcal
D(\ell_\beta)=\sigma_Q(\partial_B^*\beta)+\widetilde{\nabla_\cdot^*\beta},
\]
where $\nabla_\cdot^*\beta$ is seen as
follows as a
section of $\Gamma(\operatorname{Hom}(B,Q^*))$: 
$\left(\nabla_\cdot^*\beta\right)(b)=
\langle \nabla_\cdot^*\beta, b\rangle
\in\Gamma(Q^*)$ for all $b\in\Gamma(B)$.
\end{lemma}
\begin{proof}
  For $\beta\in\Gamma(B^*)$, the section $\dr\ell_\beta$ is a linear
  section of $T^*B\to B$. Since the anchor $\Theta$ is linear, the
  section $\mathcal D\ell_\beta=\Theta^*\dr\ell_\beta$ is linear. Since for
  any $\tau\in\Gamma(Q^*)$,
\[\langle \mathcal D(\ell_\beta),
\tau^\dagger\rangle =
\Theta(\tau^\dagger)(\ell_\beta)=q_B^*\langle \partial_B\tau, \beta\rangle,
\]
we find that $\mathcal
D(\ell_\beta)-\sigma_Q(\partial_B^*\beta)\in\Gamma(\ker\pi_Q)$.
Hence, $\mathcal D(\ell_\beta)-\sigma_Q(\partial_B^*\beta)$ is a
core-linear section of $\mathbb \E\to B$ and there exists a section
$\phi$ of $\operatorname{Hom}(B,Q^*)$ such that $\mathcal
D(\ell_\beta)-\sigma_Q(\partial_B^*\beta)=\widetilde{\phi}$.  We
have \[\ell_{\langle\phi, q\rangle}= \langle \widetilde{\phi},
\sigma_Q(q)\rangle=\langle \mathcal
D(\ell_\beta)-\sigma_Q(\partial_B^*\beta), \sigma_Q(q)\rangle
=\Theta(\sigma_Q(q))(\ell_\beta)=\ell_{\nabla_q^*\beta}
\]
and so $\phi(b)=\langle \nabla_\cdot^*\beta, b\rangle\in\Gamma(Q^*)$ for all $b\in\Gamma(B)$.
\end{proof}

For each
$q\in\Gamma(Q)$, $\nabla_q$ and $\Delta_q$ define as follows a derivation
$\lozenge_q$ of $\Gamma(\operatorname{Hom}(B,Q^*))$: for $\phi\in
\Gamma(\operatorname{Hom}(B,Q^*))$ and $b\in\Gamma(B)$
\begin{equation*}
(\lozenge_q\phi)(b)=\Delta_q(\phi(b))-\phi(\nabla_qb).
\end{equation*}

\begin{lemma}
  For $q\in\Gamma(Q)$ and $\phi\in \Gamma(\operatorname{Hom}(B,Q^*))$,
  we have
$\left\lb \sigma_Q(q), \widetilde{\phi}\right\rb= \widetilde{\lozenge_q\phi}$.
\end{lemma}

\begin{proof}
  The proof is an easy computation as in the proof of Lemma \ref{formulas}.
\end{proof}

\medskip

Now we can express all the conditions of Lemma \ref{useful_lemma} in
terms of the objects
$\partial_B,\Delta,\nabla,\lb\cdot\,,\cdot\rb_\sigma, R$ found in
\S\ref{construction_of_objects}.

\begin{proposition}\label{anchor}
The anchor satisfies $\Theta\circ\Theta^*=0$ if and only if 
 $\rho_Q\circ\partial_B^*=0$ and 
  $\nabla_{\partial_B^*\beta_1}^*\beta_2+\nabla_{\partial_B^*\beta_2}^*\beta_1=0$
  for all $\beta_1,\beta_2\in\Gamma(B^*)$.
\end{proposition}

\begin{proof}
  The composition $\Theta\circ \Theta^*$ vanishes if and only if
  $(\Theta\circ \Theta^*)\dr F=0$ for all linear and pullback
  functions $F\in C^\infty(B)$. For $f\in
  C^\infty(M)$,
  $\Theta(\Theta^*\dr(q_B^*f))=((\partial_B\circ\rho_Q^*)\dr
  f)^\uparrow$. For $\beta\in\Gamma(B^*)$, we find using Lemma \ref{formula_for_Dl}
  $\Theta(\Theta^*\dr\ell_\beta)=\Theta(\mathcal
  D\ell_\beta)=\Theta(\sigma_Q(\partial_B^*\beta)+\widetilde{\nabla_\cdot^*\beta})
  =\widehat{\nabla_{\partial_B^*\beta}}+\widetilde{\partial_B\circ
    \langle\nabla^*_\cdot\beta, \cdot\rangle}$.  Here,
  $\partial_B\circ \langle\nabla^*_\cdot\beta, \cdot\rangle$ is as
  follows a morphism $B\to B$;
  $b\mapsto \partial_B(\langle\nabla^*_\cdot\beta, b\rangle)$.  On a
  linear function $\ell_{\beta'}$, $\beta'\in\Gamma(B^*)$,
  $\Theta(\Theta^*\dr\ell_\beta)(\ell_{\beta'})
  =\ell_{\nabla_{\partial_B^*\beta}^*\beta'}+\ell_{\nabla_{\partial_B^*\beta'}^*\beta}$. On
  a pullback $q_B^*f$, $f\in C^\infty(M)$, this is
  $q_B^*(\ldr{(\rho_Q\circ\partial_B^*)(\beta)}f)$. 
\end{proof}

\begin{proposition}\label{comp}
The compatibility of $\Theta$ with the Courant algebroid
bracket $\lb\cdot\,,\cdot\rb$ 
is equivalent to
\begin{enumerate}
 \item $\partial_B\circ R(q_1,q_2)=R_\nabla(q_1,q_2)$,
 \item $\rho_Q\circ
   \lb\cdot\,,\cdot\rb_\sigma=[\cdot\,,\cdot]\circ(\rho_Q, \rho_Q)$,
   or $\Delta_q(\rho_Q^*\dr f)=\rho_Q^*\dr(\rho_Q(q)(f))$ for all
   $q\in\Gamma(Q)$ and $f\in C^\infty(M)$, and
\item $\partial_B\circ \Delta= \nabla\circ \partial_B$.
\end{enumerate}
\end{proposition}

\begin{proof}
  We have $\Theta \left\lb \sigma_Q(q_1),
    \sigma_Q(q_2)\right\rb=\left[\Theta(\sigma_Q(q_1)),
    \Theta(\sigma_Q(q_2))\right] =\left[\widehat{\nabla_{q_1}},
    \widehat{\nabla_{q_2}}\right]$ and $$\Theta\left(\sigma_Q(\lb
    q_1,
    q_2\rb_\sigma)-\widetilde{R(q_1,q_2)}\right)=\widehat{\nabla_{\lb
      q_1, q_2\rb_\sigma}}-\widetilde{\partial_B\circ R(q_1,q_2)}.$$
  Applying both derivations to a pullback function $q_B^*f$ for $f\in
  C^\infty(M)$ yields
\[\left[\widehat{\nabla_{q_1}}, \widehat{\nabla_{q_2}}\right] (q_B^*f)=q_B^*([\rho_Q(q_1), \rho_Q(q_2)]f).\]
and
\[ \left(\widehat{\nabla_{\lb q_1,
      q_2\rb_\sigma}}-\widetilde{\partial_B\circ
    R(q_1,q_2)}\right)(q_B^*f)=q_B^*(\rho_Q\lb q_1,
q_2\rb_\sigma(f))\] Applying both vector fields to a linear function
$\ell_\beta\in C^\infty(B)$, $\beta\in\Gamma(B^*)$, we get
\[\left[\widehat{\nabla_{q_1}}, \widehat{\nabla_{q_2}}\right]
(\ell_\beta)=\ell_{\nabla_{q_1}^*\nabla_{q_2}^*\beta-\nabla_{q_2}^*\nabla_{q_1}^*\beta}
\]
and 
\[ \left(\widehat{\nabla_{\lb q_1,
      q_2\rb_\sigma}}-\widetilde{\partial_B\circ
    R(q_1,q_2)}\right)(\ell_\beta)=\ell_{\nabla^*_{\lb q_1,
    q_2\rb_\sigma}\beta-R(q_1,q_2)^*\partial_B^*\beta}.\] Since
$R_{\nabla^*}(q_1,q_2)=-(R_\nabla(q_1,q_2))^*$, we find that
\[\Theta \left\lb \sigma_Q(q_1), \sigma_Q( q_2)\right\rb=[\Theta(\sigma_Q(q_1)), \Theta(\sigma_Q(q_2))]\]
for all $q_1,q_2\in\Gamma(Q)$ if and only if (1) and (2) are
satisfied.

In the same manner we compute for $q\in\Gamma(Q)$ and
$\tau\in\Gamma(Q^*)$:
\[  \Theta\left(\left\lb \sigma_Q(q),
      \tau^\dagger\right\rb\right)=(\partial_B\Delta_q\tau)^\uparrow\]
and 
\[\left[ \Theta(\sigma_Q(q)),
  \Theta(\tau^\dagger)\right]=\left[\widehat{\nabla_q},
  (\partial_B\tau)^\uparrow\right]
=\left(\nabla_q(\partial_B\tau)\right)^\uparrow.
\]
Hence, $\Theta\left(\left\lb \sigma_Q(q),
    \tau^\dagger\right\rb\right)=\left[ \Theta(\sigma_Q(q)),
  \Theta(\tau^\dagger)\right]$ if and only if
$\partial_B(\Delta_q\tau)= \nabla_q(\partial_B\tau)$.
\end{proof}

\begin{proposition}\label{CA3}
  The condition (3) of Lemma \ref{useful_lemma} is equivalent to
  $R(q_1,q_2)=-R(q_2,q_1)$
  and 
 $\lb q_1, q_2\rb_\sigma+\lb q_2, q_1\rb_\sigma=0$
for $q_1,q_2\in\Gamma(Q)$.
\end{proposition}

\begin{proof}
Choose $q_1,q_2$ in $\Gamma(Q)$.
Then we have 
\begin{equation*}\label{equation}
  \lb \sigma_Q(q_1), \sigma_Q(q_2)\rb+\lb
  \sigma_Q(q_2),\sigma_Q(q_1)\rb=\sigma_Q(\lb q_1,q_2\rb_\sigma+\lb q_2,
  q_1\rb_\sigma)
  -\widetilde{R(q_1,q_2)}-\widetilde{R(q_2,q_1)}.
\end{equation*}
By the choice of the splitting, we have $\mathcal D\langle
\sigma_Q(q_1), \sigma_Q(q_2)\rangle=\mathcal D(0)=0$.  Hence, (3) of
Lemma \ref{useful_lemma} is true on horizontal lifts of sections of
$Q$ if and only if $R(q_1,q_2)=-R(q_2,q_1)$ and $\lb q_1,
q_2\rb_\sigma+\lb q_2, q_1\rb_\sigma=0$ for all $q_1,q_2\in\Gamma(Q)$.
Further, we have $\lb \sigma_Q(q),
\tau^\dagger\rb=(\Delta_q\tau)^\dagger$ and $\lb \tau^\dagger,
\sigma_Q(q)\rb=(-\Delta_q\tau+\rho_Q^*\dr\langle \tau,
q\rangle)^\dagger$ by definition. On core sections (3) is trivially
satisfied since both the pairing and the bracket of two core sections
vanish.
\end{proof}

\begin{proposition}\label{CA2}
The derivation formula  (2) in Lemma  \ref{useful_lemma} 
is equivalent to 
\begin{enumerate}
\item $\Delta$ is dual to $\lb\cdot\,,\cdot\rb_\sigma$, that is $\lb\cdot\,,\cdot\rb_\sigma=\lb\cdot\,,\cdot\rb_\Delta$,
\item $\lb q_1, q_2\rb_\sigma+\lb q_2, q_1\rb_\sigma=0$
  for all $q_1, q_2\in\Gamma(Q)$ and 
\item $R(q_1,q_2)^*q_3=-R(q_1,q_3)^*q_2$
for all $q_1,q_2,q_3\in\Gamma(Q)$.
\end{enumerate}
\end{proposition}

\begin{proof}
  We compute (CA2) for linear and core sections.  First of all, the
  equations \[\Theta(\tau_1^\dagger)\langle \tau_2^\dagger,
  \tau_3^\dagger\rangle =\langle \lb \tau_1^\dagger,
  \tau_2^\dagger\rb, \tau_3^\dagger\rangle +\langle
  \tau_2^\dagger, \lb
  \tau_1^\dagger,\tau_3^\dagger\rb\rangle,\]
\[\Theta(\tau_1^\dagger)\langle \tau_2^\dagger,
  \sigma_Q(q)\rangle =\langle \lb \tau_1^\dagger,
  \tau_2^\dagger\rb, \sigma_Q(q)\rangle +\langle 
  \tau_2^\dagger, \lb
  \tau_1^\dagger,\sigma_Q(q)\rb\rangle\]
and 
\[\Theta(\sigma_Q(q))\langle \tau_1^\dagger, \tau_2^\dagger\rangle
=\langle \lb \sigma_Q(q), \tau_1^\dagger\rb, \tau_2^\dagger\rangle
+\langle \tau_1^\dagger, \lb \sigma_Q(q), \tau_2^\dagger\rb\rangle\]
are trivially satisfied for all
$\tau_1,\tau_2,\tau_3\in\Gamma(Q^*)$ and $q\in\Gamma(Q)$.
Next we have for $q_1,q_2\in\Gamma(Q)$ and 
  $\tau\in\Gamma(Q^*)$:
\begin{align*}
  &\Theta(\sigma_Q(q_1))\langle \sigma_Q(q_2), \tau^\dagger\rangle -\langle \lb
  \sigma_Q(q_1), \sigma_Q(q_2)\rb, \tau^\dagger\rangle-\langle \sigma_Q(q_2), \lb
  \sigma_Q(q_1),
  \tau^\dagger\rb\rangle \\
  =\,\,&\widehat{ \nabla_{q_1}}(q_B^*\langle q_2,\tau\rangle)
  -q_B^*\langle \lb q_1, q_2\rb_\sigma, \tau\rangle-q_B^*\langle q_2,
  \Delta_{q_1}\tau\rangle\\
  =\,\,&q_B^*\Bigl(\rho_Q(q_1)\langle q_2,\tau\rangle -\langle \lb
  q_1, q_2\rb_\sigma, \tau\rangle-\langle q_2,
  \Delta_{q_1}\tau\rangle\Bigr)
\end{align*}
Hence
$\Theta(\sigma_Q(q_1))\langle \sigma_Q(q_2), \tau^\dagger\rangle =\langle
\lb \sigma_Q(q_1), \sigma_Q(q_2)\rb, \tau^\dagger\rangle+\langle \sigma_Q(q_2), \lb
\sigma_Q(q_1), \tau^\dagger\rb\rangle$
for all $q_1,q_2\in\Gamma(Q)$ and $\tau\in\Gamma(Q^*)$ if and only
if $\Delta$ and $\lb\cdot\,,\cdot\rb_\sigma$ are dual to each other.
Using this, we compute
\begin{align*}
  &\Theta(\tau^\dagger)\langle \sigma_Q(q_1), \sigma_Q(q_2)\rangle -\langle \lb
  \tau^\dagger, \sigma_Q(q_1)\rb, \sigma_Q(q_2)\rangle-\langle \sigma_Q(q_1), \lb
  \tau^\dagger, \sigma_Q(q_2)
  \rb\rangle \\
  =\,\,&0-\langle -(\Delta_{q_1}\tau)^\dagger+(\rho_Q^*\dr\langle
  q_1, \tau\rangle)^\dagger, \sigma_Q(q_2)\rangle -\langle \sigma_Q(q_1),
  -(\Delta_{q_2}\tau)^\dagger+(\rho_Q^*\dr\langle q_2,
  \tau\rangle)^\dagger\rangle\\
  =\,\,&-q_B^*\langle \lb q_1, q_2\rb_\sigma+ \lb q_2,
  q_1\rb_\sigma,\tau\rangle.
\end{align*}

Finally we have $\Theta(\sigma_Q(q_1))\langle \sigma_Q(q_2), \sigma_Q(q_3)\rangle =0$ for all
$q_1,q_2,q_3\in\Gamma(Q)$, and \linebreak $\langle \lb \sigma_Q(q_1), \sigma_Q(q_2)\rb,
\sigma_Q(q_3)\rangle=\ell_{-R(q_1,q_2)^*q_3}$.  This shows that
\[\Theta(\sigma_Q(q_1))\langle \sigma_Q(q_2), \sigma_Q(q_3)\rangle=\langle \lb \sigma_Q(q_1),
\sigma_Q(q_2)\rb, \sigma_Q(q_3)\rangle+\langle \sigma_Q(q_2), \lb \sigma_Q(q_1), \sigma_Q(q_3)\rb\rangle
\]
if and only if $0=-R(q_1,q_2)^*q_3-R(q_1,q_3)^*q_2$.
\end{proof}

\begin{proposition}\label{Jacobi}
  Assume that $\Delta$ and $\lb\cdot\,,\cdot\rb_\sigma$ are dual to each
  other. The Jacobi identity in Leibniz form for sections in $\mathcal
  S$ is equivalent to
\begin{enumerate}
\item $R(q_1,q_2)\circ \partial_B=R_\Delta(q_1,q_2)$ and 
\item 
\begin{equation*}
\begin{split}
  &R(q_1, \lb q_2,q_3\rb_\Delta)-R(q_2, \lb q_1, q_3\rb_\Delta)-R(\lb q_1,q_2\rb_\Delta),q_3)\\
  &+\lozenge_{q_1}(R(q_2,q_3))-\lozenge_{q_2}(R(q_1,q_3))+\lozenge_{q_3}(R(q_1,q_2))=\,\nabla_\cdot^*\left(R(q_1,q_2)^*q_3\right)
\end{split}
\end{equation*}
\end{enumerate}
for all $q_1,q_2,q_3\in\Gamma(Q)$.
\end{proposition}
If $R$ is skew-symmetric as in (1) of Proposition \ref{CA3}, then the second
equation is $\dr_{\nabla^*}\omega=0$ for $\omega\in\Omega^3(Q,B^*)$ defined by 
$\omega(q_1,q_2,q_3)=R(q_1,q_2)^*q_3$.

\begin{proof}
  The Jacobi identity is trivially satisfied on core sections since
  the bracket of two core sections is $0$.  Similarly, for
  $\tau_1,\tau_2\in\Gamma(Q^*)$ and $q\in\Gamma(Q)$, we find $\lb
  \sigma_Q(q),\lb \tau_1^\dagger, \tau_2^\dagger\rb\rb=0$ and $\lb \lb
  \sigma_Q(q), \tau_1^\dagger\rb, \tau_2^\dagger\rb +\lb
  \tau_1^\dagger, \lb \sigma_Q(q), \tau_2^\dagger\rb\rb=0 $.  We have
\begin{equation*}\begin{split}
    \left\lb \sigma_Q(q_1), \left\lb \sigma_Q(q_2),
        \tau^\dagger\right\rb\right\rb-\left\lb \sigma_Q(q_2), \left\lb
        \sigma_Q(q_1), \tau^\dagger\right\rb\right\rb
 &=\left\lb \sigma_Q(q_1),
      (\Delta_{q_2}\tau)^\dagger\right\rb
    -\left\lb \sigma_Q(q_2), (\Delta_{q_1}\tau)^\dagger\right\rb\\
    &=(\Delta_{q_1}\Delta_{q_2}\tau)^\dagger-
    (\Delta_{q_2}\Delta_{q_1}\tau)^\dagger,
\end{split}\end{equation*}
and 
\begin{align*}
  \left\lb \left\lb \sigma_Q(q_1), \sigma_Q(q_2)\right\rb, \tau^\dagger\right\rb
 & =\left\lb \sigma_Q(\lb q_1, q_2\rb_\Delta)-\widetilde{R(q_1,q_2)},
    \tau^\dagger\right\rb= (\Delta_{\lb q_1,
    q_2\rb_\Delta}\tau)^\dagger+(R(q_1,q_2)(\partial_B\tau))^\dagger
\end{align*}
by Lemma \ref{formulas}.
We now choose $q_1,q_2, q_3\in\Gamma(Q)$ and compute 
\begin{equation*}
\begin{split}
  &\left\lb\left\lb \sigma_Q(q_1), \sigma_Q(q_2)\right\rb, \sigma_Q(q_3)\right\rb\\
&=\left\lb
    \sigma_Q(\lb q_1, q_2\rb_\Delta)-\widetilde{R(q_1,q_2)}, \sigma_Q(q_3)
  \right\rb\\
  &=\sigma_Q(\lb \lb q_1, q_2\rb_\Delta, q_3\rb_\Delta)-\widetilde{R(\lb q_1, q_2\rb_\Delta,
  q_3)}-\mathcal D\ell_{\langle R(q_1,q_2)\cdot , q_3\rangle}+\widetilde{\lozenge_{q_3}R(q_1,q_2)}\\
  &=\sigma_Q(\lb \lb q_1, q_2\rb_\Delta, q_3\rb_\Delta)-\widetilde{R(\lb q_1, q_2\rb_\Delta,
  q_3)}\\
  &\qquad -\sigma_Q(\partial_B^*\langle R(q_1,q_2)\cdot ,
  q_3\rangle)-\widetilde{\nabla_\cdot^*\langle R(q_1,q_2)\cdot , q_3\rangle}
    +\widetilde{\lozenge_{q_3}R(q_1,q_2)}
\end{split}
\end{equation*}
and 
\begin{equation*}
\begin{split}
  \left\lb \sigma_Q(q_2),\left\lb \sigma_Q(q_1),
      \sigma_Q(q_3)\right\rb\right\rb&=\left\lb \sigma_Q(q_2),
    \sigma_Q(\lb q_1,
    q_3\rb_\Delta)-\widetilde{R(q_1,q_3)}\right\rb\\
  &=\sigma_Q(\lb q_2, \lb q_1,
  q_3\rb_\Delta\rb_\Delta)-\widetilde{R(q_2,\lb q_1,
    q_3\rb_\Delta)}-\widetilde{\lozenge_{q_2}R(q_1,q_3)}.
\end{split}
\end{equation*}
We hence find that 
\[\left\lb\left\lb \sigma_Q(q_1), \sigma_Q(q_2)\right\rb,
  \sigma_Q(q_3)\right\rb+\left\lb \sigma_Q(q_2),\left\lb \sigma_Q(q_1),
    \sigma_Q(q_3)\right\rb\right\rb=\left\lb \sigma_Q(q_1),\left\lb \sigma_Q(q_2),
    \sigma_Q(q_3)\right\rb\right\rb
\]
if and only if 
\[\left\lb\left\lb q_1, q_2\right\rb_\Delta, 
  q_3\right\rb_\Delta+\left\lb q_2,\left\lb q_1,
    q_3\right\rb_\Delta\right\rb_\Delta=\left\lb q_1,\left\lb q_2,
    q_3\right\rb_\Delta\right\rb_\Delta+\partial_B^*\langle
R(q_1,q_2)\cdot,q_3\rangle
\]
and 
\begin{equation*}
\begin{split}
  &R(\lb
  q_1,q_2\rb_\Delta,q_3)+\nabla_\cdot^*\langle
  R(q_1,q_2)\cdot,q_3\rangle-\lozenge_{q_3}R(q_1,q_2)
  +R(q_2,\lb q_1,q_3\rb_\Delta)+\lozenge_{q_2}R(q_1,q_3)\\
  &=R(q_1,\lb q_2,q_3\rb_\Delta)+\lozenge_{q_1}R(q_2,q_3).
\end{split}
\end{equation*}
We conclude using \eqref{curv_dual_Jac}.
\end{proof}

A combination of Propositions \ref{anchor}, \ref{comp}, \ref{CA3},
\ref{CA2}, \ref{Jacobi} and Lemma \ref{useful_lemma} proves Theorem \ref{main}.

\section{Double vector bundles, dualisation and linear splittings}\label{appendix_dual}
We write $(D;A,B;M)$ for a double
vector bundle, i.e.~a commutative square
\begin{equation*}
\begin{xy}
\xymatrix{
D \ar[r]^{\pi_B}\ar[d]_{\pi_A}& B\ar[d]^{q_B}\\
A\ar[r]_{q_A} & M}
\end{xy}
\end{equation*}
with all four sides vector bundles projections, such that $\pi_B$
is a vector bundle morphism over $q_A$; such that $+_B: D\times_B D
\rightarrow D$ is a vector bundle morphism over $+: A\times_M A
\rightarrow A$, and such that the scalar multiplication $\R\times
D\to D$ in the bundle $D\to B$ is a vector bundle morphism over the
scalar multiplication $\R\times A\to A$.  The corresponding statements
for the operations in the bundle $D\to A$ follow. Recall that the
condition that each addition in $D$ is a morphism with respect to the
other is exactly
$ (d_1+_Ad_2)+_B(d_3+_Ad_4)=(d_1+_Bd_3)+_A(d_2+_Bd_4)$
for apropriate $d_1,d_2,d_3,d_4\in D$.  The vector bundles $A$ and $B$
are called the side bundles. The core $C$ of $D$ is the intersection of
the kernels of $\pi_A$ and of $\pi_B$. It has a natural vector bundle
structure over $M$, the projection of which we call $q_C\colon C
\rightarrow M$. We denote by $C_m \ni c \longmapsto \overline{c} \in
\pi_A^{-1}(0^A_m) \cap \pi_B^{-1}(0^B_m)$ the inclusion $C
\hookrightarrow D$.

For a smooth section $c\colon M \rightarrow C$, the corresponding core
section $c^\dagger\colon B \rightarrow D$ is defined by
$c^\dagger(b_m) = 0^D_{\vphantom{1}_{b_m}} +_A \overline{c(m)}$,
$m \in M$, $b_m \in B_m$.  We denote the corresponding core section
$A\to D$ by $c^\dagger$ also, relying on the argument to distinguish
between them. The space of core sections of $D$ over $B$ is written
$\Gamma_B^c(D)$.  A section $\xi\in \Gamma_B(D)$ is linear if
$\xi\colon B \rightarrow D$ is a bundle morphism from $B \rightarrow
M$ to $D \rightarrow A$ over a section $a\in\Gamma(A)$.  The space of
linear sections of $D$ over $B$ is written $\Gamma^\ell_B(D)$. A
section $\psi\in \Gamma(B^*\otimes C)$ defines a linear section
$\tilde{\psi}\colon B\to D$ over the zero section $0^A\colon M\to A$
by $ \widetilde{\psi}(b_m) = 0^D_{b_m}+_A \overline{\psi(b_m)}$
for all $b_m\in B$.  We call $\widetilde{\psi}$ a core-linear section.
The space of sections
of $D\to B$ is generated as a $C^{\infty}(B)$-module by its 
linear and core sections.

 Let $A, \, B, \, C$ be vector bundles over $M$.
The decomposed double vector bundle with sides $A$ and $B$ and
    core $C$ is $D=A\times_M B \times_M C$ with the vector bundle structures
  $D=q^{!}_A(B\oplus C) \to A$ and $D=q_B^{!}(A\oplus C) \to B$.
In particular, the fibered product $A\times_M B$ is a double vector
bundle over the sides $A$ and $B$ and with trivial core.
A linear splitting\footnote{Each double vector bundle admits a linear
  splitting, see \cite{Jotz17a} for comments on this, and for
  references.} of $(D; A, B; M)$ is an injective morphism of double
vector bundles $\Sigma\colon A\times_M B\hookrightarrow D$ over the
identity on the sides $A$ and $B$.
A linear splitting $\Sigma$ of a double vector bundle $D$ is 
equivalent to a splitting $\sigma_A$ of the short exact sequence of
$C^\infty(M)$-modules
\begin{equation}\label{ses_sections}
0 \longrightarrow \Gamma(B^*\otimes C) \hookrightarrow \Gamma^\ell_B(D) 
\longrightarrow \Gamma(A) \longrightarrow 0,
\end{equation}
where the third map is the map that sends a linear section $(\xi,a)$
to its base section $a\in\Gamma(A)$.  The splitting $\sigma_A$ will be
called a horizontal lift or simply a lift. Given
$\Sigma$, the horizontal lift $\sigma_A\colon \Gamma(A)\to
\Gamma_B^\ell(D)$ is given by $\sigma_A(a)(b_m)=\Sigma(a(m), b_m)$ for
all $a\in\Gamma(A)$ and $b_m\in B$.  By the symmetry of a linear
splitting, we find that a lift $\sigma_A\colon
\Gamma(A)\to\Gamma_B^\ell(D)$ is equivalent to a lift $\sigma_B\colon
\Gamma(B)\to \Gamma_A^\ell(D)$:
$\sigma_B(b)(a(m))=\sigma_A(a)(b(m))$ for all $a\in\Gamma(A)$,
$b\in\Gamma(B)$.
Note finally that two linear splittings $\Sigma^1,\Sigma^2\colon
A\times_MB\to D$ differ by a section $\phi$ of $A^*\otimes
B^*\otimes C\simeq \operatorname{Hom}(A,B^*\otimes C)\simeq
\operatorname{Hom}(B,A^*\otimes C)$:  For each
$a\in\Gamma(A)$ the difference $\sigma_A^2(a)-_B\sigma_A^1(a)$ of
horizontal lifts is the core-linear section defined by
$\phi(a)\in\Gamma(B^*\otimes C)$. By symmetry,
$\sigma_B^2(b)-_A\sigma_B^1(b)=\widetilde{\phi(b)}$ for each
$b\in\Gamma(B)$.

The space of linear sections of $D$ is a locally free and finitely
generated $C^{\infty}(M)$-module (this follows from the existence of
local splittings \cite{delCarpio-Marek15}). Hence, there is a vector bundle $\widehat{A}$ over
$M$ such that $\Gamma^l_B(D)$ is isomorphic to $\Gamma(\widehat{A})$
as $C^{\infty}(M)$-modules. The vector bundle $\widehat{A}$ is called
the fat vector bundle defined by $\Gamma^l_B(D)$. 
The short exact sequence \eqref{ses_sections} induces a short exact
sequence of vector bundles
\begin{equation}\label{ses_vb}
0 \longrightarrow B^*\otimes C \hookrightarrow \widehat{A} \longrightarrow A \longrightarrow 0.
\end{equation}
We refer to \cite{Pradines77,Mackenzie05,GrMe10a} for more
details on double vector bundles. 

\bigskip
Double vector bundles can be dualised in two distinct ways.  We denote
by $D\duer A$ the dual of $D$ as a vector bundle over $A$ and likewise
for $D\duer B$. The dual $D\duer A$ is again a double vector
bundle\footnote{The projection $\pi_{C^*}\colon D\duer A\to C^*$ is
  defined as follows: if $\Phi\in D\duer A$ projects to
  $\pi_A(\Phi)=a_m$, then $\pi_{C^*}(\Phi)\in C^*_m$ is defined by
  $\pi_{C^*}(\Phi)(c_m)=\Phi(0^D_{a_m}+_B\overline{c_m})$ for all
  $c_m\in C_m$. If $\Phi_1$ and $\Phi_2\in D\duer A$ satisfy
  $\pi_{C^*}(\Phi_1)=\pi_{C^*}(\Phi_2)$, $\pi_A(\Phi_1)=a^1_m$ and
  $\pi_A(\Phi_2)=a^2_m$, then $\Phi_1+_{C^*}\Phi_2$ is defined by
  $(\Phi_1+_{C^*}\Phi_2)(d_1+_Bd_2)=\Phi_1(d_1)+\Phi_2(d_2)$ for all
  $d_1,d_2\in D$ with $\pi_B(d_1)=\pi_B(d_2)$ and $\pi_A(d_1)=a^1_m$,
  $\pi_A(d_2)=a^2_m$.  The core element $\overline{\beta_m}\in D\duer
  A$ defined by $\beta_m\in B^*$ is
  $\overline{\beta_m}(d)=\beta_m(\pi_B(d))$ for all $d\in D$ with
  $\pi_A(d)=0^A_m$.  By playing with the vector bundle structures on
  $D\duer A$ and the addition formula, one can show that each core element
  of $D\duer A$ is of this form. See \cite{Mackenzie11}.}, with side
bundles $A$ and $C^*$ and core $B^*$ \cite{Mackenzie99,Mackenzie11}.
$$ 
{\xymatrix{
    D\ar[r]^{\pi_B}\ar[d]_{\pi_A}&   B\ar[d]^{q_{B}}\\
    A\ar[r]_{q_A}                   &  M\\
  }} \qquad\qquad {\xymatrix{
    D\duer A \ar[r]^{\pi_{C^*}}\ar[d]_{\pi_A}&   C^*\ar[d]^{q_{C^*}}\\
    A\ar[r]_{q_{A}}                   &  M\\
  }} \qquad\qquad {\xymatrix{
    D\duer B \ar[r]^{\pi_B}\ar[d]_{\pi_{C^*}}&   B\ar[d]^{q_B}\\
    C^*\ar[r]_{q_{C^*}}                   &  M\\
  }}
$$ 
By dualising again $D\duer A$ over $C^*$, we get
\[\xymatrix{
(D\duer A)\duer{C^*} \ar[r]^{\pi_{C^*}}\ar[d]_{\pi_B}&   C^*\ar[d]^{q_{C^*}}      \\
B\ar[r]_{q_{B}}             &  M,\\
}\]
with core $A^*$. In the same manner, we get a double vector bundle
$(D\duer B)\duer{C^*}$ with sides $A$ and $C^*$ and core $B^*$.

The vector bundles $D\duer B\to C^*$ and $D\duer A\to C^*$ are, up to
a sign, naturally in duality to each other \cite{Mackenzie05}. The
pairing
\[ \nsp{\cdot\,}{\cdot} \colon (D\duer A)\times_{C^*} (D\duer B)\to
\mathbb R
\]  
is defined as follows: for $\Phi\in D\duer A$ and $\Psi\in D\duer B$
projecting to the same element $\gamma_m$ in $C^*$, choose $d\in D$
with $\pi_A(d)=\pi_A(\Phi)$ and $\pi_B(d)=\pi_B(\Psi)$.  Then $\langle
\Phi, d\rangle_A-\langle \Psi,d\rangle_B$ does not depend on the
choice of $d$ and we set $\nsp{\Phi}{\Psi}=\langle \Phi,
d\rangle_A-\langle \Psi,d\rangle_B$.  This implies in particular that
$D\duer A$ is canonically (up to a sign) isomorphic to $(D\duer B)\duer
{C^*}$ (we identify $D\duer A$ with $(D\duer B)\duer{C^*}$ using
$-\nsp{\cdot\,}{\cdot}$) and $D\duer B$ is isomorphic to $(D\duer
A)\duer{C^*}$ (we identify $D\duer B$ with $(D\duer A)\duer{C^*}$ using
$\nsp{\cdot\,}{\cdot}$).

Given a horizontal lift $\sigma_A\colon\Gamma(A)\to\Gamma_B^l(D)$,
  the ``dual'' horizontal lift
  $\sigma^\star_A\colon\Gamma(A)\to\Gamma_{C^*}^l(D\duer A)$ is defined by
\[\langle \sigma_A^\star(a)(\gamma_m),\sigma_A(a)(b_m)\rangle_A=0, \qquad
\langle
\sigma_A^\star(a)(\gamma_m),c^\dagger(a(m))\rangle_A=\langle\gamma_m,
c(m)\rangle
\]
for all $a\in\Gamma(A)$, $c\in\Gamma(C)$, $b_m\in B$ and $\gamma_m\in
C^*$. In the same manner, given a horizontal lift
$\sigma_B\colon\Gamma(B)\to\Gamma_A^l(D)$, we define the dual
horizontal lift $\sigma^\star_B\colon\Gamma(B)\to\Gamma_{C^*}^l(D\duer
B)$.  We have the following equations:
\begin{equation}\label{equal_fat}
  \nsp{\sigma_A^\star(a)}{\sigma_B^\star(b)}=0, \quad
  \nsp{\sigma_A^\star(a)}{\alpha^\dagger} = -q_{C^*}^*\langle\alpha, a\rangle,\quad
  \nsp{\beta^\dagger}{\sigma_B^\star(b)} = q_{C^*}^*\langle\beta, b\rangle,
\end{equation}
for all $a\in\Gamma(A)$, $b\in\Gamma(B)$, $\alpha\in\Gamma(A^*)$ and
$\beta\in\Gamma(B^*)$.

\def\cprime{$'$} \def\polhk#1{\setbox0=\hbox{#1}{\ooalign{\hidewidth
  \lower1.5ex\hbox{`}\hidewidth\crcr\unhbox0}}} \def\cprime{$'$}
  \def\cprime{$'$}

\end{document}